\documentclass[a4paper, 11pt, fleqn, leqno]{article}
\RequirePackage{fullpage}

\usepackage{multicol}
\usepackage{amsthm,amsmath,amssymb}
\usepackage{booktabs,array}
\usepackage{subcaption}
\usepackage[dvipsnames]{xcolor}
\usepackage[pdftex,breaklinks,colorlinks,
    citecolor={BlueViolet}, linkcolor={Blue},urlcolor=Maroon]{hyperref}
\usepackage{tikz,pgfkeys}
\usetikzlibrary{decorations, decorations.markings, decorations.pathmorphing, decorations.shapes}
\usetikzlibrary{shapes}
\usetikzlibrary{arrows.meta}
\usetikzlibrary{quotes}
\usetikzlibrary{calc}
\usetikzlibrary{positioning}
\pgfdeclarelayer{bg}    \pgfsetlayers{bg,main}  

\usepackage{charter,eulervm}
\usepackage{ifthen}
\usepackage{enumerate}
\usepackage[final,expansion=alltext,protrusion=true]{microtype}

\theoremstyle{plain} 
\newtheorem{theorem}{Theorem}[section]
\newtheorem{lemma}[theorem]{Lemma}
\newtheorem{corollary}[theorem]{Corollary}
\newtheorem{proposition}[theorem]{Proposition}
\newtheorem*{definition}{Definition}
\newtheorem*{remark}{Remark}

\tikzset{
  filled vertex/.style  = {circle,draw=blue,fill=black!50,inner sep=1pt},
  empty vertex/.style  = {circle, draw, fill = white, inner sep=1.5pt, minimum width=1.5pt},
  corner/.style  = {fill=gray!30, draw= gray, thick, inner sep=2pt},
  b-vertex/.style = {fill=RubineRed, diamond, draw=blue, inner sep=1.5pt},
  a-vertex/.style = {draw=blue, fill=cyan, inner sep=2pt},
  uncertain edge/.style = {decoration={dashsoliddouble}, decorate},
  witnessed edge/.style = {ultra thick, teal}
}

\tikzset{shifted path/.style args={from #1 to #2}{insert path={
let \p1=($(#1.east)-(#1.center)$),
\p2=($(#2.east)-(#2.center)$),\p3=($(#1.center)-(#2.center)$),
\n1={.75/veclen(\x1,\y1)},\n2={.75/veclen(\x2,\y2)},\n3={atan2(\y3,\x3)} in
(#1.{\n3+180+asin(\n1)}) to (#2.{\n3-asin(\n2)})
}}}
\newcommand{\uncertain}[2]
{
  \draw[shifted path=from #1 to #2];
  \draw[densely dotted, thick, shifted path=from #2 to #1];
}

\newcommand{\lp}{\ensuremath{\operatorname{lp}}}
\newcommand{\rp}{\ensuremath{\operatorname{rp}}}

\title{Characterization of Circular-arc Graphs: III. Chordal Graphs}
\author{
  Yixin Cao\thanks{Department of Computing, Hong Kong Polytechnic University, Hong Kong, China.  \texttt{yixin.cao@polyu.edu.hk}.
    The author gratefully acknowledges the support of the K. C. Wong Education Foundation.
  } 
  \and Tomasz Krawczyk\thanks{Faculty of Mathematics and Information Science, Warsaw University of Technology, Poland. \\\texttt{tomasz.krawczyk@pw.edu.pl}.}
}

\begin{document}
\maketitle
\begin{abstract}
  We identify all minimal chordal graphs that are not circular-arc graphs, thereby resolving one of ``the main open problems'' concerning the structures of circular-arc graphs as posed by Dur{\'{a}}n, Grippo, and Safe in 2011.
  The problem had been attempted even earlier, and previous efforts have yielded partial results, particularly for claw-free chordal graphs and chordal graphs with an independence number of at most four.
  The answers turn out to have very simple structures: except for eight small graphs, with at most ten vertices, all the others have a unified description.  
  Our findings are based on a structural study of McConnell's flipping, which transforms circular-arc graphs into interval graphs with certain representation patterns.

\end{abstract}

\section{Introduction}\label{sec:intro}

All graphs discussed in this paper are finite and simple.  The vertex set and edge set of graph~$G$ are denoted by, respectively,~$V(G)$ and~$E(G)$.
For a subset~$U\subseteq V(G)$, we denote by~$G[U]$ the subgraph of~$G$ induced by~$U$, and by~$G - U$ the subgraph~$G[V(G)\setminus U]$, which is shortened to~$G - v$ when~$U = \{v\}$.
The \emph{neighborhood} of a vertex~$v$, denoted by~$N_{G}(v)$, comprises vertices adjacent to~$v$, i.e.,~$N_{G}(v) = \{ u \mid uv \in E(G) \}$, and the \emph{closed neighborhood} of~$v$ is~$N_{G}[v] = N_{G}(v) \cup \{ v \}$.
We may drop the subscript if the graph is clear from the context.
A \emph{clique} is a set of pairwise adjacent vertices, and an \emph{independent set} is a set of vertices that are pairwise nonadjacent. 
For~$\ell \ge 4$, we use $C_\ell$ to denote a simple cycle on~$\ell$ vertices; they are called \emph{holes}.
A graph is~\emph{$H$-free} if it does not contain~$H$ as an induced subgraph.

\begin{figure}[ht!]
  \centering\small
  \begin{subfigure}[b]{.23\linewidth}
    \centering
    \begin{tikzpicture}[scale=.6]\small
      \foreach \i in {1, 2, 3} 
      \draw ({120*\i-90}:1) -- ({120*\i+30}:1) -- ({120*\i-30}:2) -- ({120*\i-90}:1);
      \foreach[count =\i from 0] \n in {1, 3, 5} {
        \node[empty vertex] at ({90 - 120*\i}:2) {};
        \node at ({90 - 120*\i}:2.5) {$\n$};
        \pgfmathsetmacro{\x}{int(\n+1)}
        \node[filled vertex] at ({30 - 120*\i}:1) {};
        \node at ({30 - 120*\i}:1.4) {$\x$};
      }
    \end{tikzpicture}
    \caption{}
  \end{subfigure}  
  \;
  \begin{subfigure}[b]{.3\linewidth}
    \centering
    \begin{tikzpicture}[scale=.1]
      \begin{scope}[every path/.style={{|[left]}-{|[right]}}]        
        \draw[Sepia,thick] (120:13) arc (120.:-120:13); \draw (200:13) arc (200:160:13);
        \draw[Sepia,thick] (260:12) arc (260:-10:12); \draw (-40:12) arc (-40.:-80:12);
        \draw[Sepia,thick] (370:11) arc (370:100:11); \draw (80:11) arc (80:40:11);
      \end{scope}
      \draw[dashed,thin] (10,0) arc (0:360:10);
      \foreach[count=\j] \i/\c in {40/, -120/Sepia, -80/, 100/Sepia, 160/, -10/Sepia} {
        \draw[dashed] (\i:11) -- (\i:10);
        \node[\c] at (\i:8) {${\j}$};
      }
    \end{tikzpicture}
    \caption{}
  \end{subfigure}  
  \;
  \begin{subfigure}[b]{.3\linewidth}
    \centering
    \begin{tikzpicture}[scale=.1]
      \begin{scope}
        \foreach[count=\j from 0] \p/\q in {6/3, 2/5, 4/1} {
          \pgfmathsetmacro{\radius}{11+\j}
          \pgfmathsetmacro{\span}{180}
          \pgfmathsetmacro{\start}{240 - 120*\j}
          \draw[{|[left]}-{|[right]}, thick, Sepia]  (\start:\radius) arc (\start:{\start-\span}:\radius); 
          \draw[Sepia, dashed] ({\start}:\radius) -- ({\start}:10);
          \node[Sepia] at ({\start}:8) {${\p}$};

          \pgfmathsetmacro{\span}{20}
          \pgfmathsetmacro{\start}{340 - 120*\j}
          \draw[{|[left]}-{|[right]}]  (\start:\radius) arc (\start:{\start-\span}:\radius);
          \draw[dashed] ({\start}:\radius) -- ({\start}:10);
          \node at ({\start}:8) {${\q}$};
        }
      \end{scope}
      \draw[dashed,thin] (10,0) arc (0:360:10);
    \end{tikzpicture}
    \caption{}
  \end{subfigure}  
  \caption{A circular-arc graph and its two circular-arc models.  In (b), any two arcs for vertices~$\{2, 4, 6\}$ cover the circle; in (c), the three arcs for vertices~$\{2, 4, 6\}$ do not share any common point.}
  \label{fig:normal-and-helly}
\end{figure}

A graph is a \emph{circular-arc graph} if its vertices can be assigned to arcs on a circle such that two vertices are adjacent if and only if their corresponding arcs intersect.  Such a set of arcs is called a \emph{circular-arc model} for this graph (Figure~\ref{fig:normal-and-helly}).
If we replace the circle with the real line and arcs with intervals, we end up with interval graphs.
All interval graphs are circular-arc graphs.
Both graph classes are by definition \emph{hereditary}, i.e., closed under taking induced subgraphs.
While both classes have been intensively studied, there is a huge gap between our understanding of them.
One fundamental combinatorial problem on a hereditary graph class is its characterization by \emph{forbidden induced subgraphs}, i.e., minimal graphs that are not in the class.
For example, the {minimal forbidden induced subgraphs} of interval graphs are holes and those in Figure~\ref{fig:non-interval}~\cite{lekkerkerker-62-interval-graphs}.

\begin{figure}[ht]
  \tikzstyle{every node}=[empty vertex]
  \centering \small
  \begin{subfigure}[b]{0.22\linewidth}
    \centering
    \begin{tikzpicture}[xscale=.6, yscale=.6]
      \node (c) at (0, 0) {};
      \foreach[count=\i] \p in {below, right, below} {
        \node (u\i) at ({90*(3-\i)}:2) {};
        \node (v\i) at ({90*(3-\i)}:1) {};
        \draw (u\i) -- (v\i) -- (c);
      }
    \end{tikzpicture}
    \caption{long claw}\label{fig:long-claw-unlabeled}
  \end{subfigure}
  \,
  \begin{subfigure}[b]{0.22\linewidth}
    \centering
    \begin{tikzpicture}[xscale=.6, yscale=.7]
      \node (v7) at (0, -1) {};
      \draw (-2, 0) -- (2, 0);
      \foreach[count=\i from 2] \v/\p in {x_{1}/above, v_{1}/above, v_{2}/above right, v_{3}/above, x_{3}/above} {
        \node (v\i) at ({\i - 4}, 0) {};
        \draw (v\i) -- (v7);
      }
      \node (v1) at (0, .75) {};
      \draw (v4) -- (v1);      
    \end{tikzpicture}
    \caption{whipping top}\label{fig:whipping-top}
  \end{subfigure}
  \,
  \begin{subfigure}[b]{0.22\linewidth}
    \centering
    \begin{tikzpicture}[scale=.8]
      \draw (-1.5, 0) -- (0, 0) edge[dashed] (1.5, 0);
\draw (0, 1.75) node {} -- (0, 1) node (c) {} -- (0, 0) node {};
      \foreach[count =\i] \y in {1, 3} {
        \node (u\i) at ({3 * \i - 4.5}, 0) {};
        \node (v\i) at ({2 * \i - 3}, 0) {};
        \draw (u\i) -- (v\i) -- (c);
      }
    \end{tikzpicture}
    \caption{\dag{}}\label{fig:dag} 
  \end{subfigure}
  \,
  \begin{subfigure}[b]{0.22\linewidth}
    \centering
    \begin{tikzpicture}[scale=.8]
      \draw[dashed] (-1., 0) -- (1., 0);
      \node (v3) at (0, 0) {};
      \foreach[count =\i, evaluate={\x=int(2*\i - 3);}] \a/\b in {1/1, p/3} {
        \node (u\i) at ({1.5 * \x}, 0) {};
        \node (v\i) at ({1. * \x}, 0) {};
        \node (c\i) at ({.35 * \x}, 1) {};
        \draw (u\i) -- (v\i) -- (c\i) -- (u\i);
      }
      \foreach \i in {1, 2} {
        \foreach \j in {1, 2, 3}
        \draw (c\i) -- (v\j);
      }
      \draw (0, 1.75) node (x) {} -- (c1) -- (c2) -- (x);
    \end{tikzpicture}
    \caption{\ddag{}}\label{fig:ddag}  
  \end{subfigure}
  \caption{Minimal chordal graphs that are not interval graphs.  A~$\dag$ graph or a~$\ddag$ graph contains at least six vertices.}
  \label{fig:non-interval}
\end{figure}

\begin{theorem}[\cite{lekkerkerker-62-interval-graphs}]
  \label{thm:lb}
  A graph is an interval graph if and only if it does not contain any hole or any graph in Figure~\ref{fig:non-interval} as an induced subgraph.
\end{theorem}

The same problem on circular-arc graphs, however, has been open for sixty years~\cite{hadwiger-64-combinatorial-geometry, klee-69-cag}.
It is already very complicated to characterize chordal circular-arc graphs by forbidden induced subgraphs.
A graph is \emph{chordal} if it is~$C_{\ell}$-free for all~$\ell \geq 4$. 
Obviously, all interval graphs are chordal circular-arc graphs.
We rely on the reader to check that the long claw, shown in Figure~\ref{fig:long-claw-unlabeled}, and~$\dag$ graphs on seven or more vertices, shown in Figure~\ref{fig:dag}, are not circular-arc graphs.
Thus, they are chordal forbidden induced subgraphs of circular-arc graphs.\footnote{Note that the set of chordal forbidden induced subgraphs of circular-arc graphs is different from the set of minimal forbidden induced subgraphs of chordal circular-arc graphs.  The first is a subset of the second, which also contains all holes.}  Other chordal forbidden induced subgraphs must contain a whipping top, shown in Figure~\ref{fig:whipping-top}, a net (the~$\dag$ graph on six vertices), or a~$\ddag$ graph, shown in Figure~\ref{fig:ddag}, which are all circular-arc graphs.  However, except for adding an isolated vertex, there is no obvious way to augment them to derive minimal forbidden induced subgraphs, not to mention enumerating minimal chordal forbidden induced subgraphs exhaustively.
For example, let us consider chordal forbidden induced subgraph of circular-arc graphs on ten or fewer vertices.
As listed in the appendix, there are 20 such graphs.  Nine of them can be obtained in the aforementioned way,
while only four of the remaining 11 have been identified in literature~\cite{bonomo-09-partial-characterization-cag, francis-14-blocking-quadruple}.  

Bonomo et al.~\cite{bonomo-09-partial-characterization-cag} characterized chordal circular-arc graphs that are claw-free.
There are only four claw-free chordal graphs that are not circular-arc graphs.
Through generalizing Lekkerkerker and Boland's~\cite{lekkerkerker-62-interval-graphs} structural characterization of interval graphs, Francis et al.~\cite{francis-14-blocking-quadruple} defined a forbidden structure of circular-arc graphs.
This observation enabled them to characterize chordal circular-arc graphs with independence number at most four.
As we will see, most chordal forbidden induced subgraphs of circular-arc graphs contain an induced claw, and their independence numbers can be arbitrarily large.
These attempts motivated Dur{\'a}n et al.~\cite{duran-14-survey} to list finding the forbidden induced subgraph characterization for circular-arc graphs within the class of chordal graphs one of ``the main open problems'' regarding the classes of circular-arc graphs.
We~\cite{cao-24-split-cag} have previously obtained the list  of minimal split graphs that are not circular-arc graphs.
A graph is a \emph{split graph} if its vertex set can be partitioned into a clique and an independent set, and hence all split graphs are chordal.
In passing let us mention that Bang-Jensen and Hell~\cite{bang-jensen-94-chordal-proper-cag} characterized proper circular-arc graphs (i.e., graphs admitting circular-arc models in which no arc properly contains another) that are chordal.

\begin{figure}[ht]
  \centering \small
  \begin{subfigure}[b]{0.12\linewidth}
    \centering
    \begin{tikzpicture}[scale=.6]
      \foreach \i in {0, 2} 
      \node[filled vertex] (v\i) at (0, \i) {};
      \draw (v0) -- (v2);
      \foreach \i in {-1, 1} 
      \node[empty vertex] at (\i, 1) {};
      \node at (0, -.4) {};
    \end{tikzpicture}
    \caption{}
  \end{subfigure}
  \,
  \begin{subfigure}[b]{0.18\linewidth}
    \centering
    \begin{tikzpicture}[scale=.6]
      \def\radius{1.}
        \foreach[count =\j] \i in {0, 1, 2} {
          \node[empty vertex] (u\i) at ({120*\i-30}:{\radius*2.}) {};
          \node[filled vertex] (v\i) at ({120*\i-30}:\radius) {};
          \draw (v\i) -- (u\i);
        }
      \foreach \i in {0,..., 2} {
        \pgfmathsetmacro{\j}{int(Mod(\i + 1, 3))}
        \draw (v\i) -- (v\j);
      }
    \end{tikzpicture}
    \caption{}
  \end{subfigure}
  \,
  \begin{subfigure}[b]{0.2\linewidth}
    \centering
    \begin{tikzpicture}[scale=.7]
      \def\n{4}
      \def\radius{1.}
      \foreach \i in {1, ..., \n}
      \draw ({(\i - .5) * (360 / \n)}:1) -- ({(\i + .5) * (360 / \n)}:1) -- ({(\i) * (360 / \n)}:1.5) -- ({(\i - .5) * (360 / \n)}:1);
      \foreach \i in {1, ..., \n} {
        \node[empty vertex] (u\i) at ({(\i) * (360 / \n)}:{\radius*1.5}) {};
        \node[filled vertex] (v\i) at ({(\i - .5) * (360 / \n)}:\radius) {};
      }
      \draw (v1) -- (v3) (v2) -- (v4);
    \end{tikzpicture}
    \caption{}
  \end{subfigure}
  \quad
  \begin{subfigure}[b]{0.2\linewidth}
    \centering
    \begin{tikzpicture}[scale=.7]
      \def\n{5}
      \def\radius{1.}      
      \foreach \i in {1,..., \n} {
        \pgfmathsetmacro{\angle}{90 - (\i) * (360 / \n)}
        \foreach \j in {-1, 0, 1}
        \draw (\angle:{\radius*1.6}) -- ({90 - (\i + \j) * (360 / \n)}:{\radius});
      }
      \foreach \i in {1,..., \n} {
        \pgfmathsetmacro{\angle}{90 - (\i) * (360 / \n)}
        \node[filled vertex] (v\i) at (\angle:\radius) {};
        \node[empty vertex] (u\i) at (\angle:{\radius*1.6}) {};
        \pgfmathsetmacro{\p}{\i - 1}        
        \foreach \j in {1, ..., \p}
        \ifthenelse{\i>1}{\draw (v\i) -- (v\j)}{};
      }
    \end{tikzpicture}
    \caption{}
  \end{subfigure}
  \caption{The complements of~$k$-suns, for~$k = 2, 3, 4, 5$.}
  \label{fig:complement-suns}
\end{figure}

\paragraph{The~$\otimes$ graphs.}
The main discovery of the present paper is the family of~$\otimes$ graphs.
The \emph{complement graph}~$\overline{G}$ of a graph~$G$ is defined on the same vertex set~$V(G)$, where a pair of distinct vertices~$u$ and~$v$ is adjacent in~$\overline{G}$ if and only if~$u v \not\in E(G)$.
For~$k \ge 2$, the \emph{$k$-sun}, denoted as~$S_{k}$, is the graph obtained from a cycle of length~${2 k}$ by adding all edges among the even-numbered vertices to make them a clique.
It is a split graph with a unique split partition, and so is its complement,~$\overline{S_{k}}$, in which each vertex has precisely two non-neighbors in the other part.
See Figure~\ref{fig:complement-suns} for the smallest examples of~$\overline{S_{k}}$.
Note that~$\overline{S_{4}}$ and~$S_{4}$ are isomorphic, while~$\overline{S_{3}}$ and~$S_{3}$ are the net and the sun, respectively.

\begin{figure}[ht]
  \centering \small
  \begin{tikzpicture}[scale=1]
    \draw[fill = gray!30, draw = gray] (2.5, 0.1) ellipse [x radius=3., y radius=0.7];
    \foreach \x/\l in {0/1, 5/k} {
      \node[filled vertex, "$v_{\l}$" below] (v\x) at (\x, 0.3) {};
    }
    \foreach[count=\i from 0] \x/\l in {0/k, 1/k - 1, 2/k - 2, 4/2, 5/1, 6/0} 
    \node[empty vertex, "$w_{\l}$"] (w\i) at (\x-.5, 1.5) {};
    \foreach[count=\i] \x/\l in {.8/2, 1.7/3, 3.3/k-2, 4.2/k-1} 
    \node[filled vertex, "$v_{\l}$" below] (v\i) at (\x, 0) {};

    \node at (2.5, 0) {$\ldots$};
    \node at (2.5, 1.5) {$\ldots$};

    \foreach \v/\list in {w0/{v0, v1, v2, v4, v3}, w1/{v0, v1, v2, v3}, w2/{v0, v1, v2, v5}, w3/{v0, v3, v4, v5}, w4/{v2, v3, v4, v5}, w5/{v1, v2, v3, v4, v5}}
    \foreach \x in \list \draw (\x) -- (\v);
  \end{tikzpicture}
  \caption{The gadget~$D_{k}$.  Lines among solid nodes are omitted for clarity.
  }
  \label{fig:gadget}
\end{figure}

For~$k \ge 1$, we define the gadget~$D_{k}$ as a subgraph of~$\overline{S_{k+1}}$ obtained by removing one vertex of degree~$2k - 1$ (a solid node in~Figure~\ref{fig:complement-suns}).
An example is illustrated in Figure~\ref{fig:gadget}, where the removed vertex was adjacent to~$w_{1}, \ldots, w_{k - 1}$.
The gadget~$D_{k}$ consists of~$2 k + 1$ vertices.  
The vertex set~$\{v_{1}, v_{2}, \ldots, v_{k}\}$ is a clique, hence called the \emph{clique vertices} of this gadget.
The remaining vertices, $\{w_{0}, w_{1}, \ldots, w_{k}\}$, form an independent set, and for~$i = 0, 1, \ldots, k$,
\[
  N(w_{i}) = \{v_{1}, v_{2}, \ldots, v_{k}\} \setminus \{v_{i}, v_{i+1}\},
\]
where~ $v_{0}$ and~$v_{k+1}$ are dummy vertices.
The vertices~$w_{0}$ and~$w_{k}$, i.e., the non-neighbors of the removed vertex from~$\overline{S_{k+1}}$, are called the \emph{ends} of this gadget; their degrees in the gadget are~$k - 1$ (while the degrees of other vertices are~$k - 2$ and~$2 k - 2$; note that~$k - 1 = 2 k - 2$ only when~$k = 1$).

\begin{figure}[ht]
  \centering \small
  \begin{tikzpicture}[scale=1.]
    \def\n{3}
    \def\radius{2}
    \begin{scope}[rotate=45]
      \filldraw[draw = gray, fill=gray!20] (\radius*1, 0) ellipse [x radius=.5, y radius=0.85];
      \foreach[count=\i] \y in {-1, 0, 1} {
        \node[filled vertex] (x\i) at (\radius*.9, \y*.4) {};
        \foreach \j in {-1, 0, 1}
        \draw (x\i) -- ++(-.4, \j/10);
      }
      \draw (x1) -- (x2) -- (x3) (x1) edge[bend right] (x3);
      \foreach[count=\i] \y in {-3, -1, 1, 3}
      \node[empty vertex] (y\i) at (\radius*1.1, \y*.15) {};
      \draw (y2) -- (x1) -- (y1) -- (x2) -- (y4) -- (x3) -- (y3);

      \node at (\radius*1.5, 0) {$D_{a_{0}}$};
    \end{scope}
    \foreach \i in {1, 2, 3, 4} {
      \foreach \j in {1, 2}
      \node[empty vertex] at ({90*\i+30-\j*20}:\radius)  (u\the\numexpr\i*2-2+\j\relax) {};
      \node at ({90-90*\i}:{\radius*1.2}) {$P_{a_{\the\numexpr\i*2-1\relax}}$};      
    }
    \node[fill=white] at ({90}:{\radius*1.2}) {$P_{a_{2p-1}}$};          
    \foreach \i in {1, 3, 5, 7}
    \draw[dashed] (u\i) -- (u\the\numexpr\i+1\relax);
    \draw (u4) -- ++(.05, .3) (u1) -- ++(-.3, -.05);
    
    \draw (u2) -- (y4) (u7) -- (y1);
    \begin{scope}[rotate=-45]
      \filldraw[draw = gray, fill=gray!20] (\radius*1., 0) ellipse [x radius=.5, y radius=0.85];
      \foreach[count=\i] \y in {-1, 0, 1} {
        \node[filled vertex] (x\i) at (\radius*.9, \y*.4) {};
        \foreach \j in {-1, 0, 1}
        \draw (x\i) -- ++(-.4, \j/10);
      }
      \draw (x1) -- (x2) -- (x3) (x1) edge[bend right] (x3);
      \foreach[count=\i] \y in {-3, -1, 1, 3}
      \node[empty vertex] (y\i) at (\radius*1.1, \y*.15) {};
      \draw (y2) -- (x1) -- (y1) -- (x2) -- (y4) -- (x3) -- (y3);

      \node at (\radius*1.5, 0) {$D_{a_{2}}$};
    \end{scope}
    \draw (u5) -- (y1) (u8) -- (y4);
    \begin{scope}[rotate=225]
      \filldraw[draw = gray, fill=gray!20] (\radius*1., 0) ellipse [x radius=.5, y radius=0.85];
      \foreach[count=\i] \y in {-3, 3, 0}
      \node[empty vertex] (y\i) at (\radius*1.1, \y*.15) {};
      \foreach[count=\i] \y in {-1, 1} {
        \node[filled vertex] (x\i) at (\radius*.9, \y*.35) {};
        \foreach \j in {-1, 0, 1}
        \draw (x\i) -- ++(-.4, \j/10);
        \draw (x\i) -- (y\i);
      }
      \draw (x1) -- (x2);        
      \node at (\radius*1.4, 0) {$D_{a_{4}}$};
    \end{scope}
    \draw (u3) -- (y1) (u6) -- (y2);
    \begin{scope}[rotate=135]
      \node[rotate=45] at (0:\radius) {$\cdots\cdots$};
    \end{scope}
    
    \node[empty vertex, "$c$"] (c) at (0, 0) {};
  \end{tikzpicture}
  \caption{The graph~$\otimes(a_{0}, a_{1}, \ldots, a_{2 p - 1})$.
    Inside each shadowed ellipse is a gadget.
    Each clique vertex (solid node) of a gadget is adjacent to all vertices not in the gadget.
  }
  \label{fig:the-graph}
\end{figure}

We are now ready to describe the main family of chordal forbidden induced subgraphs of circular-arc graphs.  Let~$p$ be a positive integer and~$\langle a_{0}, a_{1}, \ldots, a_{2 p - 1} \rangle$ a sequence of~$2 p$ positive integers.  For~$i = 0, \ldots, p-1$, we introduce a gadget~$D_{a_{2 i}}$ and a path~$P_{a_{2i + 1}}$.
For each gadget and each path, we arbitrarily assign the ends as \emph{the left end} and \emph{the right end}.  (Note that the two ends are identical for a trivial path.)
We put the gadgets and paths in order circularly, and connect the right end of one gadget/path with the left end of next path/gadget.
We also introduce a special vertex~$c$, and this concludes the vertex set.
For each gadget, we add edges between its clique vertices and all other vertices not in this gadget.
The resulting graph is denoted as~$\otimes(a_{0}, a_{1}, \ldots, a_{2 p - 1})$.
See Figure~\ref{fig:the-graph} for an illustration and below are two simple examples. 
\begin{itemize}
\item Graph~$\otimes(1, a)$ with~$a \ge 2$ is the~$\dag$ graph of order~$a + 4$.
  Here~$c$ is the degree-one vertex at the top, the gadget~$D_{1}$ comprises the other two degree-one vertices and the vertex of degree~$a+1$, while the path comprises the remaining vertices (at the bottom).
\item Graph~$\otimes(2, a)$ with~$a \ge 1$ is the~$\ddag$ graph of order~$a + 5$ augmented with an isolated vertex.  Here~$c$ is the degree-two vertex at the top, the gadget~$D_{2}$ comprises the other two degree-two vertices and the two vertices of degree~$a+3$, while the path comprises the remaining vertices (at the bottom).
\end{itemize}
Illustrations of~$\otimes(a, b)$ and~$\otimes(1, a, 1, b)$ for small~$a$ and~$b$ can be found in the appendix.
Note that the order of the graph~$\otimes(a_{0}, a_{1}, \ldots, a_{2 p - 1})$ is
\[
  1 + \sum^{p-1}_{i=0} (2 a_{2 i} + 1 + a_{2 i + 1}) = p + 1 + \sum^{2 p-1}_{i=0} a_{i} + \sum^{p-1}_{i=0} a_{2 i}.
\]
It is worth noting that graphs defined by two different sequences might be the same, for example, $\otimes(1, 2, 3, 4)$,~$\otimes(3, 4, 1, 2)$,~$\otimes(3, 2, 1, 4)$, and~$\otimes(1, 4, 3, 2)$ refer to the same graph.
Graphs~$\otimes(1, 1)$ and~$\otimes(1, 2)$ are circular-arc graphs.  
They are the only circular-arc graphs that arise from this construction.
All~$\otimes(a_{0}, a_{1}, \ldots, a_{2 p - 1})$ graphs that are not~$\otimes(1, 1)$ or~$\otimes(1, 2)$ are hence called~\emph{$\otimes$ graphs}.  Every~$\otimes$ graph is a minimal chordal graph that is not a circular-arc graph.  We have a stronger statement here.
A graph is a \emph{Helly circular-arc graph} if it admits a model in which the arcs for every maximal clique have a shared point.

\begin{proposition}\label{lem:o-graphs}
  Let~$p$ be a positive integer,~$(a_{0}, a_{1}, \ldots, a_{2 p - 1})$ a sequence of positive integers different from~$(1, 1)$ and~$(1, 2)$, and let~$G$ be the graph~$\otimes(a_{0}, a_{1}, \ldots, a_{2 p - 1})$.  Then
  \begin{itemize}
  \item $G$ is a chordal graph;
  \item $G$ is not a circular-arc graph; and
  \item for any vertex~$x\in V(G)$, the graph~$G - x$ is a Helly circular-arc graph.
  \end{itemize}
\end{proposition}

\begin{figure}[ht]
  \centering \small
  \begin{subfigure}[b]{0.35\linewidth}
    \centering
    \begin{tikzpicture}
      \draw[fill = gray!30, draw = gray] (2.5, 0) ellipse [x radius=2.2, y radius=0.5];
      \node[filled vertex, "$s$" below] (s) at (4, 0) {};
      \foreach \i in {1, 2, 3}
      \node[filled vertex, "$v_\i$" below] (v\i) at (\i, 0) {};
      \foreach \i in {1, 2, 3, 4}
      \node[empty vertex, "$u_\i$"] (u\i) at (\i, 1) {};
      \foreach[count=\i] \list in {{1, 3}, {2, 3, 4}, {3, 4}}
      \foreach \x in \list \draw (u\x) -- (v\i);

      \draw (u1) -- (u2) -- (u3) -- (u4);
      \node at (4.5, 0) {$K$};
    \end{tikzpicture}
    \caption{} \end{subfigure}
  \,
  \begin{subfigure}[b]{0.4\linewidth}
    \centering
    \begin{tikzpicture}[xscale=.3]
        \def\leftend{0}
        \def\rightend{15}
        \pgfmathsetmacro{\middle}{(\leftend + \rightend)/2}
        \def\bottomend{-.6}
        
      \begin{scope}[every path/.style={{|[right]}-{|[left]}}]
        \foreach[count=\i] \l/\r/\y in {3/6/4, 5/9/5, 8/14/4, 11/12/5}{
          \draw[thick] (\l-.02, \y/5) node[left, xshift=3pt, yshift=2pt] {\footnotesize $u_\i$} to (\r+.02, \y/5);
        }
      \draw[Sepia, thick] (\middle-.5, \bottomend) node[left] {\footnotesize $s$} to (\middle+.5, \bottomend);
      \end{scope}
        \foreach[count=\i] \l/\r/\y in {4/13/2, 2/7/3, 1/10/1}{
          \draw[Sepia, dotted] (\l-.02, \y/5) to (\r+.02, \y/5) node[left] {\footnotesize $v_\i$};
          \draw[{|[right]}-{|[left]}, Sepia, thick, rounded corners] (\r+.02, \y/5) -- ({\rightend + \y/5}, \y/5) --({\rightend + \y/5}, \bottomend-\y/5)--(\leftend - \y/5, \bottomend-\y/5)--(\leftend - \y/5, \y/5)--(\l-.02, \y/5);          
        }
        \draw[] (\middle, \bottomend+.2) -- (\middle, \bottomend-.7);
      \end{tikzpicture}
    \caption{}\label{fig:normalized_model}
  \end{subfigure}

  \begin{subfigure}[b]{0.35\linewidth}
    \centering
    \begin{tikzpicture}
      \node[a-vertex, "$s$" below] (s) at (4, 0) {};
      \foreach[count=\i] \n in {3, 2, 1}
      \node[a-vertex, "$v_\n$" below] (v\n) at (\i, 0) {};
      \foreach[count=\i] \n in {1, 2, 3, 4}
      \node[b-vertex, "$u_\n$"] (u\n) at (\i, 1) {};
\foreach[count=\i] \list in {{2, 4}, {1}, {1, 2}}
      \foreach \x in \list \draw (u\x) -- (v\i);
      \foreach[count=\i] \list in {{1, 3}, {2, 3}, {3}}
      \foreach \x in \list \draw[witnessed edge] (u\x) -- (v\i);
      
      \draw (u1) -- (u2) -- (u3) -- (u4);
      \draw (v1) -- (v2) -- (v3);
      \draw[bend left=15] (v1) edge (v3);
      \foreach \i in {0, ..., 4}
      \draw (s) -- ++(\i*.1 - .4, \i*.1);
    \end{tikzpicture}
    \caption{}
  \end{subfigure}
  \begin{subfigure}[b]{0.4\linewidth}
    \centering
    \begin{tikzpicture}[xscale=.25]
        \def\leftend{0}
        \def\rightend{15}
        \pgfmathsetmacro{\middle}{(\leftend + \rightend)/2}
        \def\bottomend{-.6}
        
      \begin{scope}[every path/.style={{|[right]}-{|[left]}}]
        \foreach[count=\i] \l/\r/\y in {3/6/4, 5/9/5, 8/14/4, 11/12/5}{
          \draw[thick] (\l-.02, \y/5) node[left, xshift=3pt, yshift=2pt] {\footnotesize $u_\i$} to (\r+.02, \y/5);
        }
      \draw[Sepia, thick] (\leftend, 0) node[left, xshift=3pt, yshift=2pt] {\footnotesize $s$} to (\rightend, 0);
        \foreach[count=\i] \l/\r/\y in {4/13/2, 2/7/3, 1/10/1}{
          \draw[Sepia, thick] (\l-.02, \y/5) node[left, xshift=3pt, yshift=2pt] {\footnotesize $v_\i$} to (\r+.02, \y/5);
        }
      \end{scope}
      \end{tikzpicture}
    \caption{}\label{fig:interval_model}
  \end{subfigure}
  \caption{Illustration for McConnell's transformation.
    (a) A circular-arc graph~$G$, where edges among the vertices in the clique~$K$ (in the shadowed area) are omitted for clarity; (b) a normalized circular-arc model of~$G$; (c) the interval graph~$G^K$, where edges incident to the universal vertex~$s$ are omitted for clarity; and (d) the interval model of~$G^K$ derived from (b) by flipping the arcs containing the center of the bottom.  
  }
  \label{fig:McConnell-transformation}
\end{figure}

\paragraph{Our results.}
The algorithm of McConnell~\cite{mcconnell-03-recognition-cag} recognizes circular-arc graphs by transforming them into interval graphs.
A vertex $v \in V(G)$ is \emph{universal} in $G$ if $N[v] = V(G)$.
Let~$G$ be a circular-arc graph with no universal vertices and~$\mathcal{A}$ a fixed circular-arc model of~$G$.
By \emph{flipping} an arc~$[\lp, \rp]$ we replace it with arc~$[\rp, \lp]$ --- note that flipping is different from complementing, which ends with open arcs, and we prefer to stick to closed arcs and closed intervals for simplicity.
If we flip all arcs containing some fixed point of the circle, we end with an interval model~$\mathcal{I}$; see Figure~\ref{fig:McConnell-transformation} for an illustration.
A crucial observation of McConnell is that the resulting interval graph is decided by the set of vertices whose arcs are flipped and not by the original circular-arc models.  For a suitable clique~$K$, all circular-arc models of~$G$ with certain properties lead to the same interval graph~\cite{mcconnell-03-recognition-cag}; see also \cite{cao-24-split-cag, cao-24-cag-ii-flipping}.
Thus, it makes sense to denote it as~$G^K$.
McConnell presented an algorithm to find a suitable set~$K$ and constructed the graph directly from~$G$, without a circular-arc model.
As we have seen, the construction is very simple when~$G$ is chordal~\cite{cao-24-split-cag, cao-24-cag-ii-flipping}.
In particular, the closed neighborhood of every simplicial vertex can be used as the clique~$K$~\cite{hsu-95-independent-set-cag}.

However,~$G^K$ being an interval graph does not imply that~$G$ is a circular-arc graph.
There are restrictions on the intersection patterns of the intervals.
A previous paper of this series~\cite{cao-24-cag-ii-flipping} has fully characterized this correlation when~$G$ is~$C_{4}$-free: it is a circular-arc graph if and only if~$G^K$ does not contain certain forbidden configurations (the formal definition is deferred to the next section).
It is well known that if a minimal forbidden induced subgraph is not connected, then it has precisely two components, a non-interval subgraph and an isolated vertex.
For connected chordal graphs, we can greatly shorten this list.

\begin{figure}[ht]
  \centering \small
  \begin{subfigure}[b]{0.2\linewidth}
    \centering
    \begin{tikzpicture}[scale=.5]
      \foreach \i in {1, 2, 3} {
        \draw ({120*\i+90}:2) -- ({120*\i-30}:2);
        \draw ({120*\i-90}:1) -- ({120*\i+30}:1);
      }
      \foreach \i in {1, 2, 3} {
        \pgfmathparse{int(Mod(\i,3))}
        \node[empty vertex] (u\i) at ({90-120*\i}:2) {};
      }
      \foreach \i in {1 , 2 , 3} {
        \pgfmathparse{int(Mod(\i,3))}
        \node[empty vertex] (v\i) at ({270-120*\i}:1) {};
      }
    \end{tikzpicture}
    \caption{}
    \label{fig:sun-simplified}
  \end{subfigure}
  \begin{subfigure}[b]{0.2\linewidth}
    \centering
    \begin{tikzpicture}[label distance=-2pt, scale=.7]
      \node[empty vertex] (v7) at (0, -1) {};
      \draw (-2, 0) -- (2, 0);
      \foreach[count=\i from 2] \v/\p in {x_{1}/above, v_{1}/above, v_{0}/above right, v_{2}/above, x_{2}/above} {
        \node[empty vertex] (v\i) at ({\i - 4}, 0) {};
        \draw (v\i) -- (v7);
      }
      \node[b-vertex] (v1) at (0, .75) {};
      \foreach \i in {2, 6} \node[b-vertex] at (v\i) {};
      \draw (v4) -- (v1);      
    \end{tikzpicture}
    \caption{}
    \label{fig:whipping-top-simplified}
  \end{subfigure}
  \begin{subfigure}[b]{0.2\linewidth}
    \centering
    \begin{tikzpicture}[scale=.7]
\node[a-vertex] (a) at (0, 1.) {};
      \node[b-vertex] (u) at (0, 0) {};
      \draw[witnessed edge] (a) -- (u);
      \foreach[count=\j] \x in {-1, 1} {
        \draw (a) -- ({2*\x}, 0) node[b-vertex] (x\j) {};        
        \draw (a) -- (\x, 0) node[empty vertex] (v\j) {};
        \draw (x\j) -- (v\j);
      }
      \draw (v1) -- (u) -- (v2);
    \end{tikzpicture}
    \caption{}
    \label{fig:p5x1-simplified}
  \end{subfigure}
  \begin{subfigure}[b]{0.2\linewidth}
    \centering
    \begin{tikzpicture}[scale=.75]
      \draw (1, 1) -- (4, 1);

      \node[empty vertex] (v4) at (3, 0.2) {};
      \foreach[count=\i] \t/\v/\x in {b-/x_1/1, a-/x_2/2, a-/{\quad v}/3.5, b-/u/5} {
        \node[\t vertex] (u\i) at (\i, 1) {};
      }
      \draw (v4) -- (u1);
      \draw (u2) -- (v4) -- (u3);
      \draw (v4) -- (u4);

      \draw (u3) -- ++(0, .85) node[b-vertex] (x2) {};
      \draw[witnessed edge] (u3) -- (u4);
    \end{tikzpicture}
    \caption{}
    \label{fig:whipping-top-1-simplified}
  \end{subfigure}

  \begin{subfigure}[b]{0.2\linewidth}
    \centering
    \begin{tikzpicture}[scale=.5]
      \def\n{5}
      \def\radius{1.5}
      \foreach \i in {0,..., \the\numexpr\n-1\relax} {
        \pgfmathsetmacro{\angle}{90 - (2 - \i) * (360 / \n)}
        \node[empty vertex] (v\i) at (\angle:\radius) {};
}
      \foreach \i in {1,..., \the\numexpr\n-1\relax} {
        \draw (v\i) -- (v\inteval{\i-1});
      }
      \draw[dashed] (v0) -- (v\inteval{\n-1});      
    \end{tikzpicture}
    \caption{}
    \label{fig:holes}
  \end{subfigure}
  \begin{subfigure}[b]{0.2\linewidth}
    \centering
    \begin{tikzpicture}[scale=.75]
      \node[b-vertex] (x1) at (2, 1.75) {};
      \node[empty vertex] (x2) at (2.5, 1) {};
      \foreach \i in {3, 5} 
      \node[b-vertex] (x\i) at (\i-2, 0) {};
      \node[a-vertex] (v) at (1.5, 1) {};
      \node[empty vertex] (x4) at (2., 0) {};

      \foreach \i in {1, 4, 5} \draw (x2) -- (x\i);
      \foreach \i in {1, ..., 5} \draw (v) -- (x\i);
      \draw (x3) -- (x4) (x4) edge[dashed] (x5);
      \draw[witnessed edge] (v) -- (x5) {};
    \end{tikzpicture}
    \caption{}
    \label{fig:ddag+e-simplified}
  \end{subfigure}
  \begin{subfigure}[b]{0.2\linewidth}
    \centering
    \begin{tikzpicture}[scale=.75]
      \node[b-vertex] (x0) at (2.5, 1.75) {};
      \foreach \i in {1, 2} {
        \node[a-vertex] (v\i) at ({1+\i}, 1) {};
        \draw (x0) -- (v\i);
      }

      \foreach[count=\i] \t/\l in {b-/u_{1}, empty /x_{1}, empty /x_{p}, b-/u_{2}} {
        \node [\t vertex] (u\i) at (\i, 0) {};
      }
      
      \foreach \i in {1, 2}
      \draw[witnessed edge] (v\the\numexpr3-\i\relax) -- (u\the\numexpr\i*\i\relax);
      \foreach \i in {1, 2, 3} \draw (v1) -- (u\i);
      \foreach \i in {2, 3, 4} \draw (v2) -- (u\i);
      \draw (u1) -- (u2) (u2) edge[dashed] (u3)  (u3) -- (u4)  (v1) -- (v2);
    \end{tikzpicture}
    \caption{}
    \label{fig:ddag+2e-simplified}
  \end{subfigure}
  \begin{subfigure}[b]{0.2\linewidth}
    \centering
    \begin{tikzpicture}[scale=.75]
      \node[b-vertex] (x0) at (2.5, 1.75) {};
      \foreach \i in {1, 2} {
        \node[empty vertex] (v\i) at ({1+\i}, 1) {};
        \draw (x0) -- (v\i);
      }

      \foreach[count=\i] \t/\l in {b-/u_{1}, empty /x_{1}, empty /x_{p}, b-/u_{2}} {
        \node [\t vertex] (u\i) at (\i, 0) {};
      }
      
      \foreach \i in {1, 2, 3} \draw (v1) -- (u\i);
      \foreach \i in {2, 3, 4} \draw (v2) -- (u\i);
      \draw (u1) -- (u2) (u2) edge[dashed] (u3)  (u3) -- (u4)  (v1) -- (v2);
    \end{tikzpicture}
    \caption{}
    \label{fig:ddag-simplified}
  \end{subfigure}
  \caption{Forbidden configurations in~$G^K$.
The square nodes are ``in $N_{G}[s]$,'' rhombus ``not in~$N_{G}[s]$,'' and round ``uncertain.''
Between a square node and a rhombus node, a thick edge is in~$G$, and a thin edge is not.
  There are at least four vertices in~\ref{fig:holes}, at least six vertices in~\ref{fig:ddag+e-simplified}, and at least seven vertices in~\ref{fig:ddag+2e-simplified} and~\ref{fig:ddag-simplified}.
    }
  \label{fig:simplified-forbidden-configurations}
\end{figure}

\begin{theorem}\label{thm:simplified-forbidden-configurations}
  Let~$G$ be a minimal forbidden induced subgraph of circular-arc graphs.
  If~$G$ is connected and chordal, then for every simplicial vertex~$s$, the graph~$G^{N[s]}$ contains a forbidden configuration in Figure~\ref{fig:simplified-forbidden-configurations}.
\end{theorem}

Theorem~\ref{thm:simplified-forbidden-configurations} provides a way of finding all chordal forbidden induced subgraphs of circular-arc graphs: it suffices to identify all possible graphs from which the forbidden configurations are derived.
The task is then to ``reverse'' McConnell flipping to get all possible graphs~$G$ such that~$G^{K}$ contains a forbidden configuration.

\begin{figure}[ht]
  \centering \small
  \begin{subfigure}[b]{0.15\linewidth}
    \centering
    \begin{tikzpicture}[scale=.6]
      \def\radius{1.}
      \node[empty vertex] (c) at (0, 0) {};
      \foreach[count =\j] \i in {0, 1, 2} {
        \node[empty vertex] (u\i) at ({120*\i-30}:{\radius*2.}) {};
        \node[filled vertex] (v\i) at ({120*\i-30}:\radius) {};
        \draw (v\i) -- (u\i);
      }
      \foreach \i in {0,..., 2} {
        \pgfmathsetmacro{\j}{int(Mod(\i + 1, 3))}
        \draw (v\i) -- (v\j);
      }
    \end{tikzpicture}
    \caption{}
    \label{fig:net-star}
  \end{subfigure}
  \begin{subfigure}[b]{0.15\linewidth}
    \centering
    \begin{tikzpicture}[scale=.8]
      \node[empty vertex] (c) at (0, -0.3) {};

      \def\n{4}
      \def\radius{1.25}
      
      \foreach \i in {1, ..., \n}
      \draw ({(\i - .5) * (360 / \n)}:1) -- ({(\i + .5) * (360 / \n)}:1) -- ({(\i) * (360 / \n)}:\radius) -- ({(\i - .5) * (360 / \n)}:1);
      \foreach \i in {1, ..., \n} {
        \node[empty vertex] (u\i) at ({(\i) * (360 / \n)}:\radius) {};
        \node[filled vertex] (v\i) at ({(\i - .5) * (360 / \n)}:1) {};
      }
      \foreach \i in {1, 3} \draw (c) -- (v\i);
      \draw (v1) -- (v3) (v2) -- (v4);
    \end{tikzpicture}
    \caption{}
    \label{fig:the-weird}
  \end{subfigure}
  \begin{subfigure}[b]{0.15\linewidth}
    \centering
    \begin{tikzpicture}[scale=.6]
      \node[filled vertex] (c) at (0, 0) {};
      \foreach \i in {1, 2, 3} {
        \foreach[count=\j] \x in {0, 1} {
          \node[empty vertex](u\i\j) at ({120*\i - 30}:{1.2+\x*.8}) {};
          \draw ({120*\i-90}:1) -- (u\i\j) -- ({120*\i+30}:1);
          \ifthenelse{\x=0}{\draw (u\i\j) -- (c)}{};
        }
          \draw ({120*\i-90}:1) -- ({120*\i+30}:1);
      }

      \foreach \i in {1, 2, 3} {
        \node[filled vertex](v\i) at ({120*\i-90}:1) {};
        \draw (v\i) -- (c);
      }
    \end{tikzpicture}
    \caption{}
    \label{fig:long-claw-derived}
  \end{subfigure}
  \begin{subfigure}[b]{0.15\linewidth}
    \centering
    \begin{tikzpicture}[scale=.6]
      \node[filled vertex] (c) at (0, 0) {};
      \foreach \i/\l/\p in {1/1/, 2/6/below, 3/2/below} {
          \node[empty vertex] (u\i) at ({120*\i-30}:2) {};
          \draw ({120*\i-90}:1) -- (u\i) -- ({120*\i+30}:1);
          \draw ({120*\i-90}:1) -- ({120*\i+30}:1);
      }
      \foreach[count =\j from 3] \i/\p in {1/, 2/below, 3/} {
        \node[filled vertex] (v\i) at ({120*\i+30}:1) {};
        \draw (v\i) -- (c);
      }
      \draw (u1) -- (c);
      \foreach \i/\x/\l in {1/-1/457, 3/1/347} {
        \draw (v\i) -- ({90 - 60*\x}:2) node[empty vertex] {};
      }
    \end{tikzpicture}
    \caption{}
    \label{fig:whipping-top-derived}
  \end{subfigure}
  \qquad
  \begin{subfigure}[b]{0.14\linewidth}
    \centering
    \begin{tikzpicture}[xscale=.5, yscale=.45]
      \def\radius{1.}
        \foreach[count =\j] \i in {0, 1, 2} {
          \node[empty vertex] (u\i) at ({120*\i-30}:{\radius*2.}) {};
          \node[filled vertex] (v\i) at ({120*\i-30}:\radius) {};
          \draw (v\i) -- (u\i);
        }
      \foreach \i in {0,..., 2} {
        \pgfmathsetmacro{\j}{int(Mod(\i + 1, 3))}
        \draw (v\i) -- (v\j);
      }
      \node[empty vertex] (x) at (90:3) {};
      \draw[thick] (u1) -- (x);
    \end{tikzpicture}
    \caption{}
    \label{fig:extended-net}
  \end{subfigure}
  \begin{subfigure}[b]{0.14\linewidth}
    \centering
    \begin{tikzpicture}[scale=.8]
      \draw (-1, -1) grid (1,1);
      \draw[thick] (-1, 1) -- (1,1);
      \draw[bend right] (-1, 0) edge (1, 0);
      \draw (-1, 0) -- (0, 1) -- (1, 0) -- (0, -1)--cycle;      
      \foreach \x in {-1, 0, 1}
      \foreach \y in {-1, 0, 1}
      \node[filled vertex] at (\x, \y) {};
      \foreach \x in {-1, 1}
      \foreach \y in {-1, 1}
      \node[empty vertex] at (\x, \y) {};
      \node[empty vertex] at (0, 1) {};      
    \end{tikzpicture}
    \caption{}
    \label{fig:(S4+)-e}
  \end{subfigure}
\caption{Minimal chordal graphs that are not circular-arc graphs. The first four are split graphs. We end up with a graph in Figure~\ref{fig:net-star} and~\ref{fig:whipping-top-derived} after the thick lines removed from Figure~\ref{fig:extended-net} and~\ref{fig:(S4+)-e}, respectively.}
  \label{fig:chordal-non-cag}
\end{figure}

Let~$G$ be a minimal chordal forbidden induced subgraph of circular-arc graphs.
The analyses of the first four forbidden configurations reveal eight small forbidden induced subgraphs, listed in Figure~\ref{fig:chordal-non-cag}.
We show that if~$G^{N[s]}$ contains a hole~$C_\ell$, i.e., Configuration~\ref{fig:holes}, then~$G$ is either an~$\otimes$ graph or~$\overline{S_{\ell}^{+}}$, i.e., the graph obtained from~$\overline{S_{\ell}}$ by adding a vertex and making it adjacent to all the vertices of degree~$2 \ell -3$ (solid nodes in Figure~\ref{fig:complement-suns}).\footnote{In a sense, the graphs~$\overline{S_{k}^{+}}$ and~$C^\star_{k}$ can be viewed as~$\otimes(k, 0)$ and~$\otimes(0, k)$, respectively. Although it is not difficult to twist the definition of~$\otimes$ graphs to include them, we use the current one for the purpose of simplicity.}
Interestingly, none of the other infinite families of forbidden configurations, namely,~\ref{fig:ddag+e-simplified},~\ref{fig:ddag+2e-simplified}, and~\ref{fig:ddag-simplified}, leads to any new chordal forbidden induced subgraph. 
Indeed, if~$G^{s}$ contains configuration~\ref{fig:ddag+e-simplified},~\ref{fig:ddag+2e-simplified}, or~\ref{fig:ddag-simplified}, we can find another simplicial vertex~$s'$ of~$G$ such that~$G^{N[s']}$ contains one of Configurations~\ref{fig:sun-simplified}--\ref{fig:holes}.

The main result of the present paper is as follows.
For a graph~$G$, we use~$G^{\star}$ to denote the graph obtained from~$G$ by adding an isolated vertex.

\begin{theorem}\label{thm:main}
  A chordal graph is a circular-arc graph if and only if it does not contain an induced copy of long claw, whipping top$^\star$, a graph in Figure~\ref{fig:chordal-non-cag},~$\overline{S_{k}^{+}}, k \ge 3$, or an~$\otimes$ graph.
\end{theorem}

\begin{figure}[ht]
  \centering \small
  \tikzstyle{region number}  = [{circle, draw, fill=black, text=white, inner sep=2pt}]  
  
\def\xpos{5}
  \def\southwest{(\xpos, .5)}
  \def\corewidth{4}
  \def\coreheight{3.5}
  \def\extrawidth{1.5}
  \def\extraheight{.5}
\def\xoffset{2.5}
  \def\southeast{\southwest++(\xoffset, -.5)}

  \def\firstregion{\southeast++({-\corewidth}, 0) rectangle ++(9, 1.5)}  \def\secondregion{\southeast rectangle ++(6, 2.5)}  

  \begin{tikzpicture}[very thick,  text opacity=1]
    \foreach[count =\i from 0] \g in {split, chordal,~$C_4$-free} {
      \pgfmathparse{100/(3-\i)} 
      \draw[rounded corners, black!\pgfmathresult] \southwest rectangle ++(\corewidth+\extrawidth*\i, \coreheight+\extraheight*\i)
      node[anchor=north east, text=black]  {\g};
    }
    \foreach[count =\i from 0] \g in {{Helly\\ circular-arc}, circular-arc} {
      \pgfmathparse{100/(2-\i)} 
      \draw[rounded corners, black!\pgfmathresult] \southeast++({(\i-1)/2}, {1.5*(1-\i)}) rectangle ++({-\corewidth-\extrawidth*(\i)}, {\coreheight*(.93+.6*\i)})
node[anchor=north west, text=black, align=left]  {\g};
    }
    \begin{scope}[very thick, fill opacity=0.4, text opacity=1]
    \filldraw[red, rounded corners] \secondregion;
    \filldraw[blue, rounded corners] \firstregion;
    \end{scope}

    \foreach \i in {2, ..., 7} 
    \node[region number] at ({\xpos + (\xoffset+\corewidth+\extrawidth*int(\i/2-1)*2) /2 }, {Mod(\i, 2) + 1}) {\i};
    \foreach \i/\pos in {0/{({\xpos + \xoffset/2}, 1)}, 1/{({\xpos - (\corewidth-\xoffset)/2}, .5)}, 8/{({\xpos + \corewidth+ \extrawidth*.5 }, {.25})}, 9/{({\xpos + \corewidth+ \extrawidth*2.5 }, {1.5})}}
    \node[region number] at \pos {\i};
  \end{tikzpicture}
  \caption{The Venn diagram of the graph classes.  The blue and red areas consists of \textit{minimal} forbidden induced subgraphs of the class of Helly circular-arc graphs and the class of circular-arc graphs, respectively.  Note that every minimal~$C_{4}$-free circular-arc graph that is not a Helly circular-arc graph is a split graph~\cite{joeris-11-hcag}.
  }
  \label{fig:venn-diagram}
\end{figure}

We also characterize chordal graphs that are not Helly circular-arc graphs.
We use the Venn diagram in Figure~\ref{fig:venn-diagram} to illustrate the relationship of these classes.
Theorem~\ref{thm:main} summarizes all graphs in regions~2--5, and we can put them into the regions.
Note that all~$\otimes$ graphs are in Regions~2 and~4 by Proposition~\ref{lem:o-graphs}, and a chordal forbidden induced subgraph of circular-arc graphs is in region~3 or~5 if and only if it contains an induced copy of some graph in region~0.
Joeris et al.~\cite{joeris-11-hcag} showed that region~0 comprises $\overline{S_{k}}$ for all~$k\ge 3$, and we~\cite{cao-24-split-cag} have previously charted regions~2 and~3: Figure~\ref{fig:net-star}--\ref{fig:whipping-top-derived},~$\overline{S_{k}^{+}}, k \ge 3$, or~$\otimes(a, b)$ with~$1 \le b \le 2 \le a$.
Interestingly, region~5 comprises a single graph.

\begin{corollary}\label{cor:regions}
  Region~4 comprises long claw, whipping top$^\star$, the graph in Figure~\ref{fig:(S4+)-e},~$\otimes(a, b), b \ge 3$, and~$\otimes(a_{0}, a_{1}, \ldots, a_{2 p - 1})$ with~$p \ge 2$.
  Region~5 comprises only the graph in Figure~\ref{fig:extended-net}.
\end{corollary}

Regions~0, 2, and 4 together are the minimal chordal graphs that are not Helly circular-arc graphs.
\begin{corollary}\label{cor:chorda-hcag}
  A chordal graph is a Helly circular-arc graph if and only if it does not contain an induced copy of long claw, whipping top$^\star$, a graph in Figures~\ref{fig:long-claw-derived}, \ref{fig:whipping-top-derived}, \ref{fig:(S4+)-e},~$\overline{S_{k}}, k \ge 3$, or an~$\otimes$ graph.
\end{corollary}

Let us put our work into context.  By imposing restrictions on the intersection pattern between arcs, more than a dozen subclasses of circular-arc graphs have been defined and studied in the literature~\cite{lin-09-cag-and-subclasses}.
Several of them have been characterized by forbidden induced subgraphs.
They are mostly in the lower levels of the class hierarchy.
The class of chordal circular-arc graphs is the second subclass that contains all interval graphs, and the only previous one was normal Helly circular-arc graphs~\cite{cao-17-nhcag}.  The characterizations of both subclasses are achieved through connecting the forbidden induced subgraphs and minimal non-interval graphs.

\section{McConnell flipping}\label{sec:pre}

Let~$G$ be a chordal graph.
For each simplicial vertex~$s$ of~$G$, we can use the clique~$N[s]$ to define the auxiliary graph~$G^{N[s]}$~\cite{cao-24-split-cag}.
We use~$G^{s}$ as a shorthand for~$G^{N[s]}$.
The vertex set of~$G^{s}$ is ~$V(G)$, and the edge set is defined as follows.\footnote{In~\cite{cao-24-cag-ii-flipping} the auxiliary graph is defined for any clique~$K$ of~$G$, hence denoted by~$G^{K}$.  In the graph $G^{K}$, two vertices~$u, v\in K$ might be adjacent because~$u v$ is an edge of a~$C_{4}$ in~$G$.  The two definitions are equivalent on~chordal graphs.}
\begin{itemize}
\item The edges among vertices in~$V(G)\setminus N_{G}[s]$ are the same as in~$G$.
\item A pair of vertices~$u, v\in N_{G}[s]$ are adjacent in~$G^{s}$ if there exists a vertex adjacent to neither of them, i.e.,~$N_{G}(u)\cup N_{G}(v)\ne V(G)$.
\item A pair of vertices~$u\in N_{G}[s]$ and~$v\in V(G)\setminus N_{G}[s]$ are adjacent in~$G^{s}$ if~$N_{G}[v]\not\subseteq N_{G}[u]$.
\end{itemize}
Two quick remarks on the conditions are in order.
First, note that~$N(u)\cup N(v) = V(G)$ if and only if~$N[u]\cup N[v] = V(G)$ and~$uv\in E(G)$.
Second,~$N_{G}[v]\not\subseteq N_{G}[u]$ if and only if either they are not adjacent, or there exists a vertex adjacent to~$v$ but not~$u$ in~$G$.
Throughout this paper we will use the same notation as Figure~\ref{fig:simplified-forbidden-configurations} when illustrating subgraphs of~$G^{s}$.

Having two or more graphs on the same vertex set demands caution when we talk about adjacencies.  The only exception is the adjacency between a pair of vertices in~$V(G)\setminus N_{G}[s]$.  Since such a pair has the same adjacency in~$G$ and~$G^{s}$, we omit specifying the graph to avoid unnecessary clumsiness.

When~$G$ is connected, a forbidden configuration never involves all the vertices in~$V(G)\setminus\{s\}$.
For example, the subgraph of~$G$ induced by the four vertices in Figure~\ref{fig:claw-simplified} is also a claw, hence a circular-arc graph.  As a matter of fact, any graph on four vertices is a circular-arc graph.
The main challenge is to take into account the invisible vertices in the forbidden configurations.\footnote{This footnote is intended for the readers who are familiar with~\cite{cao-24-split-cag}, and can be safely skipped otherwise.
Although the task here is similar to the previous work~\cite{cao-24-split-cag}, the situation is far more complicated.
  Recall that the previous work was concerned with a split graph~$G$ with a unique split partition~$K\uplus I$.
First, the general strategy in~\cite{cao-24-split-cag} is to use the fact that~$K\setminus N[s]$ is a clique in both~$G$ and~$G^{N[s]}$.  Since no forbidden configuration contains a clique on five vertices, we can enumerate all possible constitutions of each forbidden configuration in~\cite{cao-24-split-cag}.  In the present paper, we cannot bound the size of~$N(s)$ or~$V(G)\setminus N[s]$ in a forbidden configuration.  In a hole of length~$\ell$ in~$G^{N[s]}$, for example, there are~$2^{\ell} - 1$ possible constitutions (the only exception is when the hole does not have any vertex from~$N_{G}(s)$).
Second, while only edges \textit{among}~$N[s]$ need witnesses in~\cite{cao-24-split-cag}, all edges \textit{incident to}~$N[s]$ do in the present paper.  To make it worse, witnesses in the present paper are not necessarily simplicial vertices of~$G$, and the connections between witnesses and the forbidden configuration and among witnesses themselves demand an extensive case analysis.
The witnesses of edges between~$N_G(s)$ and~$V(G)\setminus N_G[s]$ are especially troublesome.
The main challenge of the present paper is to decrease the number of cases by reductions.
}
\begin{definition}[Collateral edge]
  An edge of~$G^s$ is \emph{collateral} if it is an egde of~$G$ and at least one of its ends is in~$N_{G}[s]$.
\end{definition}
Note that no edge among~$V(G)\setminus N_{G}[s]$ is collateral.
If one end~$v$ of an edge is in~$N_{G}[s]$, then the edge is not collateral if and only if the other end is adjacent to neither~$s$ nor~$v$ in~$G$.
Let~$v_{1} v_{2}$ be a collateral edge.
By construction there must be a vertex adjacent to neither of them when~$v_{1}, v_{2}\in N_{G}(s)$, or a vertex in~$N_{G}(v_{2})\setminus N_{G}[v_{1}]$ when only~$v_{1}$ is in~$N_{G}(s)$.
In other words, there is always a vertex~$w\in V(G)\setminus N_G[s]$ such that~$w v_{i} \in E(G), i = 1, 2$, if and only if~$v_{i}\not\in N_G(s)$.
The vertex~$w$ can be viewed as a ``witness'' of this edge.
We can generalize the concept of witnesses to any clique~$K$ of~$G^{s}$.
\begin{definition}[Witness]
  Let~$K$ be a clique of~$G^{s}$.
  A vertex~$w\in V(G)\setminus N_{G}[s]$ is a \emph{witness} of~$K$ (in~$G^{s}$) if~$N_{G}[w]\cap K$ and~$N_{G}[s]\cap K$ partition~$K$.
  When~$K$ comprises the two ends of an edge of~$G^{s}$, we say that~$w$ is a witness of the edge (in~$G^{s}$).
  The clique or edge is then \emph{witnessed by~$w$} (in~$G^{s}$).
\end{definition}
Let~$w$ be a witness of a clique~$K$.
By definition, no edge between~$w$ and~$K$ is collateral.
We remark that a witness of a clique can be from this clique.
Indeed, if~$w\not\in K$, then~$w$ is a witness of~$K\cup \{w\}$, which is a clique of~$G^{s}$ by the definition above.
Another remark is that every collateral edge needs a witness, but a witnessed edge may or may not be collateral.
Table~\ref{tbl:inhomogeneous-triangles} lists some simple examples.

\begin{table}[ht!] 
  \caption{Connected induced subgraphs of~$G^{s}$ of order three and the corresponding structures in~$G$.
  The vertex~$w$ serves as a witness for the triple of vertices, and the vertex~$w_{i}, i = 0, 1, 2$, as a witness for the edge between~$\{v_{0}, v_{1}, v_{2}\}\setminus \{v_{i}\}$.
  }
  \label{tbl:inhomogeneous-triangles}
\tikzset{
  b-vertex/.style = {fill=RubineRed, diamond, draw=blue, inner sep=1.2pt},
  a-vertex/.style = {draw=blue, fill=cyan, inner sep=1.5pt},
  witnessed edge/.style = {ultra thick, teal}
}
  \begin{center}
    \small
    \begin{tabular}{ c | *{2}{c|} *{2}{c}| *{1}{c}| *{2}{c} } 
      \toprule
      $G^{s}$ 
      & \begin{tikzpicture}[scale=.62, baseline=({0, 0})]\small
        \node[b-vertex, "$v_{0}$"] at ({90}:1) (v) {};
        \foreach \i in {1, 2} {
          \node[a-vertex, "$v_{\i}$" below] at ({90 + 120*\i}:1) (u\i) {};
          \draw[witnessed edge] (v) -- (u\i);
        }
      \end{tikzpicture}
      &    \begin{tikzpicture}[scale=.62, baseline=({0, 0})]\small
        \foreach \i in {1, 2} 
        \node[a-vertex, "$v_{\i}$" below] at ({90 + 120*\i}:1) (u\i) {};
        \node[b-vertex, "$v_{0}$"] at ({90}:1) (v) {};
        \draw[witnessed edge] (v) -- (u1);
        \draw (v) -- (u2) -- (u1);
      \end{tikzpicture}
      &
        \multicolumn{2}{c|}{ \begin{tikzpicture}[scale=.62, baseline=({0, 0})]\small
            \node[b-vertex, "$v_{0}$"] at ({90}:1) (v) {};
            \foreach \i in {1, 2} {
              \node[a-vertex, "$v_{\i}$" below] at ({90 + 120*\i}:1) (u\i) {};
              \draw[witnessed edge] (v) -- (u\i);
            }
            \draw (u2) -- (u1);
          \end{tikzpicture}
        }
      &    \begin{tikzpicture}[scale=.62, baseline=({0, 0})]\small
        \node[a-vertex, "$v_{0}$"] at ({90}:1) (v) {};
        \foreach \i in {1, 2} {
          \node[a-vertex, "$v_{\i}$" below] at ({90 + 120*\i}:1) (u\i) {};
          \draw (v) -- (u\i);
        }
      \end{tikzpicture}
      &
        \multicolumn{2}{c}{
        \begin{tikzpicture}[scale=.62, baseline=({0, 0})]\small
          \node[a-vertex, "$v_{0}$"] at ({90}:1) (v) {};
          \foreach \i in {1, 2} {
            \node[a-vertex, "$v_{\i}$" below] at ({90 + 120*\i}:1) (u\i) {};
            \draw (v) -- (u\i);
          }
          \draw (u1) -- (u2);
        \end{tikzpicture}
        }
      \\
      \midrule
      $G$ 
      &    \begin{tikzpicture}[scale=.62, baseline=({0, 0})]\small
        \node[empty vertex, "$v_{0}$"] at ({90}:1) (v) {};
        \foreach \i in {1, 2} {
          \node[empty vertex, "$v_{\i}$" below] at ({90 + 120*\i}:1) (u\i) {};
          \draw (v) -- (u\i) -- ++(0, 1.) node[empty vertex, "$w_{\i}$"] {} -- (v);
        }
        \draw (u1) -- (u2);
      \end{tikzpicture}
      &    \begin{tikzpicture}[scale=.62, baseline=({0, 0})]\small
        \node[empty vertex, "$v_{0}$"] at ({90}:1) (v) {};
        \draw (v) -- ++(.75, 0) node[empty vertex, "$w$"] {};
        \foreach \i in {1, 2} {
          \node[empty vertex, "$v_{\i}$" below] at ({90 + 120*\i}:1) (u\i) {};
        }
        \draw (v) -- (u1);
        \draw (u2) -- (u1);
      \end{tikzpicture}
      &    \begin{tikzpicture}[scale=.62, baseline=({0, 0})]\small
        \node[empty vertex, "$v_{0}$"] at ({90}:1) (v) {};
        \draw (v) -- ++(.75, 0) node[empty vertex, "$w$"] {};
        \foreach \i in {1, 2} {
          \node[empty vertex, "$v_{\i}$" below] at ({90 + 120*\i}:1) (u\i) {};
          \draw (v) -- (u\i);
        }
        \draw (u2) -- (u1);
      \end{tikzpicture}
      &    \begin{tikzpicture}[scale=.62, baseline=({0, 0})]\small
        \node[empty vertex, "$v_{0}$"] at ({90}:1) (v) {};
        \foreach \i in {1, 2} {
          \node[empty vertex, "$v_{\i}$" below] at ({90 + 120*\i}:1) (u\i) {};
          \draw (v) -- (u\i) -- ++(0, 1.) node[empty vertex, "$w_{\i}$"] {} -- (v);
        }
        \node[empty vertex, "$w_{0}$" yshift=-18pt] at (90:.3) (w) {};
        \draw (u2) -- (u1);
      \end{tikzpicture}
      &    \begin{tikzpicture}[scale=.62, baseline=({0, 0})]\small
        \node[empty vertex, "$v_{0}$"] at ({90}:1) (v) {};
        \foreach \i in {1, 2} {
          \node[empty vertex, "$v_{\i}$" below] at ({90 + 120*\i}:1) (u\i) {};
          \draw (v) -- (u\i) -- ++(0, 1.) node[empty vertex, "$w_{\i}$"] {};
        }
        \draw (u1) -- (u2);
      \end{tikzpicture}
      &    \begin{tikzpicture}[scale=.62, baseline=({0, 0})]\small
        \node[empty vertex, "$v_{0}$"] at ({90}:1) (v) {};
        \foreach \i in {1, 2} {
          \node[empty vertex, "$v_{\i}$" below] at ({90 + 120*\i}:1) (u\i) {};
          \draw (v) -- (u\i);
        }
        \node[empty vertex, "$w$" yshift=-18pt] at (90:.3) (w) {};
        \draw (u1) -- (u2);
      \end{tikzpicture}
      &    \begin{tikzpicture}[scale=.62, baseline=({0, 0})]\small
        \node[empty vertex, "$v_{0}$"] at ({90}:1) (v) {};
        \foreach \i in {1, 2} {
          \node[empty vertex, "$v_{\i}$" below] at ({90 + 120*\i}:1) (u\i) {};
          \draw (v) -- (u\i) -- ++(0, 1.) node[empty vertex, "$w_{\i}$"] {};
        }
        \draw (v) -- ++(0, -.8) node[empty vertex, "$w_{0}$" yshift=-18pt] {};        
        \draw (u1) -- (u2);
      \end{tikzpicture}
      \\ 
      \bottomrule
    \end{tabular}

    \begin{tabular}{ c | c| *{2}{c}| c| *{2}{c}}
      \toprule
      $G^{s}$ 
      &     \begin{tikzpicture}[scale=.62, baseline=({0, 0})]\small
        \node[a-vertex, "$v_{0}$"] at ({90}:1) (v) {};
        \foreach \i in {1, 2} {
          \node[b-vertex, "$v_{\i}$" below] at ({90 + 120*\i}:1) (u\i) {};
          \draw[witnessed edge] (v) -- (u\i);
        }
      \end{tikzpicture}
      &    \multicolumn{2}{c|}{ \begin{tikzpicture}[scale=.62, baseline=({0, 0})]\small
          \node[a-vertex, "$v_{0}$"] at ({90}:1) (v) {};
          \foreach \i in {1, 2} {
            \node[b-vertex, "$v_{\i}$" below] at ({90 + 120*\i}:1) (u\i) {};
          }
          \draw[witnessed edge] (v) -- (u1);
          \draw (v) -- (u2);
        \end{tikzpicture}
        }
      &     \begin{tikzpicture}[scale=.62, baseline=({0, 0})]\small
        \node[a-vertex, "$v_{0}$"] at ({90}:1) (v) {};
        \foreach \i in {1, 2} {
          \node[b-vertex, "$v_{\i}$" below] at ({90 + 120*\i}:1) (u\i) {};
        }
        \draw[witnessed edge] (v) -- (u1);
        \draw (v) -- (u2) -- (u1);
      \end{tikzpicture}
      &    \multicolumn{2}{c}{\begin{tikzpicture}[scale=.62, baseline=({0, 0})]\small
          \node[a-vertex, "$v_{0}$"] at ({90}:1) (v) {};
          \foreach \i in {1, 2} {
            \node[b-vertex, "$v_{\i}$" below] at ({90 + 120*\i}:1) (u\i) {};
            \draw[witnessed edge] (v) -- (u\i);
          }
          \draw (u1) -- (u2);
        \end{tikzpicture}
        }
      \\
      \midrule
      $G$ 
      & \begin{tikzpicture}[scale=.62, baseline=({0, 0})]\small
        \node[empty vertex, "$v_{0}$"] at ({90}:1) (v) {};
        \foreach \i in {1, 2} {
          \node[empty vertex, "$v_{\i}$" below] at ({90 + 120*\i}:1) (u\i) {};
          \draw (v) -- (u\i) -- ++(0, 1.) node[empty vertex, "$w_{\the\numexpr3-\i\relax}$"] {};
        }
      \end{tikzpicture}
      & \begin{tikzpicture}[scale=.62, baseline=({0, 0})]\small
        \node[empty vertex, "$v_{0}$"] at ({90}:1) (v) {};
        \foreach \i in {1, 2} {
          \node[empty vertex, "$v_{\i}$" below] at ({90 + 120*\i}:1) (u\i) {};
        }
        \draw (v) -- (u1) -- ++(0, 1.) node[empty vertex, "$w_{2}$"] {};
      \end{tikzpicture}
      & \begin{tikzpicture}[scale=.62, baseline=({0, 0})]\small
        \node[empty vertex, "$v_{0}$"] at ({90}:1) (v) {};
        \foreach \i in {1, 2} {
          \node[empty vertex, "$v_{\i}$" below] at ({90 + 120*\i}:1) (u\i) {};
        }
        \draw (v) -- (u1) -- (0, 0) node[empty vertex, "$w_{2}$" yshift=-2pt] {} -- (u2);
      \end{tikzpicture}
      & \begin{tikzpicture}[scale=.62, baseline=({0, 0})]\small
        \node[empty vertex, "$v_{0}$"] at ({90}:1) (v) {};
        \foreach \i in {1, 2} {
          \node[empty vertex, "$v_{\i}$" below] at ({90 + 120*\i}:1) (u\i) {};
        }
        \draw (v) -- (u1) -- (u2);
      \end{tikzpicture}
      &    \begin{tikzpicture}[scale=.62, baseline=({0, 0})]\small
        \node[empty vertex, "$v_{0}$"] at ({90}:1) (v) {};
        \foreach \i in {1, 2} {
          \node[empty vertex, "$v_{\i}$" below] at ({90 + 120*\i}:1) (u\i) {};
          \draw (v) -- (u\i) -- ++(0, 1.) node[empty vertex, "$w_{\the\numexpr3-\i\relax}$"] {};
        }
        \draw (u1) -- (u2);
      \end{tikzpicture}
      &    \begin{tikzpicture}[scale=.62, baseline=({0, 0})]\small
        \node[empty vertex, "$v_{0}$"] at ({90}:1) (v) {};
        \node[empty vertex, "$w$" yshift=-2pt] at ({-90}:.1) (w) {};
        \foreach \i in {1, 2} {
          \node[empty vertex, "$v_{\i}$" below] at ({90 + 120*\i}:1) (u\i) {};
          \draw (v) -- (u\i) -- (w);
        }
        \draw (u1) -- (u2);
      \end{tikzpicture}
      \\ 
      \bottomrule
    \end{tabular}
  \end{center}
\end{table}

As said, when dealing with a forbidden configuration~$F$, we need to take witnesses of the collateral edges within~$F$ into consideration.
A witness of one particular collateral edge might be adjacent to additional vertices in~$F$, and witnesses of distinct collateral edges could either be adjacent or not.
These possibilities pose significant challenges in our analysis.
To mitigate these complexities, we will introduce a series of observations.
The first of them is quite intuitive; see the fifth group of Table~\ref{tbl:inhomogeneous-triangles}.

\begin{lemma}\label{lem:witness-N[s]}
  Let~$K$ be a clique of~$G^{s}$.
  If~$K\subseteq N_{G}(s)$ and~$K$ is not witnessed, then~$G$ contains an induced~$\overline{S_{3}^{+}}$.
\end{lemma}
\begin{proof}
  Let~$K'$ be a minimal subset of~$K$ that is unwitnessed.
  Note that~$|K'| > 2$: by assumption,~$G$ does not contain a universal vertex; a clique of two vertices are the endpoints of an edge and hence has a witness by construction.
  We take three vertices~$v_{0}$,~$v_{1}$, and~$v_{2}$ from~$K'$.
  By the selection of~$K'$, for each~$i = 0, 1, 2$, the clique~$K'\setminus \{v_{i}\}$ has a witness~$w_{i}$.
Since~$w_{i}$ is not a witness of~$K'$, it is in~$N_{G}(v_{i})$.  Thus,~$N_{G}(w_{i})\cap \{v_{0}, v_{1}, v_{2}\} = \{v_{i}\}$.
If~$w_{1}$ and~$w_{2}$ are adjacent, then~$v_{1}w_{1} w_{2}v_{2}$ is a hole of~$G$, which is impossible.
  For the same reason,~$w_{0}$ is not adjacent to~$w_{1}$ or~$w_{2}$.
  Then~$G[\{v_{0}, v_{1}, v_{2}, w_{0}, w_{1}, w_{2}\}]$ is an induced net (see the second graph of the fifth group of Table~\ref{tbl:inhomogeneous-triangles}), which forms an~$\overline{S_{3}^{+}}$ with~$s$.
\end{proof}

 \begin{figure}[ht]
  \centering \small
  \begin{subfigure}[b]{0.25\linewidth}
    \centering
    \begin{tikzpicture}[scale=.5]
      \foreach[count =\j] \i in {1, 2, 3} 
        \draw ({120*\i-90}:1) -- ({120*\i+30}:1) -- ({120*\i-30}:2) -- ({120*\i-90}:1);
      \foreach[count =\j] \i/\t in {1/a-, 2/a-, 3/b-} {
        \node[empty vertex] (u\i) at ({120*\i-30}:2) {};
        \node[empty vertex] (v\i) at ({120*\i-90}:1) {};
      }
      \foreach[count =\j] \i in {1, 2}
      \node at ({270 - 120*\j}:1.6) {$v_{\i}$};
      
      \node at ({90}:2.6) {$s$};
      \draw (u3) -- ++(1.5, 0) node[empty vertex] (x) {};
      \foreach[count =\i from 3] \v in {u2, v3, u3, x}
      \node[below = 1.5pt of \v] {$v_{\i}$};      
    \end{tikzpicture}
    \caption{}
  \end{subfigure}
  \,
  \begin{subfigure}[b]{0.25\linewidth}
    \centering
    \begin{tikzpicture}[xscale=.75]
      \draw (1, 0) -- (4, 0);
      \foreach \i in {1, 2} 
      \node[a-vertex, "$v_\i$"] (v\i) at (3.5-\i, 1) {};
      \foreach[count=\j from 3] \i in {1, ..., 4} {
        \node[b-vertex, "$v_\j$" below] (x\i) at (\i, 0) {};
        \draw (v2) -- (x\i);
      }
      \draw (v2) -- (v1);
      \foreach \i in {2, 3, 4} 
      \draw (v1) -- (x\i);
      \draw[witnessed edge] (x2) -- (v2) -- (x3) (v1) -- (x2);
      
      \node[empty vertex, "$s$"] (v4) at (4, 1) {};
      \foreach \i in {0, ..., 3}
      \draw (v4) -- +({180+\i*20}:.5);
    \end{tikzpicture}
    \caption{}
  \end{subfigure}
  \caption{(a) A graph~$G$ and (b)~$G^{s}$, where the edges incident to~$s$, which is universal, are omitted for clarity.
    The clique~$\{v_{1}, v_{2}, v_{4}\}$ has no witness, and
    each pair in it has a unique witness, namely, $v_{3}$, $v_{5}$, and~$v_{6}$.
    Note that~$v_{5}$ is not simplicial in~$G$.}
  \label{fig:no-witness}
\end{figure}

As shown in Figure~\ref{fig:no-witness}, Lemma~\ref{lem:witness-N[s]} does not hold for cliques of~$G^{s}$ intersecting~$V(G)\setminus N_G[s]$; see also groups~3 and~9 of Table~\ref{tbl:inhomogeneous-triangles}.
The next three propositions are on witnesses of edges and cliques involving vertices in~$V(G)\setminus N_{G}[s]$.  The first is about vertices in~$N_{G}[s]$ and their non-neighbors in~$G$.

\begin{proposition}[\cite{cao-24-cag-ii-flipping}]\label{lem:nonadjacent->dominate}
  Let~$v$ be a vertex in~$N_G[s]$.  For each vertex~$x\in V(G)\setminus N_{G}[v]$, it holds~$N_{G^{s}}[x]\subseteq N_{G^{s}}[v]$.
\end{proposition}
\begin{proof}
   Note that~$x\not\in N_G[s]$ because~$N_G[s]\subseteq N_G[v]$ by the definition of simplicial vertices.
  Suppose for contradiction that there exists a vertex~$y\in N_{G^{s}}[x]\setminus N_{G^{s}}[v]$.
  Note that~$y\in N_{G}(s)\cap N_{G}(x)$; otherwise,~$v y$ must be an edge of~$G^{s}$, witnessed by~$x$.
  By construction, there must be a witness~$w$ of the edge~$x y$.
  Note that~$w\in N_{G}(v)$, as otherwise it witnesses the edge~$v y$.
  But then~$x w v y$ is a hole of~$G$, a contradiction.
\end{proof}
\begin{corollary}\label{lem:P4}
  The middle edge of an induced path of length three in~$G^{s}$ is always present in~$G$.
\end{corollary}
\begin{proof}
It follows from the construction of~$G^{s}$ if neither or both ends of this edge are from~$N_{G}(s)$, or Proposition~\ref{lem:nonadjacent->dominate} otherwise (each end has a private neighbor in~$G^{s}$).
\end{proof}

We consider next the subgraphs of~$G$ corresponding to connected induced subgraphs of~$G^{s}$ of order three.
It is trivial when none of the three vertices is from~$N_{G}(s)$.
Lemma~\ref{lem:witness-N[s]} has characterized the case when they are all from~$N_{G}(s)$ and form a clique in~$G^{s}$.
The next statement is about a triangle of~$G^{s}$ with two vertices from~$N_{G}(s)$; see the second and third groups in Table~\ref{tbl:inhomogeneous-triangles}.

\begin{proposition}\label{lem:unwitnessed}
  Let~$\{v_{0}, v_{1}, v_{2}\}$ be a clique of~$G^{s}$ with~$v_{1}, v_{2}\in N_G(s)$ and~$v_{0}\in V(G)\setminus N_G[s]$.
  If the clique~$\{v_{0}, v_{1}, v_{2}\}$ is not witnessed, then it is a clique of~$G$.
\end{proposition}
\begin{proof}
  Since~$v_{0}$ is not a witness of~$\{v_{0}, v_{1}, v_{2}\}$, it is adjacent to at least one of~$v_{1}$ and~$v_{2}$ in~$G$.
  Assume without loss of generality that~$v_{0}\in N_{G}(v_{1})$.
  By construction, there is a witness~$w_{2}$ of~$v_{0} v_{1}$; hence~$w_{2}\in N_{G}(v_{0})\setminus N_{G}[v_{1}]$. 
  Since~$w_{2}$ does not witness~$\{v_{0}, v_{1}, v_{2}\}$, it is adjacent to~$v_{2}$ in~$G$.
  Then~$v_{0}$ and~$v_{2}$ must be adjacent in~$G$ because it is chordal; otherwise,~$w_{2} v_{0} v_{1} v_{2}$ is a hole.
\end{proof}

The last observation is on two collateral edges with different witnesses sharing a vertex.

\begin{proposition}\label{lem:P3}
  Let~$v_{1} v_{0} v_{2}$ be a path of~$G^{s}$ such that both edges~$v_{0} v_{1}$ and~$v_{0} v_{2}$ are collateral and~$|\{v_{1}, v_{2}\}\cap N_{G}(s)| \ne 1$.
  \begin{enumerate}[i)]
  \item If~$v_{0} v_{1}$ and~$v_{0} v_{2}$ have a common witness, then~$v_{1} v_{2}\in E(G)$;
  \item otherwise, there is no edge between a witness of~$v_{0} v_{1}$ and a witness of~$v_{0} v_{2}$.
  \end{enumerate}
\end{proposition}
\begin{proof}
  (i) Let~$w$ be a common witness of~$v_{0} v_{1}$ and~$v_{0} v_{2}$.
  If~$v_{1}, v_{2}\in N_{G}(s)$, then they are adjacent in~$G$ because~$N_{G}(s)$ is a clique. 
  Now suppose~$v_{1}, v_{2}\not\in N_{G}(s)$; then~$v_{0}\in N_{G}(s)$.
  Since~$v_{1} v_{0} v_{2} w$ is not a hole of~$G$, vertices~$v_{1}$ and~$v_{2}$ must be adjacent.
  
  (ii) For~$i = 1, 2$, let~$w_{i}$ be a witness of~$v_{0} v_{3-i}$.
  By assumption,~$w_{i}$ is not a witness of~$v_{0} v_{i}$, and hence~$w_{1}\ne w_{2}$. 
  Suppose for contradiction that~$w_{1}$ and~$w_{2}$ are adjacent.
  If~$v_{1}, v_{2}\in N_{G}(s)$, then~$v_{1} v_{2} w_{2} w_{1}$ is a hole of~$G$.
  In the rest,~$v_{1}, v_{2}\not\in N_{G}(s)$.
  Since~$w_{i}, i = 1, 2$, does not witness the edge~$v_{0} v_{i}$, it is neither identical nor adjacent to~$v_{i}$.
  Depending on whether~$v_{1}$ and~$v_{2}$ are adjacent, either~$v_{1} v_{2} w_{1} w_{2}$ or~$v_{1} v_{0} v_{2} w_{1} w_{2}$ is a hole of~$G$, contradicting that~$G$ is chordal.
\end{proof}

\section{Forbidden configurations}

\begin{figure}[ht]
  \centering \small
  \begin{subfigure}[b]{0.2\linewidth}
    \centering
    \begin{tikzpicture}[scale=.75]
      \node[a-vertex] (v) at (2, 1) {};
      \foreach \i in {1, 2, 3}
      \node[b-vertex] (u\i) at (\i, 0) {};
      \foreach \i in {1, 2, 3}
      \draw[very thick, teal] (v) -- (u\i);
    \end{tikzpicture}
    \caption{}
    \label{fig:claw-simplified}
  \end{subfigure}
  \begin{subfigure}[b]{0.2\linewidth}
    \centering
    \begin{tikzpicture}[scale=.75]
      \foreach \i in {1, 2}
      \node[a-vertex] (v\i) at (.5+\i, 1) {};
      \foreach \i in {1, 2, 3} {
        \node[b-vertex] (u\i) at (\i, 0) {};
}
      \draw (v1) -- (v2);
      \draw (v1) -- (u3) (v2) -- (u1);
      \draw[witnessed edge] (v1) -- (u1) (v2) -- (u2) (v2) -- (u3);
      \uncertain{v1}{u2};
    \end{tikzpicture}
    \caption{}
    \label{fig:double-claw-simplified}
  \end{subfigure}
  \begin{subfigure}[b]{0.2\linewidth}
    \centering
    \begin{tikzpicture}[scale=.75]
      \node[a-vertex] (v) at (3, 0) {};
      \foreach \i in {1, 2} {
        \node[rotate={180*\i}, b-vertex] (u\i) at (4, {1.5-\i}) {};
        \draw[witnessed edge] (v) -- (u\i);
      }
      \foreach[count=\i] \t in {b-, a-} {
        \node[\t vertex] (x\i) at ({\i},0) {};
      }
      \draw (v)--(x2)--(x1);    
    \end{tikzpicture}
    \caption{}
    \label{fig:fork-simplified}
  \end{subfigure}
  \begin{subfigure}[b]{0.2\linewidth}
    \centering
    \begin{tikzpicture}[scale=.75]
      \foreach[count=\i] \t in {b-, a-} {
        \node[\t vertex] (x\i) at ({\i},0) {};
      }
      \foreach \i in {1, 2} {
        \node[rotate={180*\i}, a-vertex] (v\i) at (3, {1.5-\i}) {};
        \node[rotate={180*\i}, b-vertex] (u\i) at (4, {1.5-\i}) {};
        \draw (x2) -- (v\i);
        \draw[witnessed edge] (v\i) -- (u\i);
      }
      \draw (x2)--(x1)--(v2)  (u1) -- (v2) -- (v1) -- (u2);
    \end{tikzpicture}
    \caption{}
    \label{fig:double-fork+1-simplified}
  \end{subfigure}

  \begin{subfigure}[b]{0.2\linewidth}
    \centering
    \begin{tikzpicture}[scale=.75]
      \draw (1, 0) -- (4, 0);
      \foreach[count=\x from 2] \t/\l/\p in {b-/u/above, b-/x_{3}/below}  \node[\t vertex] (\l) at (\x, 1) {};
      
      \foreach[count=\i] \t/\l in {b-/x_{1}, a-/v_{1}, a-/v_{2}, b-/x_{2}} {
        \node[\t vertex] (u\i) at (\i, 0) {};
      }
      \foreach \i in {2, 3} \draw (u) edge[witnessed edge] (u\i) (u\i) -- (x_{3});
    \end{tikzpicture}
    \caption{}
    \label{fig:ab-wheel-simplified}
  \end{subfigure}
    \begin{subfigure}[b]{0.2\linewidth}
    \centering
    \begin{tikzpicture}[scale=.7]
      \draw (-2.,0) node[b-vertex] (x1) {} -- (2.,0) node[b-vertex] (x2) {};
      \foreach[count=\i] \l in {v_1, v, v_2}
      \node[a-vertex] (v\i) at ({\i*1-2}, 0) {};
      \node[b-vertex] (u) at (90:1) {};
      \draw[witnessed edge] (u) -- (v2);
    \end{tikzpicture}
    \caption{}
    \label{fig:p5+1-simplified}
  \end{subfigure}
\begin{subfigure}[b]{0.2\linewidth}
    \centering
    \begin{tikzpicture}[scale=.7]
\node[a-vertex] (a) at (0, 1.) {};
      \node[b-vertex] (u) at (0, 0) {};
      \draw[witnessed edge] (a) -- (u);
      \foreach[count=\j] \x in {-1, 1} {
        \draw (a) -- ({2*\x}, 0) node[b-vertex] (x\j) {};        
        \draw (a) -- (\x, 0) node[empty vertex] (v\j) {};
        \draw (x\j) -- (v\j);
      }
      \draw (v1) -- (u) -- (v2);
    \end{tikzpicture}
    \caption{}
    \label{fig:p5x1-unlabeled}
  \end{subfigure}
  \begin{subfigure}[b]{0.2\linewidth}
    \centering
    \begin{tikzpicture}[scale=.7]
      \draw (-2.,0) -- (2.,0);
      \node[a-vertex] (a) at (0, 1.) {};
      \node[b-vertex] (u) at (0.75, 1) {};
      \foreach[count=\j] \x in {-2, -1, 0, 1, 2} {
        \ifodd\x
          \def\t{empty vertex}
        \else
          \def\t{b-vertex}
        \fi
        \draw (a) -- (\x, 0) node[\t] (x\j) {};
      }
      \draw (a) edge[witnessed edge] (u) (u) edge (x3);
    \end{tikzpicture}
    \caption{}
    \label{fig:bent-whipping-top-simplified}
  \end{subfigure}

  \begin{subfigure}[b]{0.2\linewidth}
    \centering
    \begin{tikzpicture}[scale=.75]
      \draw (1, 0) -- (4, 0);
      \node[a-vertex] (a) at (2.5, 1) {};
\path (a) ++(0, .75) node[b-vertex] (x) {};
      \uncertain{a}{x};
      
      \foreach[count=\i] \t/\l in {b-/x_{1}, b-/u_{1}, b-/u_{2}, b-/x_{2}} {
        \draw (a) -- (\i, 0) node[\t vertex] (u\i) {};
      }
      \foreach \i in {2, 3}
      \draw[witnessed edge] (a) -- (u\i);
    \end{tikzpicture}
    \caption{}
    \label{fig:(p4+p1)*1-simplified}
  \end{subfigure}
  \begin{subfigure}[b]{0.2\linewidth}
    \centering
    \begin{tikzpicture}[scale=.75]
      \draw (1, 1) -- (4, 1);

      \node[empty vertex] (v4) at (3, 0.2) {};
      \foreach[count=\i] \t/\v/\x in {b-/x_1/1, a-/x_2/2, a-/{\quad v}/3.5, b-/u/5} {
        \node[\t vertex] (u\i) at (\i, 1) {};
      }
      \draw (v4) -- (u1);
      \draw (u2) -- (v4) -- (u3);
      \draw (v4) -- (u4);

      \draw (u3) -- ++(0, .85) node[b-vertex] (x2) {};
      \draw[witnessed edge] (u3) -- (u4);
    \end{tikzpicture}
    \caption{}
    \label{fig:whipping-top-1-unlabeled}
  \end{subfigure}
  \begin{subfigure}[b]{0.2\linewidth}
    \centering
    \begin{tikzpicture}[xscale=.7, yscale=.6]
      \node[empty vertex] (c) at (0, 0) {};
      \foreach[count=\i] \j/\p in {1/below, 3/right, 2/below} {
        \node[b-vertex] (u\i) at ({90*(3-\i)}:2) {};
        \node[empty vertex] (v\i) at ({90*(3-\i)}:1) {};
        \draw (u\i) -- (v\i) -- (c);
      }
    \end{tikzpicture}
    \caption{}
    \label{fig:long-claw-simplified}
  \end{subfigure}
  \begin{subfigure}[b]{0.2\linewidth}
    \centering
    \begin{tikzpicture}[label distance=-2pt, scale=.7]
      \node[empty vertex] (v7) at (0, -1) {};
      \draw (-2, 0) -- (2, 0);
      \foreach[count=\i from 2] \v/\p in {x_{1}/above, v_{1}/above, v_{0}/above right, v_{2}/above, x_{2}/above} {
        \node[empty vertex] (v\i) at ({\i - 4}, 0) {};
        \draw (v\i) -- (v7);
      }
      \node[b-vertex] (v1) at (0, .75) {};
      \foreach \i in {2, 6} \node[b-vertex] at (v\i) {};
      \draw (v4) -- (v1);      
    \end{tikzpicture}
    \caption{}
    \label{fig:whipping-top-unlabeled}
  \end{subfigure}

  \begin{subfigure}[b]{0.2\linewidth}
    \centering
    \begin{tikzpicture}[scale=.5]
      \def\n{5}
      \def\radius{1.5}
      \foreach \i in {0,..., \the\numexpr\n-1\relax} {
        \pgfmathsetmacro{\angle}{90 - (2 - \i) * (360 / \n)}
        \node[empty vertex] (v\i) at (\angle:\radius) {};
}
      \foreach \i in {1,..., \the\numexpr\n-1\relax} {
        \draw (v\i) -- (v\inteval{\i-1});
      }
      \draw[dashed] (v0) -- (v\inteval{\n-1});      
    \end{tikzpicture}
    \caption{}
    \label{fig:holes-unlabeled}
  \end{subfigure}
  \begin{subfigure}[b]{0.2\linewidth}
    \centering
    \begin{tikzpicture}[scale=.75]
      \draw (1, 0) -- (2, 0);
      \draw[dashed] (2, 0) -- (3, 0);
      \node[a-vertex] (a) at (2, 1) {};
      \draw (a) -- ++(0, .75) node[b-vertex] {};

      \foreach[count=\i] \t/\l in {b-/u_{1}, empty /x_{1}, b-/u_{2}} {
        \draw (a) -- (\i, 0) node[\t vertex] (u\i) {};
      }
      \foreach \i in {1, 3}
      \draw[witnessed edge] (a) -- (u\i);
    \end{tikzpicture}
    \caption{}
    \label{fig:dag+2e-simplified}
  \end{subfigure}
  \begin{subfigure}[b]{0.2\linewidth}
    \centering
    \begin{tikzpicture}[scale=.75]
      \node[a-vertex] (x) at (3., 1) {};
      \draw (x) -- ++(0, .75) node[b-vertex] {};

      \draw (1, 0) -- (3, 0);
      \foreach[count=\i] \t/\v in {b-/x_{1}, a-/x_{2}, empty /x_{3}, b-/u} {
        \node[\t vertex] (u\i) at (\i, 0) {};
      }
      \draw[dashed] (u3) -- (u4);
      \foreach \i in {2, 3} \draw (x) -- (u\i);
      \draw[witnessed edge] (x) -- (u4) {};
    \end{tikzpicture}
    \caption{}
    \label{fig:dag+e-simplified}
  \end{subfigure}
  \begin{subfigure}[b]{0.2\linewidth}
    \centering
    \begin{tikzpicture}[scale=.7]
      \draw (1, 0) -- (2, 0) (4, 0) -- (5, 0);
      \draw[dashed] (4, 0) -- (2, 0);
      \node[empty vertex] (a) at (3, 1) {};
      \draw (a) -- ++(0, .75) node[b-vertex] {};

      \foreach[count=\i] \t/\l in {b-/u_{1}, empty /x_{1}, empty /x_{2}, empty /x_{p}, b-/u_{2}} {
        \node[\t vertex] (u\i) at (\i, 0) {};
      }
      \foreach \i in {2, 3, 4}
      \draw (a) -- (u\i);
    \end{tikzpicture}
    \caption{}
    \label{fig:dag-unlabeled}
  \end{subfigure}

  \begin{subfigure}[b]{0.2\linewidth}
    \centering
    \begin{tikzpicture}[scale=.5]
      \foreach \i in {1, 2, 3} {
        \draw ({120*\i+90}:2) -- ({120*\i-30}:2);
        \draw ({120*\i-90}:1) -- ({120*\i+30}:1);
      }
      \foreach \i in {1, 2, 3} {
        \pgfmathparse{int(Mod(\i,3))}
        \node[empty vertex] (u\i) at ({90-120*\i}:2) {};
      }
      \foreach \i in {1 , 2 , 3} {
        \pgfmathparse{int(Mod(\i,3))}
        \node[empty vertex] (v\i) at ({270-120*\i}:1) {};
      }
    \end{tikzpicture}
    \caption{}
    \label{fig:sun-unlabeled}
  \end{subfigure}
  \begin{subfigure}[b]{0.2\linewidth}
    \centering
    \begin{tikzpicture}[scale=.75]
      \node[b-vertex] (x1) at (2, 1.75) {};
      \node[empty vertex] (x2) at (2.5, 1) {};
      \foreach \i in {3, 5} 
      \node[b-vertex] (x\i) at (\i-2, 0) {};
      \node[a-vertex] (v) at (1.5, 1) {};
      \node[empty vertex] (x4) at (2., 0) {};

      \foreach \i in {1, 4, 5} \draw (x2) -- (x\i);
      \foreach \i in {1, ..., 5} \draw (v) -- (x\i);
      \draw (x3) -- (x4) (x4) edge[dashed] (x5);
      \draw[witnessed edge] (v) -- (x5) {};
    \end{tikzpicture}
    \caption{}
    \label{fig:ddag+e-unlabeled}
  \end{subfigure}
  \begin{subfigure}[b]{0.2\linewidth}
    \centering
    \begin{tikzpicture}[scale=.75]
      \node[b-vertex] (x0) at (0, 2.75) {};
      \node[empty vertex] (u) at (0, 1) {};
      \foreach \i in {1, 2} {
        \node[rotate={180*\i}, a-vertex] (v\i) at ({1.5-\i}, 2) {};
        \node[rotate={180*\i}, b-vertex] (x\i) at ({(1.5-\i)*2}, 1) {};
      }
      \draw (u) -- (x2) (v1) -- (v2) (x1) edge[dashed] (u);
      \foreach \i in {1, 2} {
        \foreach \j in {1, 2} 
        \draw (u)--(v\i)--(x\j);
        \draw (v\the\numexpr3-\i\relax) -- (x0);
        \draw[witnessed edge] (v\i) -- (x\the\numexpr3-\i\relax);
      }
      \draw (v1)--(x1);
    \end{tikzpicture}
    \caption{}
    \label{fig:ddag+2e-unlabeled}
  \end{subfigure}
  \begin{subfigure}[b]{0.2\linewidth}
    \centering
    \begin{tikzpicture}[scale=.75]
      \node[b-vertex] (x0) at (2.5, 1.75) {};
      \foreach \i in {1, 2} {
        \node[empty vertex] (v\i) at ({1+\i}, 1) {};
        \draw (x0) -- (v\i);
      }

      \foreach[count=\i] \t/\l in {b-/u_{1}, empty /x_{1}, empty /x_{p}, b-/u_{2}} {
        \node [\t vertex] (u\i) at (\i, 0) {};
      }
      
      \foreach \i in {1, 2, 3} \draw (v1) -- (u\i);
      \foreach \i in {2, 3, 4} \draw (v2) -- (u\i);
      \draw (u1) -- (u2) (u2) edge[dashed] (u3)  (u3) -- (u4)  (v1) -- (v2);
    \end{tikzpicture}
    \caption{}
    \label{fig:ddag-unlabeled}
  \end{subfigure}
  \caption{Forbidden configurations from \cite{cao-24-cag-ii-flipping}.
    The square nodes are ``in $N_{G}[s]$,'' rhombus ``not in $N_{G}[s]$,'' and round ``uncertain.''
    Between a square node and a rhombus node, a thick edge is ``in $G$,'' a thin edge is ``not in $G$,'' and it is ``uncertain'' otherwise (a solid line and a dotted line).
    There are at least four vertices in~\ref{fig:holes-unlabeled}, five in~\ref{fig:dag+2e-simplified}, six in \ref{fig:dag+e-simplified}, \ref{fig:dag-unlabeled}, \ref{fig:ddag+e-unlabeled}, and~\ref{fig:ddag+2e-unlabeled}, and seven in~\ref{fig:ddag-unlabeled}.
}
  \label{fig:forbidden-configurations}
\end{figure}

Graphs in Figures~\ref{fig:forbidden-configurations} and~\ref{fig:forbidden-configurations-unrestricted} are called \emph{forbidden configurations}.

\begin{definition}[Annotations]
  In a forbidden configuration, each vertex has one of the three annotations: in~$N_{G}[s]$, not in~$N_{G}[s]$, or uncertain; each edge between a vertex ``in~$N_{G}[s]$'' and a vertex ``not in~$N_{G}[s]$'' has one of the three annotations: in~$G$, not in~$G$, or uncertain.
\end{definition}
Note that there is no annotation on edges between two vertices ``in~$N_{G}[s]$,'' between two vertices ``not in~$N_{G}[s]$,'' or incident to an ``uncertain'' vertex; they are always ``in~$G$,'' ``in~$G$,'' and ``uncertain,'' respectively.
We say that the graph~$G^{s}$ \emph{contains (an annotated copy of) configuration~$F$} if there exists an isomorphism~$\varphi$ between~$F$ and an induced subgraph of~$G^{s}$ such that
  \begin{itemize}
  \item if $v$ is annotated ``in~$N_{G}[s]$'' (resp., ``not in~$N_{G}[s]$'') then $\varphi(v)\in N_{G}[s]$ (resp., $\varphi(v)\not\in N_{G}[s]$); and
  \item if an edge~$v u$ is annotated ``in~$G$'' (resp., ``not in~$G$'') then~$\varphi(v) \varphi(u)\in E(G)$ (resp., ~$\varphi(v) \varphi(u)\in E(G)$).
  \end{itemize}
Note that the vertex~$s$ is not involved in any forbidden configuration of~$G^{s}$.  By definition,~$s$ is universal in~$G^{s}$.  Every universal vertex in a forbidden configuration is adjacent to a vertex ``not in~$N_{G}[s]$'' with an edge annotated as ``in~$G$,'' which cannot happen for~$s$.
The following is from a pre-sequel~\cite{cao-24-cag-ii-flipping}.

\begin{theorem}[\cite{cao-24-cag-ii-flipping}]\label{thm:forbidden-configurations}
  Let~$G$ be a minimal forbidden induced subgraph of circular-arc graphs.
  If~$G$ is chordal, then for every simplicial vertex~$s$, the graph~$G^{s}$ contains a forbidden configuration shown in Figure~\ref{fig:forbidden-configurations}.
\end{theorem}

The list of forbidden configurations in Figure~\ref{fig:forbidden-configurations} was reduced from another list \cite{cao-24-cag-ii-flipping}.
Some of them are reproduced in Figure~\ref{fig:forbidden-configurations-unrestricted} and will be used in our proofs.
Note that Configuration~\ref{fig:p5+1-simplified} is a special case of Configuration~\ref{fig:p5+1-unlabeled unrestricted}, and others are similar.

\begin{lemma}[\cite{cao-24-cag-ii-flipping}] \label{thm:forbidden-configurations-unrestricted}
  Let~$G$ be a chordal circular-arc graph.
  For every simplicial vertex~$s$, the graph~$G^{s}$ cannot contain any configuration shown in Figure~\ref{fig:forbidden-configurations-unrestricted} .
\end{lemma}

\begin{figure}[h]
  \centering \small
  \begin{subfigure}[b]{0.2\linewidth}
    \centering
    \begin{tikzpicture}[scale=.75]
      \draw (-2.,0) node[empty vertex] (x1) {} -- (2.,0) node[empty vertex] (x2) {};
      \foreach[count=\i] \l/\t in {v_1/empty , v/a-, v_2/empty }
      \node[\t vertex] (v\i) at ({\i*1-2}, 0) {};
      \node[b-vertex] (u) at (90:1) {};
      \draw[witnessed edge] (u) -- (v2);
    \end{tikzpicture}
    \caption{}
    \label{fig:p5+1-unlabeled unrestricted}
  \end{subfigure}
  \begin{subfigure}[b]{0.2\linewidth}
    \centering
    \begin{tikzpicture}[scale=.75]
      \foreach[count=\i] \t in {b-, empty } {
        \node[empty vertex] (x\i) at ({\i},0) {};
      }
      \foreach \i in {1, 2} {
        \node[a-vertex] (v\i) at (3, {1.5-\i}) {};
        \node[b-vertex] (u\i) at (4, {1.5-\i}) {};
        \draw (x2) -- (v\i);
        \draw[witnessed edge] (v\i) -- (u\i);
      }
      \draw (x2)--(x1)--(v2)  (v2) -- (v1) -- (u2);
      \uncertain{u1}{v2};
    \end{tikzpicture}
    \caption{}
    \label{fig:double-fork+1-unlabeled unrestricted}
  \end{subfigure}
  \begin{subfigure}[b]{0.2\linewidth}
    \centering
    \begin{tikzpicture}[scale=.75]
      \draw (-2.,0) -- (2.,0);
      \node[a-vertex] (a) at (0, 1.) {};
      \node[b-vertex] (u) at (0, 0) {};
      \draw[witnessed edge] (a) -- (u);
      \node[b-vertex] (c) at (0, 0) {};
      \foreach[count=\j] \x in {-2, -1, 1, 2} {
        \draw (a) -- (\x, 0) node[empty vertex] (v\j) {};
      }
     \end{tikzpicture}
    \caption{}
    \label{fig:p5x1-unlabeled unrestricted}
  \end{subfigure}
  \begin{subfigure}[b]{0.2\linewidth}
    \centering
    \begin{tikzpicture}[scale=.75]
      \draw (1, 1) -- (4, 1);

       \node[empty vertex] (v4) at (3, 0.2) {};
      \foreach[count=\i] \t/\v/\x in {empty /x_1/1, empty /x_2/2, a-/{\quad v}/3.5, b-/u/5} {
        \draw (v4) -- (\i, 1) node[\t vertex] (u\i) {};
      }

      \draw (u3) -- ++(0, .85) node[empty vertex] (x2) {};
      \draw[witnessed edge] (u3) -- (u4);
     \end{tikzpicture}
    \caption{}
    \label{fig:whipping-top-1-unlabeled unrestricted}
  \end{subfigure}
  
  \begin{subfigure}[b]{0.2\linewidth}
    \centering
    \begin{tikzpicture}[scale=.7]
      \draw (1, 0) -- (2, 0) (4, 0) -- (3, 0);
      \draw[dashed] (3, 0) -- (2, 0);
      \node[empty vertex] (a) at (2.5, 1) {};
      \draw (a) -- ++(0, .75) node[b-vertex] {};

      \foreach[count=\i] \t/\l in {b-/u_{1}, empty /x_{1}, empty /x_{2}, empty /x_{p}} {
        \node[empty vertex] (u\i) at (\i, 0) {};
      }
      \foreach \i in {2, 3}
      \draw (a) -- (u\i);
    \end{tikzpicture}
    \caption{}
    \label{fig:dag-unlabeled unrestricted}
  \end{subfigure}
  \begin{subfigure}[b]{0.22\linewidth}
    \tikzstyle{every node}=[empty vertex]
    \centering
    \begin{tikzpicture}[scale=.75]
      \node[empty vertex] (x1) at (2, 1.75) {};
      \node[empty vertex] (x2) at (2.5, 1) {};
      \foreach \i in {3, 4} 
      \node[empty vertex] (x\i) at (\i-2, 0) {};
      \node[empty vertex] (v) at (1.5, 1) {};
      \node[empty vertex] (x5) at (3., 0) {};

      \foreach \i in {1, 4, 5} \draw (x2) -- (x\i);
      \foreach \i in {1, 3, 2, 4} \draw (v) -- (x\i);
      \draw[dashed] (x3) -- (x4) (x4) edge (x5);
    \end{tikzpicture}
    \caption{}\label{fig:ddag-unlabeled unrestricted}  
  \end{subfigure}
  \begin{subfigure}[b]{0.2\linewidth}
    \centering
    \begin{tikzpicture}[scale=.75]
      \node[empty vertex] (x1) at (2, 1.75) {};
      \node[empty vertex] (x2) at (2.5, 1) {};
      \foreach \i in {3, 4} 
      \node[empty vertex] (x\i) at (\i-2, 0) {};
      \node[a-vertex] (v) at (1.5, 1) {};
      \node[b-vertex] (x5) at (3., 0) {};

      \foreach \i in {1, 4, 5} \draw (x2) -- (x\i);
      \foreach \i in {1, ..., 5} \draw (v) -- (x\i);
      \draw (x3) -- (x4) (x4) edge[dashed] (x5);
      \draw[witnessed edge] (v) -- (x5) {};
    \end{tikzpicture}
    \caption{}
    \label{fig:ddag+e-unlabeled unrestricted}
  \end{subfigure}
  \begin{subfigure}[b]{0.2\linewidth}
    \centering
    \begin{tikzpicture}[scale=.75]
      \node[empty vertex] (x0) at (0, 2.75) {};
      \node[empty vertex] (u) at (0, 1) {};
      \foreach \i in {1, 2} {
        \node[rotate={180*\i}, a-vertex] (v\i) at ({1.5-\i}, 2) {};
        \node[rotate={180*\i}, b-vertex] (x\i) at ({(1.5-\i)*2}, 1) {};
      }
      \draw (u) -- (x2) (v1) -- (v2) (x1) edge[dashed] (u);
      \foreach \i in {1, 2} {
        \foreach \j in {1, 2} 
        \draw (u)--(v\i)--(x\j);
        \draw (v\the\numexpr3-\i\relax) -- (x0);
        \draw[witnessed edge] (v\i) -- (x\the\numexpr3-\i\relax);
      }
      \draw (v1)--(x1);
    \end{tikzpicture}
    \caption{}
    \label{fig:ddag+2e-unlabeled unrestricted}
  \end{subfigure}
  \caption{Some of the more relaxed configurations.
    There are at least six vertices in the bottom row.
  }
  \label{fig:forbidden-configurations-unrestricted} 
\end{figure}

Theorem~\ref{thm:forbidden-configurations} was proved on~$C_{4}$-free graphs, a superclass of chordal graphs \cite{cao-24-cag-ii-flipping}.  If~$G$ is a connected chordal graph, most configurations in Figure~\ref{fig:forbidden-configurations} will never appear.
The rest of this section is to pare the list, i.e., proving Theorem~\ref{thm:simplified-forbidden-configurations}.
Note that Configuration~\ref{fig:ddag+2e-simplified} has at least seven vertices, while the smallest of Configuration~\ref{fig:ddag+2e-unlabeled} has six vertices, which cannot happen.
The following simple observations are crucial for our analyses.
\begin{lemma}\label{lem:assumption}
  Let~$G$ be a minimal connected chordal graph that is not a circular-arc graph, and let~$U$ be the vertex set of an annotated copy of a forbidden configuration in~$G^{s}$.
  \begin{enumerate}[{{A}1)}] 
  \item\label{assumption:minimality} For any vertex set~$A\subsetneq V(G)$ such that~$G[A]$ does not have any universal vertex and any simplicial vertex~$x$ of~$G[A]$, the graph~$(G[A])^{x}$ does not contain any forbidden configuration.
  \item\label{assumption:N(s)} $\emptyset\subsetneq N_{G}(s) \subseteq U$. \item\label{assumption:replacable}
    If a vertex~$v\in U\cap N_{G}(s)$ is simplicial in~$U$ and annotated ``uncertain,'' there cannot be any vertex~$x\in V(G)\setminus U$ such that~$N_{G^{s}}(x)\cap U = N_{G^{s}}(v)\cap U$.
  \end{enumerate}
\end{lemma}
\begin{proof}
  A\ref{assumption:minimality}.
  If~$(G[A])^{x}$ contains any forbidden configuration, then~$G[A]$ is not a circular-arc graph by Theorem~\ref{thm:forbidden-configurations}, violating the minimality of~$G$.

  A\ref{assumption:N(s)}.
  First, $N_{G}(s)\ne \emptyset$ since~$G$ is connected and~$|V(G)| > 1$.
  If there is a vertex~$v\in N_{G}(s)\setminus U$, then~$(G - v)^{s}[U]$ is the same configuration, violating~A\ref{assumption:minimality}.

  A\ref{assumption:replacable}.  If such a vertex~$x$ exists, then~$(G - v)^{s}$ contains the same forbidden configuration as~$G^{s}[U]$, with~$v$ replaced by~$x$, violating~A\ref{assumption:minimality}.
\end{proof}

We now deal with the configurations in Figure~\ref{fig:forbidden-configurations} but not Figure~\ref{fig:simplified-forbidden-configurations}.  We process them in groups.
Throughout this section,~$G$ is a minimal connected chordal graph that is not a circular-arc graph.

\begin{figure}[ht]
  \centering \small
  \begin{subfigure}[b]{0.22\linewidth}
    \centering
    \begin{tikzpicture}[scale=.9]
      \node[a-vertex, "$v$"] (v) at (2, 1) {};
      \foreach \i in {1, 2, 3}
        \node[b-vertex, "$u_\i$" below] (u\i) at (\i, 0) {};
      \foreach \i in {1, 2, 3}
        \draw[very thick, teal] (v) -- (u\i);
    \end{tikzpicture}
    \caption{}
    \label{fig:claw}
  \end{subfigure}
  \begin{subfigure}[b]{0.22\linewidth}
    \centering
    \begin{tikzpicture}[scale=.9]
      \foreach \i in {1, 2}
      \node[a-vertex, "$v_\i$"] (v\i) at (.5+\i, 1) {};
      \foreach \i in {1, 2, 3} {
        \node[b-vertex, "$u_\i$" below] (u\i) at (\i, 0) {};
      }
      \draw (u3) -- (v1) -- (v2) -- (u1);
      \draw[witnessed edge] (v1) -- (u1) (v2) -- (u2) (v2) -- (u3);
      \uncertain{v1}{u2};
\end{tikzpicture}
    \caption{}
    \label{fig:double-claw}
  \end{subfigure}
  \begin{subfigure}[b]{0.22\linewidth}
    \centering
    \begin{tikzpicture}[scale=.85]
      \foreach[count=\i] \t in {b-, a-} {
        \draw ({\i},0) node[\t vertex] (x\i) {} ++ (0, 1em) node {$x_{\i}$};
      }
\draw (3,0) node[a-vertex] (v) {} ++ (0, 1em) node {$v$};      
      \foreach \i in {1, 2} {
        \foreach[count=\x from 4] \v/\t in {u/b-} 
          \draw (\x, {1.5-\i}) node[\t vertex] (\v\i) {} ++ (0, \inteval{3-2*\i}em) node {$\v_{\i}$};
        \draw[witnessed edge] (v) -- (u\i);
      }
      \draw (v)--(x2)--(x1);
    \end{tikzpicture}
    \caption{}
    \label{fig:fork}
\end{subfigure}
  \,
  \begin{subfigure}[b]{0.24\linewidth}
    \centering
    \begin{tikzpicture}[scale=.85]
      \foreach[count=\i] \t in {b-, a-} {
\draw ({\i},0) node[\t vertex] (x\i) {} ++ (0, 1em) node {$x_{\i}$};
      }
      \foreach \i in {1, 2} {
        \foreach[count=\x from 3] \v/\t in {v/a-, u/b-} 
          \draw (\x, {1.5-\i}) node[\t vertex] (\v\i) {} ++ (0, \inteval{3-2*\i}em) node {$\v_{\i}$};
        \draw (x2) -- (v\i);
        \draw[witnessed edge] (v\i) -- (u\i);
      }
      \draw (x2)--(x1)--(v2)  (u1) -- (v2) -- (v1) -- (u2);
    \end{tikzpicture}
    \caption{}
    \label{fig:double-fork+1}
  \end{subfigure}

  \begin{subfigure}[b]{0.22\linewidth}
    \centering
    \begin{tikzpicture}[scale=.75]
      \node[a-vertex, "$v$"] (v) at (3, 0) {};
      \node[b-vertex, "$w_{0}$" below] (w0) at (2.5, -.75) {};
      \foreach \i in {1, 2} {
        \node[b-vertex, "$u_{\i}$" right] (u\i) at (4., {1.5-\i}) {};
        \node[rotate={180*\i}, b-vertex, "$w_{\i}$" below] (w\i) at (3.5, {(1.5-\i)*1.5}) {};
        \draw[witnessed edge] (v) -- (u\i);
        \draw (v) -- (w\i) -- (u\i);
      }
      \foreach[count=\i] \t in {b-, a-} {
        \node[\t vertex, "$x_{\i}$"] (x\i) at ({\i},0) {};
      }
      \foreach \v in {v, x2, w2} \draw (w0) -- (\v);
      \draw (v)--(x2)--(x1);    
      \draw[witnessed edge] (x2) -- (w2);
    \end{tikzpicture}
    \caption{}\label{fig:lem:forks-1}
  \end{subfigure}
  \begin{subfigure}[b]{0.22\linewidth}
    \centering
    \begin{tikzpicture}[scale=.75]
      \node[a-vertex, "$v$"] (v) at (3, 0) {};
      \node[b-vertex, "$w_{0}$" below] (w0) at (2.5, -1) {};
      \foreach[count=\i] \v in {x, u} {
        \node[b-vertex, "$\v_{2}$" right] (u\i) at (4.5, {1.5-\i}) {};
        \node[rotate={180*\i}, b-vertex, "$w_{\i}$" below] (w\i) at (3.8, {(1.5-\i)*2}) {};
        \draw[witnessed edge] (v) -- (u\i);
        \draw (v) -- (w\i) -- (u\i);
      }
      \foreach \v in {v, w2} \draw (w0) -- (\v);
      \draw (w2) -- (u1) -- (u2);
\end{tikzpicture}
    \caption{}\label{fig:lem:forks-2}
  \end{subfigure}
  \begin{subfigure}[b]{0.22\linewidth}
    \centering
    \begin{tikzpicture}[scale=.85]
      \foreach[count=\i] \t in {b-, a-} {
        \node[\t vertex, "$x_{\i}$" below] (x\i) at ({\i}, 0.5) {};
      }
      \node[a-vertex, "$v_{1}$" below] (v1) at (3, .5) {};
      \foreach \i in {1, 2} {
        \node[b-vertex, "$u_{\i}$" below] (u\i) at (4., {1.5-\i}) {};
      }
      \draw (x2) -- (v1);
      \draw[witnessed edge] (v1) -- (u1);
      \node[b-vertex, "$w_{1}$"] (w1) at (3.5, 1.2) {};
      \draw (x2)--(x1)  (v1) -- (u2);
      \draw (u1)--(w1)--(v1);
      \draw[bend left, witnessed edge] (x2) edge (w1);
    \end{tikzpicture}
    \caption{}
    \label{fig:lem:forks-3}
  \end{subfigure}
  \,
  \begin{subfigure}[b]{0.24\linewidth}
    \centering
    \begin{tikzpicture}[scale=.85]
      \foreach[count=\i] \t in {b-, a-} {
        \node[\t vertex, "$x_{\i}$"] (x\i) at ({\i}, 0) {};
      }
\node[b-vertex, "$w_{2}$" below] (w2) at (3.5, -1.2) {};
      \foreach \i/\v in {1/w, 2/u} 
        \node[rotate={180*\i}, a-vertex, "$v_{\i}$" below] (v\i) at (3, {1.5-\i}) {};
      \foreach \i/\v/\p in {1/w/, 2/u/below} {
        \node[b-vertex, "$\v_{\i}$" \p] (\v\i) at (4., {1.5-\i}) {};
\draw (x2) -- (v\i);
        \foreach \j in {1, 2} \draw (w\i)--(v\j);
      }
      \draw[witnessed edge] (v2) -- (u2);
      \draw (x2)--(x1)  (v2) -- (v1) -- (u2);
      \draw (u2)--(w2)--(v2) (w2)--(v1);
      \draw[bend right] (x1) edge (v2);
      \draw[bend right, witnessed edge] (x2) edge (w2);
    \end{tikzpicture}
    \caption{}
    \label{fig:lem:forks-4}
  \end{subfigure}
  \caption{Illustrations for Lemma~\ref{lem:group-0}.}
\end{figure}

\begin{lemma}\label{lem:group-0}
  The graph~$G^{s}$ cannot contain
Configurations~\ref{fig:claw-simplified}--\ref{fig:double-fork+1-simplified}.
\end{lemma}
\begin{proof}
For Configuration~\ref{fig:claw-simplified}, we use the labels in Figure~\ref{fig:claw}.
  For~$i = 1, 2, 3$, let~$w_{i}$ be a witness of the edge~$v u_{i}$.
  By Proposition~\ref{lem:P3}, $w_{1}$, $w_{2}$, and~$w_{3}$ are distinct, and there is no edge among them.
  Thus,~$G[U\cup\{w_{1}, w_{2}, w_{3}\}]$ is a long claw, where the vertex~$v$ has degree three, and for~$i = 1, 2, 3$, the vertices~$u_{i}$ and~$w_{i}$ have degree two and one, respectively.
  This violates Assumption (minimality) since the long claw does not involve~$s$.
  Thus,~$G^{s}$ cannot contain Configuration~\ref{fig:claw-simplified}.
  
  \bigskip
  For Configuration~\ref{fig:double-claw-simplified}, we use the labels in Figure~\ref{fig:double-claw}.
  Let~$W$~be all the witnesses of the edge~$u_{1} v_{1}$; note that~$W\cap U=\emptyset$.
  By the selection of~$W$, the edge~$u_{1} v_{1}$  is missing in~$(G - W)^{s}$
  We argue that all the other edges in~$G^{s}[U]$ remain in~$(G - W)^{s}$.
By Proposition~\ref{lem:P3}, no vertex in~$W$ witnesses~$u_{2} v_{1}$.
  If a vertex~$w \in W$ witnesses~$u_{2} v_{2}$, then~$w u_{1} v_{1} v_{2} u_{2}$ is a hole of~$G$.
  For the same reason, no vertex in~$W$ witnesses~$u_{3} v_{2}$.
  By Proposition~\ref{lem:unwitnessed}, the clique~$\{v_{1}, v_{2}, u_{3}\}$ is witnessed in~$G^{s}$.  Since its witnesses are not in~$W$, it remains witnessed in~$(G - W)^{s}$.
  Thus,~$(G - W)^{s}[U]$ is Configuration~\ref{fig:dag+2e-simplified}.
  This violates A\ref{assumption:minimality}, and hence~$G^{s}$ cannot contain Configuration~\ref{fig:double-claw-simplified}.
  
  \bigskip
  For Configuration~\ref{fig:fork-simplified}, we use the labels in Figure~\ref{fig:fork}.  
  Let~$w_{0}$ be a witness of edge~$x_{2} v$, and for~$i = 1, 2$, let~$w_{i}$ be a witness of edge~$v u_{i}$.
  By Proposition~\ref{lem:P3},~$w_{1}\ne w_{2}$ and they are not adjacent.  
  By Assumption (minimality),
  \[
    V(G) = U\cup \{s, w_{0}, w_{1}, w_{2}\}.
  \]
  By Proposition~\ref{lem:nonadjacent->dominate},~$N_{G^{s}}(w_{0})\cap U = \{x_{2}, v\}$, and hence~$w_{0}\not\in \{w_{1}, w_{2}\}$ and~$N_{G}(w_{0})\cap U = \emptyset$.
  Since~$G$ is connected,~$w_{0}$ is adjacent to at least one of~$w_{1}$ and~$w_{2}$.
  On the other hand,~$w_{0}$ cannot be adjacent to both~$w_{1}$ and~$w_{2}$; otherwise,~$w_{0} w_{1} u_{1} v u_{2} w_{2}$ is a hole of~$G$.
  Assume without loss of generality that~$w_{0}$ is adjacent to~$w_{2}$.
  Then~$w_{2} x_{2}$ is a collateral edge by Proposition~\ref{lem:nonadjacent->dominate}; see Figure~\ref{fig:lem:forks-1}.
Note that~$N_{G}(x_{1}) = \{v\}$.
  The graph~$(G - \{s, u_{1}\})^{x_{1}}$ contains Configuration~\ref{fig:ddag+e-unlabeled} of order six; see Figure~\ref{fig:lem:forks-2}.
Note that the subgraph of~$(G - \{s, u_{1}\})^{x_{1}}$ induced by~$V(G)\setminus \{s, x_{1}, x_{2}, u_{1}\}$ is the same as in~$G^{s}$, and the vertex~$x_{2}$ is adjacent to~$u_{2}, w_{1}, w_{2}$, and also~$v$, where~$v x_{2}$ is a collateral edge witnessed by~$w_{1}$.
  This violates A\ref{assumption:minimality}, and hence~$G^{s}$ cannot contain Configuration~\ref{fig:fork-simplified}.
  
  \bigskip
  For Configuration~\ref{fig:double-fork+1-simplified}, we use the labels in Figure~\ref{fig:double-fork+1}.
  By Lemma~\ref{lem:witness-N[s]}, we can find a witness~$w_{0}$ of the clique~$\{v_{1}, v_{2}, x_{2}\}$.
  For~$i = 1, 2$, let~$w_{i}$ be a witness of the clique~$\{u_{i}, v_{1}, v_{2}\}$, which exists because of Proposition~\ref{lem:unwitnessed}.
By Assumption (minimality),
  \[
    V(G) = U\cup \{w_{0}, w_{1}, w_{2}\}.
  \]
By Proposition~\ref{lem:nonadjacent->dominate},~$w_{0}\not\in \{w_{1}, w_{2}\}$.
  Since both~$w_{1} u_{1} v_{1} v_{2}$ and~$v_{1} v_{2} u_{2} w_{2}$ are induced paths of~$G$, vertices~$w_{1}$ and~$w_{2}$ are different and there is no edge between~$\{u_{1}, w_{1}\}$ and~$\{u_{2}, w_{2}\}$.
As a result,~$w_{0}$ is adjacent to at most one of~$w_{1}$ and~$w_{2}$.
  Since~$G$ is connected,~$w_{0}$ is adjacent to at least one of~$w_{1}$ and~$w_{2}$.
For~$i = 1, 2$, the vertex~$w_{i}$ and~$x_{2}$ are adjacent in~$G$ by Proposition~\ref{lem:nonadjacent->dominate}, and thus they are adjacent in~$G^{s}$ if and only if~$w_{i}$ is adjacent to~$w_{0}$.
  If~$w_{0}$ and~$w_{1}$ are adjacent, then~$(G - v_{2})^{s}$ contains Configuration~\ref{fig:dag+e-simplified} ($\{u_{2}, v_{1}, x_{1}, x_{2}, w_{1}, u_{1}\}$); see Figure~\ref{fig:lem:forks-3}.
  Otherwise,~$w_{0}$ and~$w_{2}$ are adjacent, and hence~$w_{2} x_{2}$ is a collateral edge by Proposition~\ref{lem:nonadjacent->dominate}; see Figure~\ref{fig:lem:forks-4}.
  Then~$(G - u_{1})^{s}$ contains Configuration~\ref{fig:ddag+e-unlabeled} ($\{w_{1}, v_{2}, v_{1}, x_{1}, x_{2}, w_{2}, u_{2}\}$), where~$v_{2}$ is universal and~$v_{2} u_{2}$ is a collateral edge,~$x_{1} x_{2} w_{2} u_{2}$ is an induced path, while~$w_{1}$ and~$v_{1}$ have degrees~two and~five, respectively.
  All violate A\ref{assumption:minimality}, and hence~$G^{s}$ cannot contain Configuration~\ref{fig:double-fork+1-simplified}.
\end{proof}

\begin{figure}[ht]
  \centering \small
  \begin{subfigure}[b]{0.25\linewidth}
    \centering
    \begin{tikzpicture}[scale=.75]
      \draw (-2.,0) node[b-vertex, "$x_1$" below] (x1) {} -- (2.,0) node[b-vertex, "$x_2$" below] (x2) {};
      \foreach[count=\i] \l in {v_1, v_{0}, v_2}
      \node[a-vertex, "$\l$" below] (v\i) at ({\i*1-2}, 0) {};
      \node[b-vertex, "$u$"] (u) at (90:1) {};
      \draw[witnessed edge] (u) -- (v2);
    \end{tikzpicture}
    \caption{}
    \label{fig:p5+1}
  \end{subfigure}
  \begin{subfigure}[b]{0.25\linewidth}
    \centering
    \begin{tikzpicture}[scale=.72]
      \draw (-2.,0) -- (2.,0);
      \node[a-vertex, "$v_{0}$"] (a) at (0, 1.) {};
\node[b-vertex, "$u$"] (u) at (0.75, 1) {};
\foreach[count=\j] \x/\p in {-2/x_1, -1/v_{1}, 0/x_{0}, 1/v_{2}, 2/x_{2}} {
        \draw (a) -- (\x, 0) node[empty vertex, "$\p$" below] (x\j) {};
      }
      \foreach \i in {1, 3, 5} \node[b-vertex] at (x\i) {};
      \draw (a) edge[witnessed edge] (u) (u) edge (x3);\end{tikzpicture}
    \caption{}
    \label{fig:bent-whipping-top}
\end{subfigure}
  \begin{subfigure}[b]{0.25\linewidth}
    \centering
    \begin{tikzpicture}[xscale=.75, yscale=.6]
      \node[empty vertex, "$v_{0}$" below] (c) at (0, 0) {};
      \foreach[count=\i] \j/\p in {1/below, 3/right, 2/below} {
        \node[b-vertex, "$x_{\j}$" \p] (u\i) at ({90*(3-\i)}:2) {};
        \node[empty vertex, "$v_{\j}$" \p] (v\i) at ({90*(3-\i)}:1) {};
        \draw (u\i) -- (v\i) -- (c);
      }
    \end{tikzpicture}
    \caption{}
    \label{fig:long-claw}
  \end{subfigure}

  \begin{subfigure}[b]{0.25\linewidth}
    \centering
    \begin{tikzpicture}\draw (-1.,0) -- (2.,0) node[b-vertex, "$x_2$" below] (x2) {};
      \foreach[count=\i] \l/\t in {w_1/b-, v_{0}/a-, v_2/a-}
      \node[\t vertex, "$\l$" below] (v\i) at ({\i*1-2}, 0) {};
      \node[b-vertex, "$u$"] (u) at (90:1) {};
      \node[b-vertex, "$v_{1}$"] (new) at (1, 1) {};
      \draw[witnessed edge] (u) -- (v2);
      \foreach \i in {2, 3}
      \draw[witnessed edge] (new) -- (v\i);
      \foreach \v in {u, x2}
      \draw (new) -- (\v);
    \end{tikzpicture}
    \caption{}
    \label{fig:p5+1-1}
  \end{subfigure}
  \begin{subfigure}[b]{0.25\linewidth}
    \centering
    \begin{tikzpicture}[scale=.6]
      \node[empty vertex] (c) at (0, 0) {} ++ (.2, -.3) node {$u$};
      \foreach \i/\l/\p in {1/v_{0}/, 2/w_{2}/below, 3/w_{1}/below} {
          \node[empty vertex, "$\l$" \p] (u\i) at ({120*\i-30}:2) {};
          \draw ({120*\i-90}:1) -- (u\i) -- ({120*\i+30}:1);
          \draw ({120*\i-90}:1) -- ({120*\i+30}:1);
      }
      \foreach[count =\j from 3] \i/\l/\p in {1/v_{1}/, 2/x_{0}/below, 3/v_{2}/} {
        \node[empty vertex, "$\l$" \p] (v\i) at ({120*\i+30}:1) {};
        \draw (v\i) -- (c);
      }
      \draw (u1) -- (c);
      \foreach \i/\l in {1/2, 3/1} {
        \draw (v\i) -- ({90 - 60*(\i-2)}:2) node[empty vertex, "$x_{\l}$"] {};
      }
    \end{tikzpicture}
    \caption{}
    \label{fig:bent-whipping-top-1}
  \end{subfigure}
  \begin{subfigure}[b]{0.25\linewidth}
    \centering
    \begin{tikzpicture}[scale=.6]
      \node[empty vertex] (c) at (0, 0) {} ++ (.2, -.3) node {$v_{3}$};
      \foreach \i/\l/\p in {1/x_{3}/, 2/w_{2}/below, 3/w_{1}/below} {
          \node[empty vertex, "$\l$" \p] (u\i) at ({120*\i-30}:2) {};
          \draw ({120*\i-90}:1) -- (u\i) -- ({120*\i+30}:1);
          \draw ({120*\i-90}:1) -- ({120*\i+30}:1);
      }
      \foreach[count =\j from 3] \i/\l/\p in {1/v_{1}/, 2/v_{0}/below, 3/v_{2}/} {
        \node[empty vertex, "$\l$" \p] (v\i) at ({120*\i+30}:1) {};
        \draw (v\i) -- (c);
      }
      \draw (u1) -- (c);
      \foreach \i/\l in {1/2, 3/1} {
        \draw (v\i) -- ({90 - 60*(\i-2)}:2) node[empty vertex, "$x_{\l}$"] {};
      }
    \end{tikzpicture}
    \caption{}
    \label{fig:long-claw-1}
  \end{subfigure}
  \caption{Illustrations for Lemma~\ref{lem:group-p5}.}
\label{fig:group-2}
\end{figure}

\begin{lemma}\label{lem:group-p5}
  The graph~$G^{s}$ cannot contain Configuration~\ref{fig:p5+1-simplified},~\ref{fig:bent-whipping-top-simplified}, or~\ref{fig:long-claw-simplified}.
\end{lemma}
\begin{proof}
  For Configuration~\ref{fig:p5+1-simplified}, we use the labels in Figure~\ref{fig:p5+1}.
  Let~$w_{0}$,~$w_{1}$, and $w_{2}$ be a witnesses of the collateral edges~$v_{0} u$,~$v_{0} v_{1}$, and $v_{0} v_{2}$, respectively.
  They are distinct by Proposition~\ref{lem:nonadjacent->dominate}.
  By Assumption (minimality),
  \[
    V(G) = U\cup \{s, w_{0}, w_{1}, w_{2}\}.
  \]
  Note that~$N_{G}(x_{1}) = \{v_{0}, v_{2}\}$.
  In~$(G - s)^{x_{1}}$, the set~$V(G)\setminus \{s, x_{1}, v_{1}\}$ induces the same subgraph as in~$G^{s}$, and the vertex~$v_{1}$ is adjacent to~$u$ and~$x_{2}$, and hence~$v_{0}$ and~$v_{2}$ (there is a hole in $(G - s)^{x_{1}}$ otherwise).
  See Figure~\ref{fig:p5+1-1}.
  Since the edges $v_{1}v_0$ and $v_{1}v_{2}$ are collateral, $(G - s)^{x_{1}}$ contains Configuration~\ref{fig:whipping-top-1-unlabeled} ($U\cup \{w_{1}\}\setminus \{x_{1}\}$).
  This violates A\ref{assumption:minimality}, and hence~$G^{s}$ cannot contain Configuration~\ref{fig:p5+1-simplified}.
  
  \bigskip
For Configuration~\ref{fig:bent-whipping-top-simplified}, we use the labels in Figure~\ref{fig:bent-whipping-top}.
Note that
  \begin{equation}
    \label{eq:35}
    \text{for~$i = 1, 2$, the edge~$v_{0} v_{i}$ is collateral if and only if~$v_{i}\in N_{G}(s)$.}
  \end{equation}
  Otherwise,~$x_0$ is a witness of both~$v_0u$ and $v_0 v_{i}$, violating Proposition~\ref{lem:P3}.

  We then claim that
  \begin{equation}
    \label{eq:36}
    \text{for~$i = 1, 2$, the edge~$v_{i} x_{i}$ is not collateral.}
  \end{equation}
  We consider~$i =1$, and the other holds by symmetry.  
  Suppose that the edge~$v_{1} x_{1}$ is collateral; then~$v_{1}\in N_{G}(s)$.
  By Proposition~\ref{lem:unwitnessed}, there is a witness~$w$ of the clique~$\{v_{0}, v_{1}, x_{1}\}$.
  By Proposition~\ref{lem:nonadjacent->dominate} and A\ref{assumption:replacable},~$N_{G^{s}}(w)\cap U$ must be~$\{v_{0}, v_{1}, x_{1}\}$.  Then~$(G - x_{1})^{s}$ contains an annotated copy of the same configuration, with~$x_{1}$ replaced by~$w$.
This violates A\ref{assumption:minimality}, and hence~$v_{1} x_{1}$ is not collateral.

For~$i = 1, 2$, if~$v_{i}\in N_{G}(s)$, the edge~$v_{i} x_{0}$ is collateral by Corollary~\ref{lem:P4}.
  Moreover, since the edge~$v_{0} x_{0}$ is not collateral, there is a witness~$w_{i}$ of the clique~$\{v_{0}, v_{i}, x_{0}\}$ by Proposition~\ref{lem:unwitnessed}.
  By Proposition~\ref{lem:nonadjacent->dominate} and A\ref{assumption:replacable},~$N_{G^{s}}(w_{i}) \cap U = \{v_{0}, v_{i}, x_{0}\}$.
  We take~$w_{i}$ to be a witness of~$\{v_{0}, v_{i}, x_{0}\}$ if~$v_{i}\in N_{G}(s)$, or~$v_{i}$ otherwise.
  By Assumption (minimality),
  \[
    V(G) = U\cup \{s, w_{1}, w_{2}\}.
  \]
  If only one of~$v_{1}$ and~$v_{2}$ is in~$N_{G}(s)$, say~$v_{1}$, then~$x_{1}$ is isolated by~\eqref{eq:36}, violating Assumption (connectivity).
If both~$v_{1}$ and~$v_{2}$ are in~$N_{G}(s)$, then~$G - s$ is Figure~\ref{fig:bent-whipping-top-1}, which is isomorphic to Figure~\ref{fig:whipping-top-derived}. 
In the rest, neither~$v_{1}$ nor~$v_{2}$ is in~$N_{G}(s)$.
  By~\eqref{eq:35},~$v_{1}, v_{2}\not\in N_{G}(v_{0})$.
  Then~$G - s$ is a long claw, where~$u$ is the only neighbor of~$v_{0}$.
  Both violate Assumption (minimality), and hence~$G^{s}$ cannot contain Configuration~\ref{fig:bent-whipping-top-simplified}.
  
  \bigskip
   For Configuration~\ref{fig:long-claw-simplified}, we use the labels in Figure~\ref{fig:long-claw}.
   We start with arguing that 
  \begin{equation}
    \label{eq:23}
    \text{for~$i = 1, 2, 3$, the edge~$x_{i} v_{i}$ is not collateral.} \end{equation}
  We consider~$i = 1$, and the others are symmetric.
  Suppose that~$x_{1} v_{1}$ is collateral, and let~$w$ be its witness.
  By Proposition~\ref{lem:nonadjacent->dominate} and A\ref{assumption:replacable},~$N_{G^{s}}(w)\cap U = \{x_{1}, v_{1}\}$, and then~$(G - x_{1})^{s}$ contains Configuration~\ref{fig:long-claw-simplified} ($U\cup w\setminus \{x_{1}\}$).

  For~$i = 1, 2, 3$, let~$w_{i}$ be a witness of the edge~$v_{0} v_{i}$ if it is collateral.
  By Assumption (minimality),
  \[
    V(G) = U\cup \{s, w_{i}\mid 1\le i\le 3, v_{0} v_{i} \in E(G)\}.
  \]
Suppose first that~$v_{0}\in N_G(s)$.
  For all~$i = 1, 2, 3$, the edge~$v_{0} v_{i}$ is collateral by Corollary~\ref{lem:P4}.
  We note~$v_{i}\in N_G(s)$; otherwise,~$(G - x_{i})^{s}$ contains Configuration~\ref{fig:p5+1-unlabeled unrestricted}, violating Lemma~\ref{thm:forbidden-configurations-unrestricted}.
By Proposition~\ref{lem:nonadjacent->dominate}, $N_{G^{s}}(w_{i}) = \{v_{0}, v_{i}\}$.
  Then~$G - s$ is isomorphic to Figure~\ref{fig:long-claw-derived}, where for~$i = 1, 2, 3$, the vertex~$v_{i}$ has degree seven,~$x_{i}$ degree three, and~$w_{i}$ degree two, while~$v_{0}$ has degree six.  This violates Assumption (minimality).

  Henceforth,~$v_{0}\not\in N_G(s)$.
  If~$v_{i} \in N_G(s)$, the edge~$v_{0} v_{i}$ is collateral by Corollary~\ref{lem:P4}, and~$N_{G^{s}}(w_{i}) = \{v_{0}, v_{i}\}$ by Proposition~\ref{lem:nonadjacent->dominate} and A\ref{assumption:replacable}.
  By Proposition~\ref{lem:P3}, there are no edges among the witnesses.
  If~$v_{i}, i = 1, 2, 3$, is the only vertex in~$U\cap N_G(s)$, then~$x_{i}$ is an isolated vertex in~$G$.
  If~$v_{1}, v_{2}, v_{3}\in N_G(s)$, then~$G - s$ is isomorphic to Figure~\ref{fig:long-claw-derived}, where for~$i = 1, 2, 3$, the vertex~$v_{i}$ has degree seven,~$w_{i}$ degree three, and~$x_{i}$ degree two, while~$v_{0}$ has degree six. 
  Otherwise, assume without loss of generality that~$v_{1}, v_{2}\in N_G(s)$ and~$v_{3}\not\in N_G(s)$.
  Then~$G - s$ is Figure~\ref{fig:long-claw-1}, which is isomorphic to Figure~\ref{fig:whipping-top-derived}. 
  Both violate Assumption (minimality), and hence~$G^{s}$ cannot contain Configuration~\ref{fig:long-claw-simplified}.
\end{proof}

\begin{figure}[ht]
  \centering \small
  \begin{subfigure}[b]{0.22\linewidth}
    \centering
    \begin{tikzpicture}[scale=.9]
      \draw (1, 0) -- (4, 0);
      \foreach[count=\x from 2] \t/\l/\p in {b-/u/above, b-/x_{3}/below}  \node[\t vertex, "$\l$"] (\l) at (\x, 1) {};
      
      \foreach[count=\i] \t/\l in {b-/x_{1}, a-/v_{1}, a-/v_{2}, b-/x_{2}} {
        \node[\t vertex, "$\l$" below] (u\i) at (\i, 0) {};
      }
      \foreach \i in {2, 3} \draw (u) edge[witnessed edge] (u\i) (u\i) -- (x_{3});
    \end{tikzpicture}
    \caption{}
    \label{fig:ab-wheel}
  \end{subfigure}
  \begin{subfigure}[b]{0.2\linewidth}
    \centering
    \begin{tikzpicture}[scale=.75]
      \draw (1, 0) -- (4, 0);
      \node[a-vertex, "$v$" right] (a) at (2.5, 1) {};
\path (a) ++(0, .75) node[b-vertex, "$x_{0}$" right] (x) {};
      \uncertain{a}{x};

      \foreach[count=\i] \t/\l in {empty /x_{1}, b-/u_{1}, b-/u_{2}, empty /x_{2}} {
        \draw (a) -- (\i, 0) node[b-vertex, "$\l$" below] (u\i) {};
      }
      \foreach \i in {2, 3}
      \draw[witnessed edge] (a) -- (u\i);
    \end{tikzpicture}
    \caption{}
    \label{fig:(p4+p1)*1}
  \end{subfigure}
  \begin{subfigure}[b]{0.2\linewidth}
    \centering
    \begin{tikzpicture}[scale=.5]
      \foreach \i in {1, 2, 3} {
        \draw ({120*\i+90}:2) -- ({120*\i-30}:2);
        \draw ({120*\i-90}:1) -- ({120*\i+30}:1);
      }
      \foreach[count =\j from 3] \i/\t in {1/a-, 2/a-, 3/empty } {
        \node[] at ({270-120*\i}:{1.7}) {\pgfmathparse{int(Mod(\i,3))}$v_{\pgfmathresult}$};
        \node[\t vertex] (v\i) at ({270-120*\i}:1) {};
      }
      \foreach \i in {1, 2, 3} {
        \pgfmathparse{int(Mod(\i,3))}
        \node[rotate={int(\i/3)*180}, b-vertex, "$u_{\pgfmathresult}$" below] (u\i) at ({90-120*\i}:2) {};
      }
      \foreach \i in {1, 2} 
      \draw[witnessed edge] (u\i) -- (v\i);
    \end{tikzpicture}
    \caption{}
    \label{fig:sun+2}
  \end{subfigure}

\begin{subfigure}[b]{0.2\linewidth}
    \centering
    \begin{tikzpicture}[xscale=.5, every node/.style={empty vertex}]
      \draw (-2, 0) -- (2, 0);

      \node["$u$"] (u) at (0, 1) {};
      \foreach[count=\i] \x in {-1, 1} {
        \node["$v_\i$" below] (v\i) at (\x, 0) {};
        \draw (u) -- (v\i);
        \node["$x_{\inteval{3-\i}}$" below] (v\inteval{3-\i}) at ({\x*2}, 0) {};
      }
      \draw (u) -- ++(1, 0) node["$w$"] {} -- ++ (1, 0) node["$x_{3}$"] {};
    \end{tikzpicture}
    \caption{}
    \label{fig:ab-wheel-0}
  \end{subfigure}
  \quad
  \begin{subfigure}[b]{0.3\linewidth}
    \centering
    \begin{tikzpicture}[scale=1.]
      \draw (1, 0) -- (3, 0);
      \foreach[count=\x from 2] \t/\l/\p in {b-/u/above, b-/x_{3}/below}  \node[\t vertex, "$\l$"] (\l) at ({\x*1.5 - 1}, 1) {};
      
      \foreach[count=\i] \t/\l in {empty /x_{1}, a-/v_{1}, a-/v_{2}} {
        \node[\t vertex, "$\l$" below] (u\i) at (\i, 0) {};
      }
      \node[b-vertex, "$w_{1}$"] (w1) at (1.25, 1) {};
      \node[b-vertex, "$w_{2}$"] (w2) at (2.75, 1) {};            
      \foreach \i in {2, 3} \draw (u) edge[witnessed edge] (u\i) (u\i) -- (x_{3});
      \foreach \i in {1, 2} \draw (u) -- (w\i) -- (u\inteval{\i+1});
      \draw (x_{3}) -- (w2);
      \draw (w2) edge[witnessed edge] (u2);
      \draw (w1) edge[densely dotted, thick] (u1);
    \end{tikzpicture}
    \caption{}
    \label{fig:ab-wheel-1}
  \end{subfigure}
  \caption{Illustrations for the proof of Lemma~\ref{lem:group-5}.  The edge~$w_{1} x_{1}$ in~\ref{fig:ab-wheel-1} may or may not exist.}
\end{figure}

\begin{lemma}\label{lem:group-5}
  The graph~$G^{s}$ cannot contain Configuration~\ref{fig:ab-wheel-simplified},~\ref{fig:(p4+p1)*1-simplified}, or~\ref{fig:ddag+2e-unlabeled} on six vertices.
\end{lemma}
\begin{proof}
  For Configuration~\ref{fig:ab-wheel-simplified}, we use the labels in Figure~\ref{fig:ab-wheel}.
If the clique~$\{u, v_{1}, v_{2}\}$ has a witness~$w$, then~$V(G) = U\cup\{s, w\}$ by Assumption (minimality).
  Since~$G$ is connected, the vertex~$x_{3}$ must be adjacent to~$w$.  
  Then~$G - s$, shown in Figure~\ref{fig:ab-wheel-0}, is isomorphic to Figure~\ref{fig:extended-net}, violating Assumption (minimality).

  In the rest, the clique~$\{u, v_{1}, v_{2}\}$ is not witnessed.
  For~$i = 1, 2$, let~$w_{i}$ be a witness of the edge~$u v_{i}$.
  By Assumption (minimality),~$V(G) = U\cup\{s, w_{1}, w_{2}\}$ (note that the collateral edge~$v_{1} v_{2}$ is witnessed by~$x_3$).
  By Proposition~\ref{lem:P3},~$w_{1}$ and~$w_{2}$ are not adjacent.
  Since~$G$ is connected, the vertex~$x_{3}$ must be adjacent to~$w_{1}$ or~$w_{2}$.
On the other hand,~$x_{3}$ cannot be adjacent to both~$w_{1}$ and~$w_{2}$; otherwise,~$w_{1} u w_{2} x_{3}$ is a hole of~$G$.
  Assume without loss of generality that~$w_{2}$ and~$x_{3}$ are adjacent; see Figure~\ref{fig:ab-wheel-1}.
  Note that~$v_{1} w_{2}$ is a collateral edge witnessed by~$x_{3}$.
In the graph~$(G - \{v_{2}, x_{2}\})^{s}$, the subgraph induced by~$V(G)\setminus \{s, v_{2}, x_{2}\}$ is Configuration~\ref{fig:p5x1-unlabeled} (when~$w_{1}$ and~$x_{1}$ are adjacent) or Figure~\ref{fig:(p4+p1)*1-simplified} (when~$w_{1}$ and~$x_{1}$ are not adjacent), violating A\ref{assumption:minimality}.
Thus,~$G^{s}$ cannot contain Configuration~\ref{fig:ab-wheel-simplified}.

  \bigskip
  For Configuration~\ref{fig:(p4+p1)*1-simplified}, we use the labels in Figure~\ref{fig:(p4+p1)*1}.  
  Note that~$x_{i}$ witnesses~$v u_{i}, i = 1, 2$.
  If~$v x_{0}$ is not collateral, then~$V(G) = U\cup \{s\}$ by Assumption (minimality), but~$x_{0}$ is isolated, violating Assumption.
  Hence,~$v x_{0}$ is collateral, and let~$w$ be its witness.
  By Proposition~\ref{lem:P3},~$w$ cannot be adjacent to~$u_{1}$ or~$u_{2}$.
  If~$w$ is adjacent to both~$x_{1}$ and~$x_{2}$, then~$w x_{1} u_{1}u_{2} x_{2}$ is a hole of~$G$.
  Otherwise, depending on whether~$w$ is adjacent to one of~$x_{1}$ and~$x_{2}$, the subgraph of~$(G - x_{0})^{s}$ induced by~$U\cup\{w\}\setminus\{x_{0}\}$ is Configuration~\ref{fig:p5x1-unlabeled} or~\ref{fig:(p4+p1)*1-simplified}, violating A\ref{assumption:minimality}.
  Thus,~$G^{s}$ cannot contain Configuration~\ref{fig:(p4+p1)*1-simplified}.

  \bigskip
  For Configuration~\ref{fig:ddag+2e-unlabeled} on six vertices, we use the labels in Figure~\ref{fig:sun+2}.
  We show that
  \[
    v_{0}\in N_{G}(s), \text{ and both~$v_{0} u_{1}$ and~$v_{0} u_{2}$ are collateral.}
  \]
  Suppose that~$v_{0}\not\in N_{G}(s)$. 
  Since~$u_{1}v_{1} v_{2} u_{2}$ is an induced path of~$G$, both~$v_{0} u_{1}$ and~$v_{0} u_{2}$ are collateral.
  In~$G[U \cup \{s\}]^{s}$, edge~$v_{i} v_{j}$ with~$i, j \in \{0, 1, 2\}$ is witnessed by~$u_{3- i - j}$, while edges~$u_{1} v_{1}$ and~$u_{2} v_{2}$ are absent without witnesses.
  Thus, the subgraph induced by~$U$ is Configuration~\ref{fig:sun-unlabeled}, violating A\ref{assumption:minimality}.
We end up with the same contradiction if neither~$v_{0} u_{1}$ nor~$v_{0} u_{2}$ is collateral.
  Hence, we may assume without loss of generality that~$v_{0} u_{2}$ is collateral.
  We argue that~$v_{0} u_{1}$ is collateral as well.
  Suppose otherwise, the vertices~$u_{0}$ and~$u_{1}$ witness~$v_{1} v_{2}$ and~$v_{0} v_{2}$, respectively.
  Let~$w$ be a witness of~$v_{0} u_{2}$.
  Note that~$w$ cannot witness~$u_{1} v_{1}$; otherwise,~$w u_{1} v_{1} v_{0} u_{2}$ is a hole of~$G$.
  Thus,~$G[U\cup\{s, w\}]^{s}$ contains all the edges in~$G^{s}[U]$ except~$u_{1} v_{1}$ and possibly~$u_{2} v_{2}$ and~$v_{0} v_{1}$.
  There is a hole~$v_{0} u_{2} v_{1} v_{2}$ if both are missing; $U$ induces a sun if only~$u_{2} v_{2}$ is missing; and~$U$ induces Configuration~\ref{fig:p5x1-unlabeled} if only~$v_{0} v_{1}$ is missing.  All of them violate A\ref{assumption:minimality} because~$U\cup\{s, w\}\subsetneq V(G)$.
  
  For~$i = 1, 2$, let~$w_{i}$ be a witness of the clique~$\{v_{0}, u_{i}, v_{3 - i}\}$, which is witnessed because of Proposition~\ref{lem:unwitnessed}.  We argue that
  \[
    V(G) = U\cup\{s, w_{1}, w_{2}\}, \text{ and~$w_{i}, i= 1, 2,$ witnesses the clique~$\{u_{i}, v_{0}, v_{1}, v_{2}\}$.}
  \]
Note that~$G[U\cup\{s, w_{1}, w_{2}\}]^{s}$ contains all the edges in~$G^{s}[U]$ except possibly~$u_{1} v_{1}$ and~$u_{2} v_{2}$.  If any of them is missing, then~$U$ induces a sun or Configuration~\ref{fig:ddag+e-unlabeled}.
  Thus, $V(G) = U\cup\{s, w_{1}, w_{2}\}$ by Assumption (minimality).
  Since~$w_{i}, i = 1, 2$, is not adjacent to~$u_{3-i}$ (Proposition~\ref{lem:P3} applied to the path $u_{1} v_{0} u_{2}$), it cannot witness~$u_{3-i} v_{3-i}$.
  Thus, the only possible witness of~$u_{i} v_{i}$ has to be~$w_{i}$.

  Since~$N_{G}(u_{0}) = \{v_{0}\}$, the graph~$(G - s)^{u_{0}}$ is well defined.
  In~$(G - s)^{u_{0}}$, the path~$u_{1} v_{1} v_{2} u_{2}$ is induced; the collateral edges~$v_{0} u_{1}$ and~$v_{0} u_{2}$ are witnessed by~$w_{1}$ and~$w_{2}$, respectively; while~$v_{0} v_{1}$ and~$v_{0} v_{2}$ are absent without witnesses.
  Thus,~$v_{0} u_{1} v_{1} v_{2} u_{2}$ is a hole in~$(G - s)^{u_{0}}$, violating A\ref{assumption:minimality}.
  Thus,~$G^{s}$ cannot contain Configuration~\ref{fig:ddag+2e-unlabeled} on six vertices.
\end{proof}
We are left with Configurations~\ref{fig:dag+2e-simplified},~\ref{fig:dag+e-simplified}, and~\ref{fig:dag-unlabeled}.
We start with some special cases of them in the first row of Figure~\ref{fig:group-3}.
In particular, Configurations~\ref{fig:net-labeled} and~\ref{fig:dag+2e-small} are the smallest incarnations of Configuration~\ref{fig:dag+2e-simplified} and~\ref{fig:dag-unlabeled}, respectively.
Configurations~\ref{fig:dag+e-small} and~\ref{fig:dag-3} are Configurations~\ref{fig:dag+e-simplified} with six vertices and~\ref{fig:dag-unlabeled} with seven vertices, respectively, where~$v\in N_{G}(s)$ and~$x_{2}\in V(G)\setminus N_{G}[s]$.
Figure~\ref{fig:long-dag-2} is Configuration~\ref{fig:dag-unlabeled} with seven or more vertices and~$v\in V(G)\setminus N_{G}(s)$.

\begin{figure}[ht]
  \centering \small
  \begin{subfigure}[b]{0.17\linewidth}
    \centering
    \begin{tikzpicture}[scale=.75]
      \draw (1, 0) -- (4, 0);
      \node[empty vertex, "$v_{0}$" right] (a) at (2.5, 1) {};
      \draw (a) -- ++(0, .75) node[b-vertex, "$x_{0}$" right] {};

      \foreach[count=\i] \t/\l in {b-/x_{1}, empty /v_{1}, empty /v_{2}, b-/x_{2}} {
        \node[\t vertex, "$\l$" below] (u\i) at (\i, 0) {};
      }
      \foreach \i in {2,3}
      \draw (a) -- (u\i);
    \end{tikzpicture}
    \caption{}
    \label{fig:net-labeled}
  \end{subfigure}
  \begin{subfigure}[b]{0.14\linewidth}
    \centering
    \begin{tikzpicture}[scale=.75]
      \draw (1, 0) -- (3, 0);
      \node[a-vertex, "$v$" right] (a) at (2, 1) {};
      \draw (a) -- ++(0, .75) node[b-vertex, "$u_{0}$" right] {};

      \foreach[count=\i] \t/\l in {b-/u_{1}, empty /x, b-/u_{2}} {
        \draw (a) -- (\i, 0) node[\t vertex, "$\l$" below] (u\i) {};
      }
      \foreach \i in {1, 3}
      \draw[witnessed edge] (a) -- (u\i);
    \end{tikzpicture}
    \caption{}
    \label{fig:dag+2e-small}
  \end{subfigure}
  \begin{subfigure}[b]{0.2\linewidth}
    \centering
    \begin{tikzpicture}[scale=.75]
      \node[a-vertex, "$v$" right] (v) at (3., 1) {};
      \draw (v) -- ++(0, .75) node[b-vertex, "$u_{0}$" right] {};
      \draw (1, 0) -- (2, 0) (3, 0) -- (4, 0);
      \foreach[count=\i] \t/\v in {b-/u_{1}, a-/x_{1}, b-/x_{2}, b-/u_{2}} {
        \node[\t vertex, "$\v$" below] (u\i) at (\i, 0) {};
      }
\draw (v) -- (u2);
      \draw[witnessed edge] (v) -- (u4);      
\uncertain{u3}{u2};
      \uncertain{v}{u3};
    \end{tikzpicture}
    \caption{}
    \label{fig:dag+e-small}
  \end{subfigure}
  \begin{subfigure}[b]{0.22\linewidth}
    \centering
    \begin{tikzpicture}[scale=.75, xscale=.95]
      \draw (1, 0) -- (5, 0);
      \node[a-vertex, "$v$" right] (a) at (3, 1) {};
      \draw (a) -- ++(0, .75) node[b-vertex, "$u_{0}$" right] {};

\foreach[count=\i] \t/\l in {b-/u_{1}, empty /x_{1}, b-/x_{2}, empty /x_{3}, b-/u_{2}} {      
        \node[\t vertex, "$\l$" below] (u\i) at (\i, 0) {};
      }
      \foreach \i in {2, 4}
      \draw (a) -- (u\i);
\uncertain{a}{u3};
    \end{tikzpicture}
    \caption{}
    \label{fig:dag-3}
  \end{subfigure}
  \begin{subfigure}[b]{0.22\linewidth}
    \centering
    \begin{tikzpicture}[scale=.75, xscale=.95]
      \draw (1, 0) -- (3, 0) (4, 0) -- (5, 0);
      \draw[dashed] (4, 0) -- (3, 0);
      \node[b-vertex, "$v$" right] (a) at (3, 1) {};
      \draw (a) -- ++(0, .75) node[b-vertex, "$u_{0}$" right] {};

      \foreach[count=\i] \t/\l in {b-/u_{1}, empty /x_{1}, empty /x_{2}, empty /x_{p}, b-/u_{2}} {
        \node[\t vertex, "$\l$" below] (u\i) at (\i, 0) {};
      }
      \foreach \i in {2, 3, 4}
      \draw (a) -- (u\i);
    \end{tikzpicture}
    \caption{$p \ge 3$}
    \label{fig:long-dag-2}
  \end{subfigure}

  \begin{subfigure}[b]{0.17\linewidth}
    \centering
    \begin{tikzpicture}[yscale=.8]
      \draw (3, 0) -- (4, 0);
      \node[a-vertex, "$v_{0}$" left] (a) at (2.5, 1) {};
      \draw (a) -- ++(0, .75) node[b-vertex, "$x_{0}$" right] (x0) {};

      \foreach[count=\i from 2] \t/\l in {b-/v_{1}, a-/v_{2}, b-/x_{2}} {
        \node[\t vertex, "$\l$" below] (u\i) at (\i, 0) {};
      }
      \foreach \i in {2,3}
      \draw (a) -- (u\i);
      \draw (a) -- ++ (.75, 0) node[b-vertex, "$w_{0}$" right] (w0) {};
      \draw (x0) -- (w0) -- (u3);
      \draw[witnessed edge] (u2) -- (u3);
    \end{tikzpicture}
    \caption{}
    \label{fig:lem:net-1}
  \end{subfigure}
  \begin{subfigure}[b]{0.2\linewidth}
    \centering
    \begin{tikzpicture}[scale=.8, every node/.style={empty vertex}]
      \node["$x_{2}$" below] (c) at (0, 0.2) {};
      \foreach \i in {0, 1, 2} 
      \draw ({120*\i-30}:1.75) -- ({120*\i-30}:1) -- ({120*\i+90}:1);
      \foreach[count=\i from 0] \p/\l in {right/x_{1}, left/x_{0}, right/w} {
        \node["$\l$" \p] (u\i) at ({120*\i+90}:1.75) {};
        \node["$v_{\i}$" \p] (v\i) at ({120*\i+90}:1) {};
        \draw (c) -- (v\i);
      }
    \end{tikzpicture}
    \caption{}
    \label{fig:net-0}
  \end{subfigure}
  \begin{subfigure}[b]{0.25\linewidth}
    \centering
    \begin{tikzpicture}[yscale=.8]
      \draw (2, 0) -- (4, 0);
\draw (2.5, 1) node[a-vertex] (a) {} ++ (-.6em, .7em) node {$v_{0}$};
      \draw (a) -- ++(0, .75) node[b-vertex, "$x_{0}$" right] (x0) {};
      \foreach[count=\i from 2] \t/\l in {b-/v_{1}, b-/v_{2}, b-/x_{2}} {
        \node[\t vertex, "$\l$" below] (u\i) at (\i, 0) {};
      }
      \foreach \i/\p in {1/right, 2/left}
      \draw (a) -- ++ ({(1.5 - \i)*1.5}, -.25) node[b-vertex, "$w_{\i}$" \p] (w\i) {} -- (u\inteval{4-\i});
      \foreach \i in {2,3}
      \draw[witnessed edge] (a) -- (u\i);
      \draw (x0) -- (w1);
\end{tikzpicture}
    \caption{}
    \label{fig:lem:net-2}
  \end{subfigure}
  \begin{subfigure}[b]{0.25\linewidth}
    \centering
    \begin{tikzpicture}[yscale=.8]
      \draw (2, 0) -- (3, 0);
      \draw (2.5, 1) node[a-vertex] (a) {} ++ (-.6em, .7em) node {$v_{0}$};
      \draw (a) -- ++(0, .75) node[b-vertex, "$x_{0}$" right] (x0) {};
      \foreach[count=\i from 2] \t/\l in {a-/v_{1}, b-/v_{2}} {
        \node[\t vertex, "$\l$" below] (u\i) at (\i, 0) {};
      }
      \foreach \i/\p in {1/right, 2/left}
      \draw (a) -- ++ ({(1.5 - \i)*1.5}, -.25) node[b-vertex, "$w_{\i}$" \p] (w\i) {} -- (u\inteval{4-\i});
      \foreach \i in {2,3}
      \draw (a) -- (u\i);
      \node[b-vertex, "$w_{0}$" below] (w0) at (2.5, -.5) {};
      \draw[witnessed edge] (w0) -- (a) -- (u3);
      \draw[witnessed edge] (u2) -- (u3);
      \draw (u2) -- (w0) -- (u3);
      \draw[bend left] (w0) -- (w2);
    \end{tikzpicture}
    \caption{}
    \label{fig:lem:net-3}
  \end{subfigure}
  \caption{Illustrations for Lemma~\ref{lem:net}.}
  \label{fig:group-3}
\end{figure}

\begin{lemma}\label{lem:net}
  The graph~$G^{s}$ cannot contain any configuration in Figure~\ref{fig:net-labeled}--\ref{fig:long-dag-2}.
\end{lemma}
\begin{proof}Configuration~\ref{fig:net-labeled}.
  By Corollary~\ref{lem:P4},
  \[
    \text{$\{v_{0}, v_{1}, v_{2}\}$ is a clique of~$G$.}
  \]
  We argue that 
\begin{equation}
    \label{eq:25}
    \text{for all~$i = 0, 1, 2$, the edge~$v_{i} x_{i}$ is not collateral.}
  \end{equation}
  By symmetry, it suffices to consider~$i = 0$.
  Suppose for contradiction that~$v_{0} x_{0}$ is collateral, i.e.,~$v_{0}\in N_{G}(s)\cap N_{G}(x_{0})$. 
  Let~$w_{0}$ be a witness of the edge~$v_{0} x_{0}$.
  By Proposition~\ref{lem:nonadjacent->dominate},~$N_{G^{s}}(w_{0}) \cap U \subseteq \{x_{0}, v_{0}, v_{1}, v_{2}\}$.
  By A\ref{assumption:replacable}, they cannot be equal.
  If~$N_{G^{s}}(w_{0}) \cap U = \{x_{0}, v_{0}\}$, then~$(G - x_{0})^{s}$ contains an induced net (replacing~$x_{0}$ with~$w_{0}$ in~$U$).
  Assume without loss of generality that~$N_{G^{s}}(w_{0}) \cap U = \{x_{0}, v_{0}, v_{2}\}$; see Figure~\ref{fig:lem:net-1}.
  By Proposition~\ref{lem:P3},~$v_{2}\in N_{G}(s)$.
  If~$v_{1}\not\in N_{G}(s)$, then~$(G - x_{1})^{s}$ contains Configuration~\ref{fig:whipping-top-1-unlabeled} ($U\cup\{w_{0}\}\setminus \{x_{1}\}$).
  Hence,~$v_{1}\in N_{G}(s)$.
  Since~$|V(G)| > 7$, the clique~$\{v_{0}, v_{1}, v_{2}\}$ has a witness by Lemma~\ref{lem:witness-N[s]}; let it be~$w_{3}$.
  By Proposition~\ref{lem:nonadjacent->dominate},~$N_{G}(w_{3})\cap U = \emptyset$.
  By assumption,~$G$ is connected, and~$w_{3}$ has a neighbor, which is a witness~$w_{j}$ of~$v_{j} x_{j}, j = 0 ,1, 2$ by Assumption (minimality).
  Then~$N_{G^{s}}(w_{j}) = N_{G^{s}}(v_{j})$ by Proposition~\ref{lem:nonadjacent->dominate}, violating A\ref{assumption:replacable}.

  By A\ref{assumption:N(s)},~$\{v_{0}, v_{1}, v_{2}\}\cap N_{G}(s)\ne \emptyset$.
  We may assume that
  \[
    v_{0} \in N_{G}(s).
  \]
  We show that
  \[
\text{$\{v_{0}, v_{1}, v_{2}\}$ cannot be witnessed.}
  \]
  Suppose that~$\{v_{0}, v_{1}, v_{2}\}$ has a witness~$w$.
  By~\eqref{eq:25} and the minimality,~$V(G) = U\cup \{s, w\}.$
  By Proposition~\ref{lem:nonadjacent->dominate},~$w$ is not adjacent to~$x_{1}$ or~$x_{2}$.
  By A\ref{assumption:replacable}, it is not adjacent to~$x_{0}$ either; otherwise,~$N_{G^{s}}(w)\cap U = N_{G^{s}}(v_{0})\cap U$.
If~$\{v_{0}, v_{1}, v_{2}\}\subseteq N_{G}(s)$, then~$w$ is isolated; if~$\{v_{1}, v_{2}\}\cap N_{G}(s) = \emptyset$, then~$x_{0}$ is isolated.  Both violate Assumption (connectivity).
  Hence, assume without loss of generality that~$v_{1}$ is in~$N_{G}(s)$ and~$v_{2}$ is not.
  Then~$G-s$, shown in Figure~\ref{fig:net-0}, is isomorphic to~$\overline{S_{3}^{+}}$, violating Assumption (minimality). 

  If~$U\cap N_{G}(s) = \{v_{0}, v_{1}, v_{2}\}$, then~$G$ is isomorphic to~$\overline{S_{3}^{+}}$ by Lemma~\ref{lem:witness-N[s]}.  But~$|V(G)| > |U\cup \{s\}| = 7$ because the edge~$v_{0} v_{1}$ has a witness that is not in~$U\cup \{s\}$.
  If~$v_{0}$ is the only vertex in~$U\cap N_{G}(s)$, then for~$i = 1, 2$, we take a witness~$w_{i}$ of edge~$v_{0} v_{3-i}$; see Figure~\ref{fig:lem:net-2}.
   Since~$\{v_{0}, v_{1}, v_{2}\}$ is not witnessed,~$w_{1} \ne w_{2}$ and they are not adjacent by Proposition~\ref{lem:P3}.
   By Assumption (minimality),~$V(G) = U\cup \{s, w_{1}, w_{2}\}$.
   Since~$G$ is connected,~$x_{0}$ is adjacent to at least one of~$w_{1}$ and~$w_{2}$.
   If~$x_{0}$ is adjacent to both~$w_{1}$ and~$w_{2}$, then~$x_{0} w_{1} v_{2} v_{1} w_{2}$ is a hole of~$G$.
   Hence, assume without loss of generality that~$x_{0}$ is adjacent to~$w_{1}$ but not~$w_{2}$.
   Then~$(G - x_{1})^{s}$ contains an induced whipping top ($V(G)\setminus \{s, x_{1}\}$).

   Hence,~$|U\cap N_{G}(s) | = 2$.
   Assume without loss of generality that~$U\cap N_{G}(s) = \{v_{0}, v_{1}\}$.
   For~$i = 0, 1, 2$, we take a witness~$w_{i}$ of~$\{v_{0}, v_{1}, v_{2}\}\setminus \{v_{i}\}$.
   They are distinct because~$\{v_{0}, v_{1}, v_{2}\}$ is not witnessed.
   Moreover,~$w_{0}$ and~$w_{1}$ are not adjacent (Proposition~\ref{lem:P3}), and~$w_{2}$ is not adjacent to~$v_{2}$.
   See Figure~\ref{fig:lem:net-3}.
   By Assumption (minimality),
   \[
     V(G) = U\cup \{s, w_{0}, w_{1}, w_{2}\}.
   \]
   For~$i = 0, 1$, if the edge~$v_{i} w_{i}$ is present in~$G^{s}$, then it must be collateral and witnessed by~$w_{2}$.
   They cannot be both present; otherwise,~$w_{0} w_{2} w_{1} v_{2} $ is a hole of~$G$.
   Assume without loss of generality that~$w_{1}$ and~$w_{2}$ are not adjacent.
Then~$(G - x_{0})^{s}$ contains Configuration~\ref{fig:sun-unlabeled} ($\{v_{0}, v_{1}, v_{2}, w_{0}, w_{1}, w_{2}\}$, when~$w_{0}$ and~$w_{2}$ are not adjacent), or Configuration~\ref{fig:(p4+p1)*1-simplified} ($\{x_{0}, v_{0}, w_{2}, w_{0}, v_{2}, w_{1}\}$, when~$w_{0}$ and~$w_{2}$ are adjacent).
   Thus,~$G^{s}$ cannot contain Configuration~\ref{fig:net-labeled}.

\bigskip
   Configuration~\ref{fig:dag+2e-small}.
  For~$i = 1, 2$, let~$w_{i}$ be a witness of~$v u_{i}$.
  By Proposition~\ref{lem:P3},
  \[
    \text{there is no edge between~$\{u_{1}, w_{1}\}$ and~$\{u_{2}, w_{2}\}$.}
  \]
  We argue that
  \begin{equation}
    \label{eq:24}
    x\in N_{G}(s).\end{equation}
  Suppose for contradiction that~$x\not\in N_{G}(s)$.
  If~$v x$ is not collateral, then~$x$ is a witness of both~$v u_{1}$ and~$v u_{2}$, violating Proposition~\ref{lem:P3}.
  Hence,~$v x$ is collateral; let~$w$ be a witness of the edge~$v x$.
  By Assumption (minimality),~$V(G) = U\cup\{s, w, w_{1}, w_{2}\}$.
  If~$w\not\in\{w_{1}, w_{2}\}$, then~$(G - w)^{s}$, where~$v x$ is absent, contains a hole~$v u_{1} x u_{2}$.
  Assume without loss of generality that~$w = w_{1}$.
  Since~$G$ is connected,~$u_{0}$ is adjacent to at least one of~$w_{1}$ and~$w_{2}$.
  On the other hand, if~$u_{0}$ is adjacent to both~$w_{1}$ and~$w_{2}$, then~$u_{0} w_{1} x w_{2}$ or~$u_{0} w_{1} x u_{2}w_{2}$ is a hole of~$G$.
  Hence,~$u_{0}$ is adjacent to either~$w_{1}$ or~$w_{2}$.
  If~$w_{2}$ is not adjacent to~$x$, then~$(G - u_{1})^{s}$ contains Configuration~\ref{fig:p5x1-unlabeled} ($V(G)\setminus \{s, u_{1}\}$).
  Hence,~$w_{2}$ is adjacent to~$x$.
  If~$u_{0}$ is adjacent to~$w_{1}$, then~$(G - u_{2})^{s}$ contains Configuration~\ref{fig:ddag+e-unlabeled} ($V(G)\setminus \{s, u_{2}\}$); and it is symmetric if~$u_{0}$ is adjacent to~$w_{2}$.

  Second, we argue that~$w_{1}, w_{2}\not\in N_{G}(x)$, i.e., 
  \begin{equation}
    \label{eq:8}
\text{For~$i = 1, 2$, the vertex~$w_{i}$ witnesses~$\{v, u_{i}, x\}$.}    
  \end{equation}
  It suffices to show that both~$\{v, u_{1}, x\}$ and~$\{v, u_{2}, x\}$ are witnessed, and then~\eqref{eq:8} follows from Assumption (minimality).
  For~$i = 1, 2$, let~$w'_{i}$ be a witness of~$u_{i} x$ if it is collateral; note that~$u_{i} x$ is collateral by~Proposition~\ref{lem:unwitnessed} if~$\{v, u_{i}, x\}$ is not witnessed.
  If neither~$\{v, u_{1}, x\}$ nor~$\{v, u_{2}, x\}$ is witnessed, then~$G[U\cup \{w_{1}, w'_{1}, w_{2}, w'_{2}\})^{s}$ contains a hole~$v u_{1} x u_{2}$, where~$v x$ is absent without a witness.
Hence, assume without loss of generality that~$\{v, u_{1}, x\}$ is witnessed and~$\{v, u_{2}, x\}$ is not.
  Then~$V(G) = U\cup\{s, w_{1}, w_{2}, w'_{2}\}$ by Assumption (minimality).
  By Proposition~\ref{lem:P3} (considering paths~$v u_{2} x$ and~$u_{1} v u_{2}$), the witnesses~$w_{1}$, $w_{2}$, and~$w'_{2}$ are different and there is no edge among them.
  Thus, neither~$v w'_{2}$ nor~$x w_{2}$ is present in~$G^{s}$, and~$(G - u_{0})^{s}$ contains Configuration~\ref{fig:sun-unlabeled} ($V(G)\setminus \{s, u_{1}, u_{0}\}$), violating A\ref{assumption:minimality}.
  
  Hence,~$V(G) = U\cup\{s, w_{1}, w_{2}\}.$
  By~Proposition~\ref{lem:nonadjacent->dominate} and~\eqref{eq:24},~$u_{0}$ is adjacent to neither~$w_{1}$ nor~$w_{2}$.
  Thus,~$N_{G}(u_{0}) = \{x\}$.
  The graph~$(G -s)^{u_{0}}$ contains a hole~$v u_{1} x u_{2}$: for~$i = 1, 2$, the edge~$v u_{i}$ is not collateral, and the edge~$x u_{i}$ is witnessed by~$w_{i}$; the edge~$v x$ is absent without a witness.
  Thus,~$G^{s}$ cannot contain Configuration~\ref{fig:dag+2e-small}.

  \bigskip
  Configuration~\ref{fig:dag+e-small}.
  By Corollary~\ref{lem:P4}, the edge~$x_{1} x_{2}$ is collateral.
  Let~$w_{1}$ and~$w_{2}$ be witnesses of~$v x_{1}$ and~$v u_{2}$, respectively.
  We argue that
  \begin{equation}
    \label{eq:18}
    \text{$w_{1}$ is adjacent to~$x_{2}$; i.e.,~$w_{1}$ witnesses the clique~$\{v, x_{1}, x_{2}\}$.}    
  \end{equation}
If~$v x_{2}$ is not collateral, then~$\{v, x_{1}, x_{2}\}$ is witnessed by Proposition~\ref{lem:unwitnessed}, and the statement follows from Assumption (minimality).   
  Now suppose that~$v x_{2}$ is collateral.
  By Proposition~\ref{lem:nonadjacent->dominate},~$N_{G^{s}}(w_{1})\cap U \subseteq \{v, x_{1}, x_{2}\}$.
  If~$w_{1}$ is not adjacent to~$x_{2}$, then~$(G - u_{2})^{s}$ contains Configuration~\ref{fig:ab-wheel-simplified} ($U\cup \{w_{1}\}\setminus \{u_{2}\}$), violating A\ref{assumption:minimality}.

  By Assumption (minimality),~$V(G) = U\cup\{s, w_{1}, w_{2}\}$.
  If~$w_{2}\not\in U$, then~$(G-w_{2})^{s}[U]$ is isomorphic to the net.
  Hence, $w_{2}\in U$ and it is~$x_{2}$.
  Then~$w_{1} x_{2} u_{2} v u_{1}$ is a path of length four in~$G$, and~$x_{1}$ is adjacent to $x_{2}$, $u_{2}$, $v$, and~$u_{0}$.  The subgraph~$G - s$ is isomorphic to~$\otimes(1, 3)$, violating the minimality.
  Thus,~$G^{s}$ cannot contain Configuration~\ref{fig:dag+e-small}.

  \bigskip
  Configuration~\ref{fig:dag-3}.
  We argue that \[
      x_{1}\in N_{G}(s)\cap N_{G}(x_{2})\setminus N_{G}(u_{1}),
       \text{i.e., the edge~$x_{1} x_{2}$ is collateral and~$u_{1} x_{1}$ is not.}
    \]
By Corollary~\ref{lem:P4}, the edge~$v x_{1}$ is collateral.
    If~$x_{1}\not\in N_{G}(s)$, then~$(G - u_{1})^{s}$ contains Configuration~\ref{fig:dag+e-simplified} ($U\setminus \{u_{1}\}$, when~$x_{3} \in N_{G}(s)$) or Configuration~\ref{fig:dag+2e-simplified} ($U\setminus \{u_{1}, u_{2}\}$, otherwise).
    Hence~$x_{1}\in N_{G}(s)$, and~$x_{3}\in N_{G}(s)$ by symmetry.
    By Corollary~\ref{lem:P4},~$x_{1}\in N_{G}(x_{2})$.
    Suppose that~$u_{1} x_{1}$ is collateral and let~$w$ be its witness.
  By Proposition~\ref{lem:P3},~$w$ is not adjacent to~$x_{2}$.
  By Proposition~\ref{lem:nonadjacent->dominate},~$N_{G^{s}}(w)\cap U\subseteq \{v, u_{1}, x_{1}\}$.
  Then~$(G - u_{1})^{s}$ contains an annotated copy of the same configuration (when~$v \not\in N_{G^{s}}(w)$) or Configuration~\ref{fig:dag+e-simplified} (when~$v \in N_{G^{s}}(w)$), both with the vertex set~$U\cup\{w\}\setminus \{u_{1}\}$.

  Let~$w_{1}$ be a witness of~$v x_{1}$. With the same argument as \eqref{eq:18},~$w_{1}$ witnesses the clique~$\{v, x_{1}, x_{2}\}$.

  By symmetry, $x_{3}\in N_{G}(s)$, the edge~$x_{2} x_{3}$ is collateral and~$u_{2} x_{3}$ is not, and the clique~$\{v, x_{2}, x_{3}\}$ is witnessed; let~$w_{2}$ be its witness.
  By Assumption (minimality),~$V(G) = U\cup \{s, w_{1}, w_{2}\}$.
  By Proposition~\ref{lem:P3}, vertices~$w_{1}$ and~$w_{2}$ are not adjacent.
In~$G$, the neighborhoods of~$w_{1}$,~$w_{2}$,~$u_{1}$,~$u_{2}$, and~$u_{0}$ are~$\{x_{2}, x_{3}\}$,~$\{x_{1}, x_{2}\}$,~$\{v, x_{3}\}$,~$\{v, x_{1}\}$, and~$\{x_{1}, x_{3}\}$, respectively.
  With the possible exception of~$v x_{2}$, all the edges among~$\{v, x_{1}, x_{2}, x_{3}\}$ are present in~$G$.
  Depending on whether~$v x_{2}$ is present, $G - s$ is isomorphic to either~$\otimes(1, 1, 1, 1)$ or the graph in~Figure~\ref{fig:the-weird}.  They violate the minimality, and hence~$G^{s}$ cannot contain Configuration~\ref{fig:dag-3}.

  \bigskip
  Configuration~\ref{fig:long-dag-2}.
  By A\ref{assumption:N(s)},~$N_{G}(s)\subseteq \{x_{1}, x_{2}, \ldots, x_{p}\}$, and all vertices in~$V(G)\setminus U$ are witnesses of collateral edges in~$G[U]$.
  By Proposition~\ref{lem:nonadjacent->dominate},~$N_{G}(s)\subseteq N_{G}(v)$, and~$u_{0}$ is not adjacent to any vertex in~$V(G)\setminus U$.  Thus,~$N_{G}(u_{0}) = N_{G}(s)\cup \{v\}$,
  and~$u_{0}$ is simplicial in~$G$.
  Let~$P$ denote the path~$u_{1} x_{1} \cdots x_{p} u_{2}$ of~$G^{s}$.  For convenience, we may use~$x_{p+1}$ to refer to~$u_{2}$. 
  For each collateral edge on~$P$, take a witness; let~$W$ denote this set of witnesses.  We consider the graph
  \[
    H = G[U\cup W]^{u_{0}}.
  \]
  Note that~$G[U\cup W]$ is a circular-arc graph because it does not contains~$s$.
  In~$H$, both edges~$v u_{1}$ and~$v u_{2}$ are present and non-collateral; and~$P$ remains a path: an edge on~$P$ is collateral if and only if it is collateral in~$G^{s}$, and the same vertex witnesses them in both graphs.
  By A\ref{assumption:minimality} (note that~$s\not\in V(H)$), the graph~$H$ is chordal, and hence
\[
    v x_{i}\in E(H), i = 1, \ldots, p.
  \]
  By Lemma~\ref{thm:forbidden-configurations-unrestricted},~$H[\{v, u_{1}, x_{1}, x_{2}, x_{3}, x_{4}\}]$ is not Configuration~\ref{fig:p5x1-unlabeled unrestricted}, and hence
  \[
    x_{2}\in N_G(s)\subsetneq N_G(u_{0}).
  \]
  Let~$w$ be a witness of the collateral edge~$v x_{2}$ in~$H$.  It cannot be from~$U$; hence~$w\in W$.
  By the selection of~$W$, the vertex~$w$ must be a witness of the edge~$x_{1} x_{2}$ or~$x_{2} x_{3}$ in~$G^{s}$.
  Since~$w$ is a witness of~$v x_{2}$ in~$H$, it is not adjacent to~$v$ in~$G$, which further implies that
  \[
    v w \not\in E(G^{s}).
  \]
  By Propositions~\ref{lem:nonadjacent->dominate} and~\ref{lem:P3}, at most one of~$x_{1}$ and~$x_{3}$ is in~$N_{G^{s}}(w)$.
If~$w$ is a witness of the edge~$x_{1} x_{2}$, then~$(G - u_{1})^{s}$ contains Configuration~\ref{fig:dag-unlabeled}. of order~$p + 3$ ($U\cup\{w\}\setminus \{u_{1}, x_{1}\}$).
  Otherwise,~$(G - u_{2})^{s}$ contains an induced net ($\{u_{0}, u_{1}, v, w, x_{1}, x_{2}\}$).  Both violate A\ref{assumption:minimality}.
  Thus,~$G^{s}$ cannot contain Configuration~\ref{fig:long-dag-2}.
\end{proof}

Note that in all the configurations in Figure~\ref{fig:group-4}, $v\in N_{G}(s)$, both~$u_{1} v$ and~$u_{2} v$ are edges of~$G$, and the differences among them are how many of these two edges are in~$G^{s}$, where they are collateral.

\begin{figure}[ht]
  \centering \small
  \begin{subfigure}[b]{0.2\linewidth}
    \centering
    \begin{tikzpicture}[scale=.75]
      \draw (1, 0) -- (2, 0) (3, 0) -- (4, 0);
      \draw[dashed] (2, 0) -- (3, 0);
      \node[a-vertex, "$v$" right] (a) at (2.5, 1) {};
      \draw (a) -- ++(0, .75) node[b-vertex, "$u_{0}$" right] {};

      \foreach[count=\i] \t/\l in {b-/u_{1}, empty /x_{1}, empty /x_{p}, b-/u_{2}} {
        \draw (a) -- (\i, 0) node[\t vertex, "$\l$" below] (u\i) {};
      }
      \foreach \i in {1, 4}
      \draw[witnessed edge] (a) -- (u\i);
    \end{tikzpicture}
    \caption{$p \ge 2$}
    \label{fig:dag+2e}
  \end{subfigure}
  \,
  \begin{subfigure}[b]{0.2\linewidth}
    \centering
    \begin{tikzpicture}[scale=.75]
      \node[a-vertex, "$v$" right] (v) at (3., 1) {};
      \draw (v) -- ++(0, .75) node[b-vertex, "$u_{0}$" right] {};
      \draw (1, 0) -- (2, 0) (3, 0) -- (4, 0);
      \foreach[count=\i] \t/\v in {b-/u_{1}, a-/x_{1}, empty /x_{p}, b-/u_{2}} {
        \node[\t vertex, "$\v$" below] (u\i) at (\i, 0) {};
      }
      \draw[dashed] (u2) -- (u3);
      \foreach \i in {2, 3, 4} \draw (v) -- (u\i);
\draw[witnessed edge] (v) -- (u4) {};
    \end{tikzpicture}
    \caption{$p \ge 2$}
    \label{fig:dag+e}
  \end{subfigure}
  \begin{subfigure}[b]{0.25\linewidth}
    \centering
    \begin{tikzpicture}[scale=.75]
      \draw (1, 0) -- (3, 0) (4, 0) -- (5, 0);
      \draw[dashed] (4, 0) -- (3, 0);
      \node[a-vertex, "$v$" right] (a) at (3, 1) {};
      \draw (a) -- ++(0, .75) node[b-vertex, "$u_{0}$" right] {};

      \foreach[count=\i] \t/\l in {b-/u_{1}, empty /x_{1}, empty /x_{2}, empty /x_{p}, b-/u_{2}} {
        \node[\t vertex, "$\l$" below] (u\i) at (\i, 0) {};
      }
      \foreach \i in {2, 3, 4}
      \draw (a) -- (u\i);
    \end{tikzpicture}
    \caption{$p \ge 3$}
    \label{fig:long-dag-1}
  \end{subfigure}
  \caption{Illustrations for Lemma~\ref{lem:long-dags}.}
  \label{fig:group-4}
\end{figure}

\begin{lemma}\label{lem:long-dags}
  The graph~$G^{s}$ cannot contain Configuration~\ref{fig:dag+2e-simplified},~\ref{fig:dag+e-simplified}, or~\ref{fig:dag-unlabeled}.
\end{lemma}
\begin{proof}
  By Lemma~\ref{lem:net}, if~$G^{s}$ contains Configuration~\ref{fig:dag+2e-simplified}, then it consists of at least six vertices; if~$G^{s}$ contains Configuration~\ref{fig:dag-unlabeled}, then $v\in N_{G}(s)$.
  We use the labels in Figure~\ref{fig:group-4}.
  Let~$P$ denote the path~$u_{1} x_{1} \cdots x_{p} u_{2}$ of~$G^{s}$.  For convenience, we may use~$x_{0}$ and~$x_{p+1}$ to refer to~$u_{1}$ and~$u_{2}$, respectively. 
  
  Note that
  \begin{equation}
    \label{eq:2}
    N_G(u_{0}) \cap U = N_G(s)\setminus\{v\}.
  \end{equation}
  By Assumption (minimality), each vertex in~$V(G) \setminus U$ is a witness of some collateral edge of~$G^{s}[U]$.
  By Proposition~\ref{lem:nonadjacent->dominate}, an edge witnessed by a vertex in~$N_G(u_{0}) \setminus U$ must be between~$v$ and a vertex from~$V(P) \setminus N_G[s]$.
  If two vertices in~$N_G(u_{0})\setminus U$ are not adjacent, we can find a path in~$G^{s}$ connecting them with all internal vertices from~$P$, which forms a hole with~$u_{0}$ in~$G^{s}$. Thus,~$N_G(u_{0})\setminus U$ is a clique.
  By Proposition~\ref{lem:nonadjacent->dominate}, all the edges between~$N_G(u_{0}) \cap U$ and~$N_G(u_{0})\setminus U$ are present in~$G$.
  In summary, we have
  \[
    \text{the vertex~$u_{0}$ is simplicial in~$G$.}
  \]
  
  For each~$i = 0, \ldots, p$, if the edge~$x_{i} x_{i+1}$ is collateral, let~$w_{i}$ be its witness.
  By Proposition~\ref{lem:nonadjacent->dominate},
  
  \begin{equation}
    \label{eq:4}
    w_{i}\not\in N_{G}(u_{0}).
  \end{equation}
Let~$W$ denote these witnesses, and let
 \[
   H = (G[U\cup W\cup\{s\}])^{u_{0}},
 \]
 which is well defined because~$u_{0}$ is simplicial in~$G$, hence also in~$G[U\cup W\cup\{s\}]$.
By~\eqref{eq:2} and~\eqref{eq:4}, all edges on~$P$ remain in~$H$: for each~$i = 0, \ldots, p$, the edge~$x_{i} x_{i+1}$ is collateral in~$G^{s}$ if and only if it is collateral in~$H$, both witnessed by~$w_{i}$.
 Thus,~$P$ remains a path in~$H$.
 On the other hand, since none of~$v$, $u_{1}$, and~$u_{2}$ is in~$N_{G}(u_{0})$, both edges~$u_{1} v$ and~$u_{2} v$ are present in~$H$, as in~$G$.
 We note that
 \begin{equation}
   \label{eq:9}
   \text{$v$ is adjacent to all the vertices on~$P$ in~$H$.}   
 \end{equation}
 Otherwise, there is a hole in~$(G[U\cup W])^{u_{0}}$, violating A\ref{assumption:minimality} because~$U\cup W\subsetneq V(G)$.  (Note that~$N_{H}(s) = \{v\}$.)

 Next, we argue that \begin{equation}
   \label{eq:5}
   x_{i} \in N_G(s), i=1,\ldots,p.
 \end{equation}
 Suppose for contradiction that there is~$i\in \{1, \ldots, p\}$ such that~$x_{i}\not\in N_{G}(s)$.
By \eqref{eq:9},~$x_{i} v\in E(G)$, and hence it is collateral in~$G^{s}$.  
If~$i < p$ in Figure~\ref{fig:dag+2e} or~\ref{fig:dag+e}, then~$(G - u_{1})^{s}$ contains Configuration~\ref{fig:dag+2e-simplified} ($U\setminus \{u_{1}, x_{2}, \ldots, x_{i-1}\}$), and it is symmetric if~$i = p$ in Figure~\ref{fig:dag+2e}.
   If~$i > 2$ in Figure~\ref{fig:dag+e} or~\ref{fig:long-dag-1}, then~$(G - u_{2})^{s}$ contains Configuration~\ref{fig:dag+e-simplified} ($U\setminus \{x_{i+1}, x_{i+2}, \ldots, x_{p}, u_{2}\}$), and it is symmetric if~$i < p - 1$ in Figure~\ref{fig:long-dag-1}.
   The remaining cases are~$i = p = 2$ in Figure~\ref{fig:dag+e} and~$i = p - 1 = 2$ in Figure~\ref{fig:long-dag-1}.  We have seen that they cannot exist in Lemma~\ref{lem:net}: they are Figure~\ref{fig:dag+e-small} and Figure~\ref{fig:dag-3}.
   As a result,
   \[
     \text{for~$i = 1, \ldots, p-1$, the vertex~$w_{i}$ witnesses of~$\{v, x_{i}, x_{i+1}\}$.}
   \]
   Since~$|V(G)| > 7$, the clique~$\{v, x_{i}, x_{i+1}\}$ is witnessed in~$G^{s}$ by Lemma~\ref{lem:witness-N[s]}.
   If its witness is not~$w_{i}$, then~$(G - w_{i})^{s}$ contains the same configuration, violating A\ref{assumption:minimality}.

   We note that~$u_{2} x_{p}$ is collateral and~$w_{p}$ witnesses~$\{v, u_{2}, x_{p}\}$ in Figures~\ref{fig:dag+2e} and~\ref{fig:dag+e}.
   Otherwise, let~$w$ be a witness of the collateral edge~$v u_{2}$.
   By Proposition~\ref{lem:nonadjacent->dominate}, Proposition~\ref{lem:P3}, and the minimality,~$w$ cannot witness any other collateral edge.
   Then~$(G - w)^{s}$ contains either Configuration~\ref{fig:dag+e-simplified} or~\ref{fig:dag-unlabeled}, violating A\ref{assumption:minimality}.
   By symmetry,~$u_{1} x_{1}$ is collateral and~$w_{0}$ witnesses~$\{v, u_{1}, x_{1}\}$ in Figures~\ref{fig:dag+2e}.
By Assumption (minimality),
   \[
     V(G) = U\cup W \cup \{s\}.
   \]
   Note that~$x_{2} x_{3}$ is collateral in all the three graphs in Figure~\ref{fig:group-4} (note that~$p \ge 3$ in Figure~\ref{fig:long-dag-1}).
   Hence,~$N_{G^{s}}(x_{2}) = \{v, x_{1}, x_{3}, w_{1}, w_{2}\}$.
   However, there cannot be a witnesses of the collateral edge~$v x_{2}$ in~$H$, which must be in~$V(G)\setminus N_{G}[x_{2}]\subseteq N_{G^{s}}(x_{2})$, contradicting~\eqref{eq:9}.  This concludes the proof.
\end{proof} 

Lemmas 2.10--14 imply Theorem~\ref{thm:simplified-forbidden-configurations}.

\section{Proof of Theorem~\ref{thm:main}}

We rely on the reader to check the small graphs listed in Theorem~\ref{thm:main}, namely, long claw, whipping top$^\star$, and those in Figure~\ref{fig:chordal-non-cag}, are minimal forbidden induced subgraphs.
It is also easy to check graphs~$\overline{S_{k}^{+}}, k \ge 3$, thanks to the strong symmetries \cite{cao-24-split-cag}.\footnote{A quick argument is as follows.  In~$\overline{S_{k}^{+}}$, every maximal clique is the closed neighborhood of some simplicial vertex.  Thus, if~$\overline{S_{k}^{+}}$ is a circular-arc graph, then it has to be a Helly circular-arc graph, but it is already known that~$\overline{S_{k}}$ is not~\cite{joeris-11-hcag}.  The minimality follows from a similar observation.}
It is not that straightforward to show that all the~$\otimes$ graphs are minimal forbidden induced subgraphs of circular-arc graphs.  
We defer the proof of~Proposition~\ref{lem:o-graphs} to the appendix.

The rest of this section is to prove the sufficiency of Theorem~\ref{thm:main}, i.e., the completeness of this list, for which it suffices to focus on minimal connected chordal graphs that are not circular-arc graphs.
By Theorem~\ref{thm:simplified-forbidden-configurations},~$G^{s}$ contains a forbidden configuration in Figure~\ref{fig:simplified-forbidden-configurations}.
By translating forbidden configurations in~$G^{s}$ back to graphs in~$G$, we show that~$G$ must be one of the graphs listed in Theorem~\ref{thm:main}.
Throughout we use~$U$ to denote the vertex set of an annotated copy of the forbidden configuration under discussion.
We start with holes.

\begin{lemma}\label{lem:hole}
  If~$G^{s}$ is not chordal, then~$G$ is isomorphic to~$\overline{S_{\ell}^{+}}$ with~$\ell\ge 4$ or~$\otimes(a_{0}, a_{1}, \ldots, a_{2 p - 1})$ with~$\sum a_{i} \ge 4$.
\end{lemma}
\begin{proof}
  Let~$U$ be the vertex set of a hole of~$G^{s}$.
  We number vertices in~$U$ such that the hole is~$v_{0} \cdots v_{|U| - 1}$.
  In this proof, indices of vertices in~$U$ are modulo~$|U|$.
  Note that for each vertex~$v_{i}\in U\cap N_{G}(s)$, both edges~$v_{i - 1} v_{i}$ and~$v_{i}v_{i + 1}$ are collateral by Proposition~\ref{lem:nonadjacent->dominate}.
  For~$i = 0, \ldots, |U| - 1$, if~$v_{i} v_{i+1}$ is a collateral edge, let~$w_{i}$ be its witness.  Let~$W$ denote all the witnesses.
  By Assumption (minimality),
  \begin{equation}
    \label{eq:17}
    V(G) = U\cup W\cup \{s\}.
  \end{equation}
  For each witness~$w_{i}$,
  \begin{equation}
    \label{eq:3}
    N_{G^{s}}(w_{i}) = \{v_{i}, v_{i+1}\}.
  \end{equation}
  By Proposition~\ref{lem:nonadjacent->dominate},~$N_{G^{s}}(w_{i}) \cap U\subseteq \{v_{i-1}, v_{i}, v_{i+1}\}$ when~$v_{i}\in N_G(s)$.  They cannot be equal by A\ref{assumption:replacable}.  It is similar if~$v_{i+1}\in N_G(s)$.  Moreover,
  \begin{equation}
    \label{eq:11}
    \text{$W$ is an independent set}.
  \end{equation}

  Suppose for contradiction that there are distinct~$i, j\in\{0, \ldots, |U| - 1\}$ such that~$w_{i} w_{j}\in E(G)$.
  If both~$v_{i}$ and~$v_{i+1}$ are from~$N_G(s)$, then~$\{v_{i}, v_{i+1}\} \subseteq N_{G^{s}}(w_{j})$ by Proposition~\ref{lem:nonadjacent->dominate}.  Since~$\{v_{i}, v_{i+1}\} \ne \{v_{j}, v_{j+1}\}$, we have a contradiction to~\eqref{eq:3}.
  It is similar when both~$v_{j}, v_{j+1}\in N_{G}(s)$.
  Thus, only one of~$\{v_{i}, v_{i+1}\}$ and only one of~$\{v_{j}, v_{j+1}\}$ are in~$N_G(s)$.
  We may assume without loss of generality that~$v_{j}$ is from~$N_G(s)$ and~$v_{j+1}$ is not.
  Note that~$v_{j} \in N_{G^{s}}(w_{i})$ by Proposition~\ref{lem:nonadjacent->dominate} and~$v_{i} = v_{j - 1}$ by~\eqref{eq:3}.
  This contradicts Proposition~\ref{lem:P3}(ii), applied to the path~$v_{j - 1} v_{j} v_{j + 1}$.

  If~$U\subseteq N_G(s)$, then~$G$ is a split graph, with the clique~$U\cup\{s\}$, by~\eqref{eq:17} and~\eqref{eq:11}.  Thus,~$G$ is isomorphic to~$\overline{S_{|U|}^{+}}$~\cite{cao-24-split-cag}.
  In the rest,~$U$ contains vertices from both~$N_G(s)$ (by A\ref{assumption:N(s)}) and~$V(G)\setminus N_G[s]$.
  Note that the number of edges on the hole between~$N_G(s)$ and~$V(G)\setminus N_G[s]$ is even.
  Let it be~$2 p$ for some positive integer~$p$.
  Removing all these~$2 p$ edges from the hole leaves a sequence of (possibly trivial) paths.
  For~$i = 0, 1, \ldots, 2 p-1$, let~$V_{i}$ denote the vertex set of the~$i$th path.
  Note that~$V_{i}$ is a subset of either~$N_G(s)$ or~$V(G)\setminus N_G[s]$, and if~$V_{i}\subseteq N_G(s)$, then~$V_{i+1\pmod{2 p}}\subseteq V(G)\setminus N_G[s]$, and vice versa.
  We may assume without loss of generality that~$V_{0}\subseteq N_G(s)$.  Hence,
  \[
    N_{G}(s) = V_{0}\cup V_{2}\cup \cdots\cup V_{2p - 2}.
  \]  
  For~$i = 0, \ldots, {p} - 1$, all the edges on the hole incident to~$V_{2 i}$ are collateral (Proposition~\ref{lem:nonadjacent->dominate}); let~$W_{2i}$ denote the witnesses of these~$|V_{2i}|+1$ collateral edges. 
  Note that~$W = W_{0}\cup W_{2}\cup \cdots\cup W_{2 p - 2}$.
  Two of the witnesses are adjacent to~$V_{2i-1\pmod{2 p}}$ and~$V_{2i+1}$, and let them be~$w_{2i}^1$ and~$w_{2i}^2$, respectively.
  The subgraph~$G[V_{2i}\cup W_{2i}]$ is the gadget~$D_{|V_{2i}|}$, the ends of which are~$w_{2i}^1$ and~$w_{2i}^2$.
  On the other hand,~$G[V_{2i+1}]$ is the same simple path as~$G^{s}[V_{2i+1}]$.
  The vertices in~$V_{2 i}$ are complete to vertices not in this gadget, while the two ends of~$G[V_{2i+1}]$ are connected to~$w_{2i}^2$ and~$w_{2i+2\pmod{2 p}}^1$, respectively.
  Thus,~$G$ is isomorphic to~$\otimes(|V_{0}|, |V_{1}|, \ldots, |V_{2p - 1}|)$, with~$c = s$.
\end{proof}

As we have characterized all minimal split graphs that are not circular-arc graph (Theorem~\ref{thm:split}), we can stop after~$G$ is determined to be a split graph.

\begin{figure}[ht]
  \centering \small
  \begin{subfigure}[b]{0.2\linewidth}
    \centering
    \begin{tikzpicture}[scale=.5]
      \foreach \i in {1, 2, 3} {
        \draw ({120*\i+90}:2) -- ({120*\i-30}:2);
        \draw ({120*\i-90}:1) -- ({120*\i+30}:1);
      }
      \foreach \i in {1, 2, 3} {
        \pgfmathparse{int(Mod(\i,3))}
        \node[rotate={int(\i/3)*180}, empty vertex, "$u_{\pgfmathresult}$" below] (u\i) at ({90-120*\i}:2) {};
      }
      \foreach \i in {1 , 2 , 3} {
        \pgfmathparse{int(Mod(\i,3))}
        \node[rotate={\i*240}, empty vertex, "$v_{\pgfmathresult}$" below] (v\i) at ({270-120*\i}:1) {};
      }
    \end{tikzpicture}
    \caption{}
    \label{fig:sun}
  \end{subfigure}
  \begin{subfigure}[b]{0.2\linewidth}
    \centering
    \begin{tikzpicture}[scale=.8, every node/.style={empty vertex}]
      \node["$u_{0}$" below] (c) at (0, 0.2) {};
      \foreach \i in {0, 1, 2} 
      \draw ({120*\i-30}:1.75) -- ({120*\i-30}:1) -- ({120*\i+90}:1);
      \foreach[count=\i from 0] \p/\l in {right/s, left/u_{2}, right/u_{1}} {
        \node["$\l$" \p] (u\i) at ({120*\i+90}:1.75) {};
        \node["$v_{\i}$" \p] (v\i) at ({120*\i+90}:1) {};
        \draw (c) -- (v\i);
      }
    \end{tikzpicture}
    \caption{}
    \label{fig:sun-0}
  \end{subfigure}
  \begin{subfigure}[b]{0.2\linewidth}
    \centering
    \begin{tikzpicture}[scale=.5]
      \begin{scope}[yshift=1cm]
        \foreach[count=\i from 0] \t/\v in {a-/v_{0}, b-/w_{1 0}, b-/w_{2 0}} {
          \draw ({120*\i-90}:2) -- ({120*\i+30}:2);
          \coordinate (x\i) at ({120*\i-90}:2) {};
        }
      \end{scope}
      \foreach[count=\i from 0] \t in {a-, b-, b-} {
        \draw ({120*\i+90}:2) -- ({120*\i-30}:2);
        \coordinate (u\i) at ({120*\i+90}:2) {};          
      }
      \foreach \t/\v/\l in {b-/x2/w_{1 0}, a-/u0/u_{0}, b-/x1/w_{2 0}} 
      \node[\t vertex, "$\l$"] at (\v) {};
      \foreach \t/\v/\l in {b-/u1/u_{2}, a-/x0/v_{0}, b-/u2/u_{1}} 
      \node[\t vertex, "$\l$" below] at (\v) {};
      
      \foreach[count=\i] \p in {left, right} 
      \node[b-vertex, "$v_{\i}$" \p] (v\i) at ({270-120*\i}:1) {};
      \draw (v1) -- (v2);
\end{tikzpicture}
    \caption{}
    \label{fig:sun-1}
  \end{subfigure}
  \begin{subfigure}[b]{0.22\linewidth}
    \centering
    \tikzstyle{every node}=[empty vertex]
    \begin{tikzpicture}[scale=.8]
      \begin{scope}[shift={(-1, -1)}]
\draw[step=2] (0, 0) grid (2, 2);
      \end{scope}
      \node[empty vertex, "$s$" right] (u0) at (0.2, 0) {};
      \draw (0, -1) -- (0, 1) -- (u0) --  (0, -1);
      \draw (-1, 0) -- (0, 1) -- (1, 0) -- (0, -1)--cycle;      
      \foreach \y/\p/\l in {-1/below/{u_{2}, u_{0}, u_{1}}, 1/above/{w_{1 0}, v_{0}, w_{2 0}}}
      \foreach[count=\x from -1] \v in \l
\node[empty vertex, "$\v$" \p] at (\x, \y) {};
      \node[empty vertex, "$v_{1}$" left] (v1) at (-1, 0) {};
      \node[empty vertex, "$v_{2}$" right] (v2) at (1, 0) {};
      \draw[bend right] (v1) edge (v2);
    \end{tikzpicture}
    \caption{}
    \label{fig:sun-2}
  \end{subfigure}
\caption{Illustrations for the proof of Lemma~\ref{lem:sun}.}
  \label{fig:lem:sun}
\end{figure}

\begin{lemma}\label{lem:sun}
  If~$G^{s}$ contains Configuration~\ref{fig:sun-simplified}, then~$G$ is isomorphic to~$\overline{S_{3}^{+}}$ or Figure~\ref{fig:the-weird}.
\end{lemma}
\begin{proof}
  We label the sun as in Figure~\ref{fig:sun}.  We claim that
  \[
    \{v_{0}, v_{1}, v_{2}\}\cap N_{G}(s)\ne \emptyset.
  \]
  Suppose for contradiction that~$\{v_{0}, v_{1}, v_{2}\}$ is disjoint from~$N_{G}(s)$.
  By A\ref{assumption:N(s)},~$\{u_{0}, u_{1}, u_{2}\}\cap N_{G}(s)\ne \emptyset$.
  Assume without loss of generality that~$u_{0}\in N_{G}(s)$.
  We note that~$\{v_{1}, v_{2}, u_{0}\}$ is not witnessed; if~$w$ is a witness of~$\{v_{1}, v_{2}, u_{0}\}$, then~$N_{G^{s}}(w) \cap U = \{v_{1}, v_{2}, u_{0}\}$ by Proposition~\ref{lem:nonadjacent->dominate}, violating~A\ref{assumption:replacable}.
  For~$i = 1, 2$, let~$w_{i}$ be a witness of the edge~$u_{0} v_{i}$; note that~$N_{G^{s}}(w_{i})\cap U = \{u_{0}, v_{i}\}$ by Proposition~\ref{lem:nonadjacent->dominate} and the fact that it is not a witness of~$\{v_{1}, v_{2}, u_{0}\}$.
  By Proposition~\ref{lem:P3}(ii),~$w_{1}$ and~$w_{2}$ are not adjacent.
  But then~$(G - \{u_{1}, u_{2}\})^{s}$ contains Configuration~\ref{fig:sun-unlabeled}, with~$u_{1}, u_{2}$ replaced by~$w_{1}, w_{2}$ in~$U$, violating A\ref{assumption:minimality}.

  For distinct~$i, j\in \{0, 1, 2\}$, let~$w_{ij}$ be a witness of the edge~$v_{i} u_{j}$ if it is collateral, and let~$W$ denote the set of witnesses.
  Note that~$U$ and~$W$ are disjoint by Proposition~\ref{lem:nonadjacent->dominate}.
  We claim that
  \begin{equation}
    \label{eq:15}
    V(G) = U\cup W\cup\{s\}.
  \end{equation}
  Let~$H= (G[U\cup W\cup\{s\}])^{s}$.
  By construction,~$H[U]$ is a subgraph of~$G^{s}[U]$, in which all the edges incident to~$u_{i}, i\in \{0, 1, 2\}$, are present.
  By A\ref{assumption:minimality},~$H[U]$ is chordal, which means that~$\{v_{0}, v_{1}, v_{2}\}$ is a clique.
  Thus, \eqref{eq:15} follows from Assumption (minimality).
  As a result,
  \begin{equation}
    \label{eq:32}
    \text{$\{v_{0}, v_{1}, v_{2}\}$ is not witnessed.}    
  \end{equation}
  Note that a witness~$\{v_{0}, v_{1}, v_{2}\}$, if one exists, must be in~$W$.  Suppose it is~$w_{i j}$ for some~$i, j\in \{0, 1, 2\}$.
  Then~$v_{i}\in N_{G}(s)$ and~$u_{j}\not\in N_{G}(s)$.
  It violates either Proposition~\ref{lem:nonadjacent->dominate}, when~$v_{j}\in N_{G}(s)$, or Proposition~\ref{lem:P3} otherwise.

  If~$\{v_{0}, v_{1}, v_{2}\}\subseteq N_{G}(s)$, then by~\eqref{eq:32} and Lemma~\ref{lem:witness-N[s]},~$G$ is isomorphic to~$\overline{S_{3}^{+}}$.
  Henceforth, assume without loss of generality that~$v_{0}\in N_{G}(s)$ and~$v_{2}\not\in N_{G}(s)$.
  By Corollary~\ref{lem:P4}, the set~$\{v_{0}, v_{1}, v_{2}\}$ is always a clique in~$G$.
  Thus, an edge among~$\{v_{0}, v_{1}, v_{2}\}$ is collateral if and only if at least one of its ends is in~$N_{G}[s]$.
  We show that 
  \begin{equation}
    \label{eq:33}
    v_{1}\not\in N_{G}(s).    
  \end{equation}
Suppose for contradiction that~$v_{1}\in N_{G}(s)$.
  Let~$w$ be a witness of~$v_{0} v_{1}$.
  By~\eqref{eq:32},~$w$ is not adjacent to~$v_{2}$.
  By Proposition~\ref{lem:nonadjacent->dominate},~$w$ is not adjacent to~$u_{0}$ or~$u_{1}$ in~$G^{s}$.
  If~$w \ne u_{2}$, then~$(G - u_{2})^{s}$ contains Configuration~\ref{fig:sun-unlabeled} (with~$u_{2}$ replaced by~$w$ in~$U$).
  Thus,~$w = u_{2}$, and neither~$v_{0} u_{2}$ nor~$v_{1} u_{2}$ is collateral.
  By Assumption,~$G$ is connected, and hence~$\{u_{0}, u_{1}\}\cap N_{G}(s)\ne\emptyset$ or~$u_{2}$ is adjacent to a vertex in~$W$.
  \begin{itemize}
  \item 
  If a vertex in~$W$ is adjacent to~$u_{2}$, it must be~$w_{0 1}$ or~$w_{1 0}$ by Proposition~\ref{lem:nonadjacent->dominate}.
Assume without loss of generality that~$w_{1 0}$ exists and is adjacent to~$u_{2}$; then~$u_{0}\not\in N_{G}(s)$ by Proposition~\ref{lem:nonadjacent->dominate}.
  Then~$v_{0} w_{1 0}$ is a collateral edge witnessed by~$u_{2}$.
  By Proposition~\ref{lem:nonadjacent->dominate},~$N_{G^{s}}(w_{1 0})\cap U \subseteq U\setminus \{u_{1}\}$, and they cannot be equal by A\ref{assumption:replacable}.
  Thus,~$w_{1 0}$ and~$v_{2}$ are not adjacent, but then~$v_{0} v_{2} u_{0} w_{1 0}$ is a hole in~$(G - v_{1})^{s}$, violating A\ref{assumption:minimality}.
\item Assume without loss of generality that~$u_{0}\in N_{G}(s)$.
  Let~$w'$ be a witness of the collateral edge~$v_{0} v_{2}$.
  By Proposition~\ref{lem:nonadjacent->dominate},~$w'$ is not adjacent to~$u_{0}$ in~$G^{s}$.  By~\eqref{eq:15},~$w'$ is~$w_{0 1}$ or~$w_{2 1}$.  In either case, $N_{G^{s}}(w')\cap U = \{u_{1}, v_{0}, v_{2}\}$ by Proposition~\ref{lem:nonadjacent->dominate}.  Then~$(G - u_{1})^{s}$ contains Configuration~\ref{fig:sun-unlabeled} (with~$u_{1}$ replaced by~$w'$ in~$U$), violating~A\ref{assumption:minimality}.
  \end{itemize}
  This verifies~\eqref{eq:33}.
    
  Finally, we argue that
  \[
    \text{neither~$v_{0}u_{1}$ nor~$v_{0}u_{2}$ is collateral.}
  \]
  We show~$v_{0}u_{1}$ and the other follows by symmetry.
  Suppose that~$v_{0}u_{1}$ is collateral, witnessed by~$w_{0 1}$.
  By Proposition~\ref{lem:P3}, $w_{0 1}$ is not adjacent to~$v_{1}$.
  By Proposition~\ref{lem:nonadjacent->dominate},~$N_{G^{s}}(w_{0 1})\cap U \subseteq \{v_{0}, v_{2}, u_{1}, u_{2}\}$.
  If~$u_{2}\in N_{G^{s}}(w_{0 1})$, then~$(G - u_{0})^{s}$ contains a hole ($v_{1} v_{2} w_{0 1} u_{2}$ or~$v_{1} v_{2} u_{1} w_{0 1} u_{2}$).
  If~$w_{0 1}$ is adjacent to~$v_{2}$, then~$(G - u_{1})^{s}$ contains Configuration~\ref{fig:sun-unlabeled} (with~$u_{1}$ replaced by~$w_{0 1}$ in~$U$).
  All violate A\ref{assumption:minimality}.
  Otherwise,~$N_{G^{s}}(w_{0 1})\cap U = \{v_{0}, u_{1}\}$, and~$(G - u_{0})^{s}$ contains Configuration~\ref{fig:p5x1-unlabeled unrestricted} ($U\cup\{w_{0 1}\}\setminus \{u_{0}\}$), violating Lemma~\ref{thm:forbidden-configurations-unrestricted}.

  As a result,~$u_{1}, u_{2}\in V(G)\setminus N_{G}[s]$.
  If~$u_{0}\not\in N_{G}(s)$, then~$V(G) = U\cup \{s\}$, and~$G$ is isomorphic to~$\overline{S_{3}^{+}}$; see Figure~\ref{fig:sun-0}.
Hence,~$u_{0}\in N_{G}(s)$, and~$W = \{w_{1 0}, w_{2 0}\}$.
  See Figure~\ref{fig:sun-1}.
  Note that~$w_{1 0}\ne w_{2 0}$; otherwise,~$(G - u_{0})^{s}$ contains Configuration~\ref{fig:sun-unlabeled} (with~$u_{0}$ replaced by~$w_{1 0}$).
  By Proposition~\ref{lem:P3},~$w_{1 0}$ and~$w_{2 0}$ are not adjacent.  By Proposition~\ref{lem:nonadjacent->dominate},~$N_{G^{s}}(w_{i 0}) = \{u_{0}, v_{i}\}, i = 1, 2$.
Hence,~$G$ is isomorphic to the graph in Figure~\ref{fig:the-weird}; see Figure~\ref{fig:sun-2}.
\end{proof}

\begin{figure}[ht]
  \centering \small
  \quad
  \begin{subfigure}[b]{0.22\linewidth}
    \centering
    \begin{tikzpicture}[label distance=-2pt, scale=.7]
      \node[empty vertex, "$v_{3}$" below] (v7) at (0, -1) {};
      \draw (-2, 0) -- (2, 0);
      \foreach[count=\i from 2] \v/\p in {x_{1}/above, v_{1}/above, v_{0}/above right, v_{2}/above, x_{2}/above} {
        \node[empty vertex, "$\v$" \p] (v\i) at ({\i - 4}, 0) {};
        \draw (v\i) -- (v7);
      }
      \node[b-vertex, "$x_{0}$"] (v1) at (0, .75) {};
      \foreach \i in {2, 6} \node[b-vertex] at (v\i) {};
      \draw (v4) -- (v1);      
    \end{tikzpicture}
    \caption{}
    \label{fig:whipping-top-labeled}
  \end{subfigure}
  \begin{subfigure}[b]{0.22\linewidth}
    \centering
    \begin{tikzpicture}[scale=.75]
      \draw (-2.,0) -- (2.,0);
      \node[a-vertex, "$v_{0}$"] (a) at (0, 1.) {};
      \node[b-vertex, "$u$" below] (u) at (0, 0) {};
      \draw[witnessed edge] (a) -- (u);
      \foreach[count=\j] \x in {-2, 2} {
        \draw (a) -- (\x, 0) node[b-vertex, "$x_{\j}$" below] (v\j) {};
      }
      \foreach[count=\j] \x in {-1, 1} {
        \draw (a) -- (\x, 0) node[empty vertex, "$v_{\j}$" below] (v\j) {};
      }
     \end{tikzpicture}
    \caption{}
    \label{fig:p5x1-labeled}
  \end{subfigure}
  \caption{Configurations~\ref{fig:whipping-top-simplified} and~\ref{fig:p5x1-simplified}, reproduced and labeled.}
  \label{fig:long claw & whipping top}
\end{figure}

\begin{lemma}\label{lem:whipping-top}
  If~$G^{s}$ contains Configuration~\ref{fig:whipping-top-simplified}, then~$G$ is a split graph.
\end{lemma}
\begin{proof}
  We use the labels in Figure~\ref{fig:whipping-top-labeled}.
  Note that~$v_{3}$ cannot witness any collateral edge by Proposition~\ref{lem:nonadjacent->dominate}.
  Since~$(G - v_{3})^{s}$ cannot contain Configuration~\ref{fig:p5+1-simplified} ($U\setminus \{v_{3}\}$),
  {the edge~$v_{0} x_{0}$ is not collateral.}
For~$i = 1, 2$,
  \begin{equation}
    \label{eq:20}
    \text{$v_{i} x_{i}$ is not collateral, and if~$v_{i} v_{0}$ is collateral, its witness is not adjacent to~$x_{i}$.}
  \end{equation}
  We consider~$i = 1$, and it is symmetric for~$i = 2$.
Note that~$v_{1}\in N_{G}(s)$ if either assertion of~\eqref{eq:20} is false (Proposition~\ref{lem:nonadjacent->dominate}).
  For the first part, suppose that~$v_{1} x_{1}$ is collateral, and let~$w$ be its witness.
  By Proposition~\ref{lem:nonadjacent->dominate},~$N_{G^{s}}(w)\cap U \subseteq \{v_{0}, v_{1}, v_{3}, x_{1}\}$, and they cannot be equal by A\ref{assumption:replacable}.
  If~$N_{G^{s}}(w)\cap U = \{v_{0}, v_{1}, x_{1}\}$, then~$(G - x_{0})^{s}$ contains a hole~$w v_{0} v_{3} x_{1}$.
  Thus,~$N_{G^{s}}(w)\cap U \subseteq \{v_{1}, v_{3}, x_{1}\}$.
  Then~$(G - x_{1})^{s}$ contains either an annotated copy of the same configuration (with~$x_{1}$ replaced by~$w$ in~$U$, when~$v_{3}\in N_{G^{s}}(w)$), or an induced net ($U\cup \{w\}\setminus \{x_{1}, v_{2}\}$).
  Thus,~$v_{1} x_{1}$ is not collateral.
  Now suppose that~$v_{0} v_{1}$ is collateral and its witness~$w$ is adjacent to~$x_{i}$.
  By Proposition~\ref{lem:nonadjacent->dominate} and A\ref{assumption:replacable},~$N_{G^{s}}(w)\cap U = \{v_{0}, v_{1}, x_{1}\}$,
  and then~$(G - x_{0})^{s}$ contains a hole~$w v_{0} v_{3} x_{1}$.  This verifies~\eqref{eq:20}.

  Since~$(G - x_{0})^{s}$ cannot contain Configuration~\ref{fig:p5x1-unlabeled} ($U\setminus \{x_{0}\}$), 
  \begin{equation}
    \label{eq:10}
      \text{if~$v_{3}\in N_G(s)$, then~$v_{0}\in N_G(s)$.}
\end{equation}
  We further argue that
  \begin{equation}
    \label{eq:27}
   \text{if~$v_{0}\in N_G(s)$, then~$v_{1}, v_{2}\in N_G(s)$.}    
 \end{equation}
 We consider~$i = 1$, and the others are symmetric.
  Suppose that~$v_{0}$ is in~$N_G(s)$ and~$v_{1}$ is not.
  By Lemma~\ref{thm:forbidden-configurations-unrestricted},~$(G - x_{1})^{s}$ cannot contain Configuration~\ref{fig:whipping-top-1-unlabeled unrestricted} 
  ($U\setminus\{x_{1}\}$), and hence the edge~$v_{1} v_{3}$ must be collateral and the vertex~$x_{1}$ must be its only witness.
  Then~$v_{3}\in N_{G}(s)$.
  Let~$w$ be a witness of~$v_{0} v_{3}$.
  Since~$w$ is not a witness of~$v_{1} v_{3}$, it is not adjacent to~$v_{1}$.
  Note that if~$x_{2}$ is a witness of~$v_{2} v_{3}$, then~$v_{2}$ cannot be a witness of~$x_{2} v_{3}$.
  The graph~$(G - v_{2})^{s}$ contains Configuration~\ref{fig:ab-wheel-simplified} ($U\cup\{w\}\setminus\{x_{1}, v_{2}\}$), or the graph~$(G - x_{2})^{s}$ contains Configuration~\ref{fig:dag+2e-simplified} ($U \setminus \{x_{1}, x_{2}\}$).
  We end up with a contradiction to A\ref{assumption:minimality}.

  For~$i = 1, 2$, we take~$w_{i}$ to be a witness of the edge~$v_{i} v_{0}$ if it is collateral, or set~$w_{i} = v_{i}$ otherwise.
We argue that
  \begin{equation}
    \label{eq:28}
\text{$w_{i}$ is a witness of~$\{v_{0}, v_{i}, v_{3}\}$.}
  \end{equation}
  We consider~$i = 1$, and it symmetric for~$i = 2$.
  If~$v_{3}\in N_{G}(s)$, then
~$\{v_{0}, v_{1}, v_{3}\}\subseteq N_{G}(s)$ by~\eqref{eq:10} and~\eqref{eq:27}, and it is witnessed by Lemma~\ref{lem:witness-N[s]} (note that~$G$ cannot be isomorphic to~$\overline{S_{3}^{+}}$ because~$|V(G)| > 7$).  If~$w_{1}$ is not a witness of~$\{v_{0}, v_{1}, v_{3}\}$, then~$(G - w_{1})^{s}[U]$ is the same configuration by~\eqref{eq:20}.
Hence, assume that~$v_{3}\not\in N_{G}(s)$, and it suffices to show that~$w_{1}$ is adjacent to~$v_{3}$. 
We have nothing to argue when~$v_{0}, v_{1}\not\in N_{G}(s)$.  Otherwise, $v_{1}\in N_{G}(s)$ by~\eqref{eq:27}, and~$N_{G^{s}}(w_{1}) \cap U = \{v_{0}, v_{1}\}$ by~\eqref{eq:20}.
Then~$(G - x_{0})^{s}$ contains Configuration~\ref{fig:sun-unlabeled} ($\{w_{1}, v_{0}, v_{1}, v_{2}, v_{3}, x_{1}\}$), violating A\ref{assumption:minimality}.

  Finally, we argue that
  \begin{equation}
    \label{eq:13}
    V(G) = U\cup \{s, w_{1}, w_{2}\}.    
  \end{equation}
  We consider the graph~$H = G[U\cup \{s, w_{1}, w_{2}\}])^{s}$. 
If~$U\cup \{s, w_{1}, w_{2}\} \subsetneq V(G)$, then~$H$ cannot contain any forbidden configuration by A\ref{assumption:minimality}.
  Except for~$v_{3} x_{1}$ and~$v_{3} x_{2}$, all the other edges in~$G^{s}[U]$ are present in~$H$.
  Depending on how many of~$v_{3} x_{1}$ and~$v_{3} x_{2}$ are present in~$H$, it contains Configuration~\ref{fig:dag-unlabeled}, with 6 or 7 vertices, or Configuration~\ref{fig:whipping-top-unlabeled}.
  This violates A\ref{assumption:minimality}, and hence~\eqref{eq:13} holds.

  By~\eqref{eq:10},~\eqref{eq:27}, and~$U\cap N_{G}(s) \ne \emptyset$, at least one of~$v_{1}$ and~$v_{2}$ is from~$N_{G}(s)$.
  If they are both in~$N_{G}(s)$, the witnesses~$w_{1}$ and~$w_{2}$ are not adjacent by Proposition~\ref{lem:P3}.
  Thus, the graph~$G$ is a split graph, with clique~$\{v_{0}, v_{1}, v_{2}, v_{3}\}$.
\end{proof}

\begin{lemma}\label{lem:p5x1}
  If~$G^{s}$ contains Configuration~\ref{fig:p5x1-simplified}, then~$G$ is isomorphic to long claw or Figure~\ref{fig:extended-net}, or is a split graph.
\end{lemma}
\begin{proof} We use the labels in Figure~\ref{fig:p5x1-labeled}.
  We argue that
  \begin{equation}
    \label{eq:21}
\text{neither~$v_{1} x_{1}$ nor~$v_{2} x_{2}$ is collateral.}
  \end{equation}
  We consider~$i = 1$, and it is symmetric for~$i = 2$.
  Suppose that~$v_{1} x_{1}$ is collateral, hence~$v_{1}\in N_{G}(s)$.
  Let~$w$ be a witness of~$\{v_{0}, v_{1}, x_{1}\}$, which exists because of Proposition~\ref{lem:unwitnessed}.
  By Proposition~\ref{lem:nonadjacent->dominate} and A\ref{assumption:replacable},~$N_{G^{s}}(w) \cap U = \{v_{0}, v_{1}, x_{1}\}$.  Then~$(G - x_{1})^{s}$ contains an annotated copy of the same forbidden configuration (with~$x_{1}$ replaced by~$w$ in~$U$), violating~A\ref{assumption:minimality}.
  Thus, \eqref{eq:21} holds. 

  \[
    \text{If~$v_{0}$ is the only vertex in~$N_{G}(s)$, then~$G$ is isomorphic to long claw or Figure~\ref{fig:extended-net}.}
\]
Note that~$x_{1} v_{1} u v_{2} x_{2}$ is a path in~$G$, and for~$i = 1, 2$, the vertex~$x_{i}$ witnesses~$v_{0} v_{i}$.
  Since~$G$ is chordal, vertices in~$N_{G}(v_{0})$ are consecutive on it.
  We first show that at most one of~$v_{1}$ and~$v_{2}$ can be in~$N_{G}(v_{0})$. 
  If~$N_{G}(v_{0})\cap U = \{v_{1}, u, v_{2}\}$, then by Assumption (minimality),~$V(G) = U\cup \{s, w\}$, where~$w$ is a witness of~$v_{0} u$, and the graph~$(G - w)^{s}$ contains a hole~$v_{0} v_{1} u v_{2}$.
If~$u$ is the only vertex in~$N_{G}(v_{0})\cap U$, then~$V(G) = U\cup\{s\}$ and~$G$ is a long claw ($v_{1}$ witnesses~$v_{0} u$).
Hence, precisely one of~$v_{1}$ and~$v_{2}$ is adjacent to~$v_{0}$ in~$G$.
  We may assume~$v_{1}\in N_{G}(v_{0})$ and~$v_{2}\not\in N_{G}(v_{0})$; the other case is symmetric.
  By Assumption (minimality),~$V(G) = U\cup\{s\}$ ($x_{1}$ and~$v_{2}$ witness~$v_{0} v_{1}$ and~$v_{0} u$, respectively).
  Hence,~$G$ is isomorphic to Figure~\ref{fig:extended-net}.

  In the rest,~$|N_{G}(s)\cap U| > 1$.  We show that~$G$ is a split graph.
We may assume without loss of generality that~$v_{1}\in N_{G}(s)$.
  Let~$w_{1}$ be a witness of~$v_{1} u$, which is collateral by Corollary~\ref{lem:P4}.
  Let~$w_{2}$ be a witness of~$v_{2} u$ if~$v_{2}\in N_{G}(s)$ or~$v_{2}$ otherwise.
  By~\eqref{eq:21},~$x_{1}$ witnesses~$v_{0} v_{1}$ and~$x_{2}$ witnesses~$v_{0} v_{2}$.
By Assumption (minimality),~$V(G) = U\cup \{s, w_{1}, w_{2}\}$.
By Proposition~\ref{lem:nonadjacent->dominate},~$N_{G^{s}}(w_{1})\subseteq \{u, v_{0}, v_{1}, x_{1}\}$, and they cannot be equal by A\ref{assumption:replacable}.
If~$v_{0}\not\in N_{G^{s}}(w_{1})\}$, then~$(G - x_{2})^{s}$ contains Configuration~\ref{fig:holes-unlabeled} ($v_{0} x_{1} w_{1} u$) when~$w_{1}$ is adjacent to~$x_{1}$, or Configuration~\ref{fig:sun-unlabeled} $(\{v_0,v_1,u,x_1,w_1,w_2\})$ otherwise.
  If $w_{1}$ is non-djacent to $v_0$ in $G^s$, then $w_2$ is a witness of $uv_{0}$, and $(G - x_{2})^{s}$ contains Configuration~\ref{fig:sun-unlabeled} $(\{v_0,v_1,u,x_1,w_1,w_2\})$.  Hence, we have~$N_{G^{s}}(w_{1}) = \{v_{0}, v_{1}, u\}$.
  Since~$G$ is connected,~$N_{G}(x_{1}) \ne \emptyset$.
    We must have~$v_{2}\in N_{G}(s)$.
    By symmetry,~$N_{G^{s}}(w_{2}) = \{v_{0}, v_{2}, u\}$.
    By Proposition~\ref{lem:P3},~$w_{1}$ and~$w_{2}$ are not adjacent.
    Thus,~$G$ is a split graph, with the clique~$\{u, v_{0}, v_{1}, v_{2}\}$.
\end{proof}

\begin{figure}[ht]
  \centering \small
  \begin{subfigure}[b]{0.22\linewidth}
    \centering
    \begin{tikzpicture}[scale=.75]
      \draw (1, 1) -- (4, 1);

       \node[empty vertex, "$x_{3}$" below] (v4) at (3, 0.2) {};
      \foreach[count=\i] \t/\v/\x in {b-/x_1/1, a-/x_2/2, a-/{\quad v}/3.5, b-/u/5} {
        \draw (v4) -- (\i, 1) node[\t vertex, "$\v$"] (u\i) {};
      }

      \draw (u3) -- ++(0, .85) node[b-vertex, "$x_{4}$"] (x2) {};
      \draw[witnessed edge] (u3) -- (u4);
     \end{tikzpicture}
    \caption{}
    \label{fig:whipping-top-1}
  \end{subfigure}
  \begin{subfigure}[b]{0.22\linewidth}
    \centering
    \begin{tikzpicture}[scale=1.]
      \draw (1, 1) -- (4, 1);
      \node[empty vertex, "$x_{3}$" below] (v4) at (3, 0.) {};
      \foreach[count=\i] \t/\v/\x in {b-/x_1/1, a-/x_2/2, a-/{\quad v}/3.5, b-/u/5} {
        \draw (v4) -- (\i, 1) node[\t vertex, "$\v$"] (u\i) {};
      }
      \foreach[count=\i] \x in {-1, 1} {
        \draw (v4) -- ++({.35*\x}, .6) node[b-vertex] (w\i) {} ++({.45*\x}, -.1) node {$w_\i$};
        \draw (u\inteval{\i+1}) -- (w\i) -- (u\inteval{\i+2});
      }
      \draw (w1) -- (w2);      
      \draw (u3) -- ++(0, .7) node[b-vertex, "$x_{4}$"] (x2) {};
      \draw[witnessed edge] (u3) -- (u4) (u2) -- (w2);
      \draw[bend left, densely dotted, thick] (x2) edge (w2);
     \end{tikzpicture}
    \caption{}
    \label{fig:whipping-top-1-1}
  \end{subfigure}
  \begin{subfigure}[b]{0.18\linewidth}
    \centering
    \begin{tikzpicture}[scale=.8]
      \draw (-1, -1) grid (1,1);
      \draw (-1, 1) -- (1,1);
      \draw[bend right] (-1, 0) edge (1, 0);
      \draw (-1, 0) -- (0, 1) -- (1, 0) -- (0, -1)--cycle;      
      \foreach \x in {-1, 0, 1}
      \foreach \y in {-1, 0, 1}
      \node[empty vertex] at (\x, \y) {};
      \foreach \x in {-1, 1}
      \foreach \y in {-1, 1}
      \node[empty vertex] at (\x, \y) {};
      \foreach \y/\l in {-1/{w_{1}, w_{2}, x_{4}}, 1/{x_{1}, v, s}}
      \foreach[count=\x from -1] \v in \l
      \node at (\x, {\y*1.4}) {$\v$};
      \foreach \x/\v in {-1/x_{3}, 1/x_{2}}
      \node at ({\x*1.4}, 0) {$\v$};
      \node at (.25, .25) {$u$};
    \end{tikzpicture}
    \caption{}
    \label{fig:(S4+)-e-labeled}
  \end{subfigure}
  \begin{subfigure}[b]{0.14\linewidth}
    \centering
    \begin{tikzpicture}[scale=1.]
      \draw (2, 1) -- (4, 1);

       \node[b-vertex, "$x_{3}$" below] (v4) at (3, 0.2) {};
      \foreach[count=\i from 2] \t/\v/\x in {a-/x_2/2, a-/{v\quad}/3.5, b-/u/5} {
        \draw (v4) -- (\i, 1) node[\t vertex, "$\v$"] (u\i) {};
      }
      \node[b-vertex, "$w_{2}$" above right] (w2) at (3.6, 1.6) {};
      \draw (u3) -- ++(0, .9) node[b-vertex, "$x_{4}$"] (x2) {};
      \foreach \v in {x2, u3, u4} \draw (w2) -- (\v);
      \draw[witnessed edge] (u3) -- (u4);
      \draw[densely dotted, thick] (w2) -- (v4);
     \end{tikzpicture}
    \caption{}
    \label{fig:whipping-top-1-2}
  \end{subfigure}
  \caption{Illustrations for Lemma~\ref{lem:whipping-top-1}.  The edge~$w_{2} x_{4}$ in~\ref{fig:whipping-top-1-1} and the edge~$w_{2} x_{3}$ in~\ref{fig:whipping-top-1-2} may or may not be present.}
\end{figure}

\begin{lemma}\label{lem:whipping-top-1}
  If~$G^{s}$ contains Configuration~\ref{fig:whipping-top-1-simplified}, then 
  either~$G$ is isomorphic to the graph in Figure~\ref{fig:(S4+)-e}, or~$G$ is a split graph.
\end{lemma}
\begin{proof}
  We use the labels in Figure~\ref{fig:whipping-top-1}.
  We take a witness~$w_{1}$ of~$x_{2} v$ and a witness~$w_{2}$ of~$v u$.
We claim that 
  \[
    \text{$N_{G^{s}}(w_{1})\cap U = \{v, x_{2}, x_{3}\}$ and~$w_{1}$ is a witness of~$\{v, x_{2}, x_{3}\}$.}
  \]
  By Proposition~\ref{lem:nonadjacent->dominate},~$N_{G^{s}}(w_{1})\cap U \subseteq \{v, x_{2}, x_{3}\}$.
  Since~$(G - x_{4})^{s}$ does not contain Configuration~\ref{fig:sun-unlabeled} ($U \cup \{w_{1}\}\setminus\{x_{4}\}$), the edge~$w_{1} x_{3}$ must be present in~$G^{s}$.
Then~$w_{1}$ is a witness of~$\{v, x_{2}, x_{3}\}$ when~$x_{3}\not\in N_G(s)$. Suppose then~$x_{3}\in N_G(s)$.
  By Lemma~\ref{lem:witness-N[s]}, the clique~$\{v, x_{2}, x_{3}\}$ has a witness.
  It has to be~$w_{1}$ by Assumption (minimality): otherwise~$(G - w_{1})^{s}[U]$ is the same configuration.
   We also note that
  \begin{equation}
    \label{eq:31}
    \text{neither~$x_{3} u$ nor~$x_{3} x_{1}$ is collateral.}    
  \end{equation}
  If~$x_{3} u$ is collateral,  then~$x_{3}\in N_{G}(s)$, and~$(G - x_{2})^{s}$ contains Configuration~\ref{fig:ab-wheel-simplified} ($U \cup \{w_{1}\}\setminus\{x_{2}\}$).
  If~$x_{1} x_{3}$ is collateral, witnessed by~$w$, then~$(G - w)^{s}[U]$ is Configuration~\ref{fig:dag+e-simplified} (where~$x_{1} x_{3}$ is absent by minimality).
  Thus,
  \[
    V(G) = U\cup \{s, w_{1}, w_{2}\}.
  \]
  Note that~$w_{1}$ and~$w_{2}$ are adjacent if and only if~$w_{2} x_{2}$ is a collateral edge, witnessed by~$w_{1}$.
  See Figure~\ref{fig:whipping-top-1-1}.
  Moreover, if~$w_{1}$ and~$w_{2}$ are adjacent, then~$w_{2}$ is adjacent to~$x_{3}$ as otherwise~$u w_{2} x_{2} x_{3}$ is a hole of~$(G - x_{4})^{s}$, and~$w_{2}$ is adjacent to~$x_{4}$ as otherwise~$(G - x_{3})^{s}$ contains Configuration~\ref{fig:dag+e-simplified} ($U \cup \{w_{2}\}\setminus\{x_{3}\}$).
  As a result,
  \[
    x_{3}\in V(G)\setminus N_G[s].
  \]
  Otherwise,~$w_{1}$ must be adjacent to~$w_{2}$ because~$G$ is connected.  By Proposition~\ref{lem:unwitnessed} and~\eqref{eq:31}, the clique~$\{u, v, x_{3}\}$ has a witness, which has to be~$w_{2}$.  But then~$w_{2}$ cannot be adjacent to $x_{4}$ by Proposition~\ref{lem:nonadjacent->dominate}.

  By Proposition~\ref{lem:nonadjacent->dominate},~$N_{G^{s}}(w_{2})\cap U \subseteq \{v, u, x_{2}, x_{3}, x_{4}\}$.
  When they are equal (Figure~\ref{fig:whipping-top-1-1}), the graph~$G$, as labeled in Figure~\ref{fig:(S4+)-e-labeled}, is isomorphic to the graph in Figure~\ref{fig:(S4+)-e}.
  As said above, this must be the case when~$w_{1}$ and~$w_{2}$ are adjacent.
   Hence,~$w_{1} w_{2}\not \in E(G)$.
   If~$w_{2}$ is adjacent to~$x_{4}$, then the subgraph of~$(G - x_{1})^{s}$ induced by~$U\cup \{w_{2}\}\setminus \{x_{1}\}$ is Configuration~\ref{fig:ddag+e-unlabeled unrestricted} (when~$w_{2}$ is adjacent to~$x_{3}$) or Configuration~\ref{fig:p5x1-unlabeled unrestricted} (otherwise); see Figure~\ref{fig:whipping-top-1-2}.
   Thus, $N_{G^{s}}(w_{2}) \subseteq \{v, u, x_{3}\}$, and~$G$ is a split graph, with the clique~$\{u, v, x_{2}, x_{3}\}$.
   This concludes the proof.
\end{proof}
The last group are Configurations~\ref{fig:ddag+e-simplified}, \ref{fig:ddag+2e-simplified}, and~\ref{fig:ddag-simplified}.
Note that Configurations~\ref{fig:ddag+2e-simplified} and~\ref{fig:ddag-simplified} have at least seven vertices.
We deal with Configuration~\ref{fig:ddag+e-simplified} on six vertices separately, and leave the others together.

\begin{figure}[ht]
  \centering \small
  \begin{subfigure}[b]{0.18\linewidth}
    \centering
    \begin{tikzpicture}[scale=.5]
      \foreach \i in {1, 2, 3} {
        \draw ({120*\i+90}:2) -- ({120*\i-30}:2);
        \draw ({120*\i-90}:1) -- ({120*\i+30}:1);
      }
      \foreach \i/\t in {1/a-, 2/empty , 3/empty } {
        \pgfmathparse{int(Mod(\i,3))}
        \node[rotate={int(\i/3)*180}, b-vertex, "$u_{\pgfmathresult}$" below] (u\i) at ({90-120*\i}:2) {};
      }
      \foreach[count=\i] \p in {left, right} 
      \node[empty vertex, "$v_{\i}$" \p] (v\i) at ({270-120*\i}:1) {};
      \node[a-vertex, "$v_{0}$" below] (v0) at ({270}:1) {};
\draw[witnessed edge] (u3) -- (v0);
    \end{tikzpicture}
    \caption{}
    \label{fig:sun+1}
  \end{subfigure}
  \begin{subfigure}[b]{0.18\linewidth}
    \centering
    \begin{tikzpicture}[scale=.5]
      \foreach \i in {1, 2, 3} {
\draw ({120*\i-90}:1) -- ({120*\i+30}:1);
      }
      \foreach \i/\t in {1/a-, 2/empty , 3/empty } {
        \pgfmathparse{int(Mod(\i,3))}
\node[b-vertex, label={int(\i/3)*180 - 90}:$u_{\pgfmathresult}$] (u\i) at ({90-120*\i}:2) {};
      }
      \foreach[count=\i] \p in {left, right} 
      \node[a-vertex, "$v_{\i}$" \p] (v\i) at ({270-120*\i}:1) {};
      \node[a-vertex, "$v_{0}$" below] (v0) at ({270}:1) {};
      \draw[witnessed edge] (v1) -- (u3) -- (v0);
       \draw (v1) -- (u2) -- (v0) -- (u1) -- (v2);
\uncertain{v2}{u3};
      \node[b-vertex, "$w_{1}$" above left] (w) at (120:2) {};       
      \foreach \v in {v0, v1, v2, u3}
      \draw (w) -- (\v);
    \end{tikzpicture}
    \caption{}
    \label{fig:sun+1-3}
  \end{subfigure}
  \begin{subfigure}[b]{0.2\linewidth}
    \centering
    \begin{tikzpicture}[scale=.8, label distance = -2pt, every node/.style={empty vertex}]
      \node["$v_{0}$" below] (c) at (0, 0.) {};
\foreach \i in {1, 2} 
      \draw ({120*\i-30}:1) -- ({120*\i+90}:1);
      \foreach[count=\i from 0] \p/\l in {right/w_{1}, /u_{1}, /u_{2}} {
        \node["$\l$" \p] (u\i) at ({120*\i+90}:1.75) {};
        \draw (c) -- (u\i);
      }
      \node["$u_{0}$" right] (u0) at (90:1) {};
      \foreach[count=\i] \p in {left, right}
        \node["$v_{\i}$" above \p] (v\i) at ({120*\i+90}:1) {};
      \draw[densely dotted, thick] (u0) -- (v2);      
    \end{tikzpicture}
    \caption{}
    \label{fig:sun+1-4}
  \end{subfigure}
  \begin{subfigure}[b]{0.18\linewidth}
    \centering
    \begin{tikzpicture}[scale=.5]
      \foreach \i in {1, 2, 3} {
        \draw ({120*\i+90}:2) -- ({120*\i-30}:2);
}
      \foreach \i/\t in {1/a-, 2/empty , 3/empty } {
        \pgfmathparse{int(Mod(\i,3))}
        \node[b-vertex, label={int(\i/3)*180 - 90}:$u_{\pgfmathresult}$] (u\i) at ({90-120*\i}:2) {};
}
      \foreach[count=\i] \p in {left, right} 
      \node[b-vertex, "$v_{\i}$" \p] (v\i) at ({270-120*\i}:1) {};
      \node[a-vertex, "$v_{0}$" below] (v0) at ({270}:1) {};
      \draw[witnessed edge] (u3) -- (v0);
      \draw (v1) -- (v2);
\uncertain{v2}{v0} \uncertain{v0}{v1}
    \end{tikzpicture}
    \caption{}
    \label{fig:sun+1-1}
  \end{subfigure}
  \begin{subfigure}[b]{0.2\linewidth}
    \centering
    \tikzstyle{every node}=[empty vertex]
    \begin{tikzpicture}[scale=.5]
      \node["$u_{0}$"] (u0) at ({90}:2) {};      
      \foreach \i in {1, 2} {
        \node["$u_{\i}$" below] (u\i) at ({90-120*\i}:2) {};
        \draw (u\i) -- (u0);
      }
      \foreach[count=\i] \p in {left, right} 
      \node["$v_{\i}$" \p] (v\i) at ({270-120*\i}:1) {};
      \draw (v1) -- (v2);
      \draw (u0) -- ++(1.5, 0) node["$v_{0}$"] (v0) {} -- ++(1.5, 0) node["$s$"] (s) {};
      \draw[densely dotted, thick] (v0) -- (v2);
    \end{tikzpicture}
    \caption{}
    \label{fig:sun+1-2}
  \end{subfigure}
  \caption{Illustrations for the proof of Lemma~\ref{lem:sun+1}.
    The edge~$u_{0} v_{2}$ in~\ref{fig:sun+1-4} and the edge~$v_{0} v_{2}$ in~\ref{fig:sun+1-2} may or may not be present.}
  \label{fig:lem:rising-sun-1}
\end{figure}

\begin{lemma}\label{lem:sun+1}
  If~$G^{s}$ contains Configuration~\ref{fig:ddag+e-simplified} on six vertices, then~$G$ is isomorphic to~$\otimes(1,3)$ or Figure~\ref{fig:extended-net}.
\end{lemma}
\begin{proof}We use the labels in Figure~\ref{fig:sun+1}.
  We claim that
  \begin{equation}
    \label{eq:14}
\text{neither~$u_{1} v_{2}$ nor~$u_{2} v_{1}$ is a collateral edge.}
  \end{equation}
  We consider~$u_{1} v_{2}$, and the other is symmetric.
  Suppose that~$u_{1} v_{2}$ is collateral, hence~$v_{2}\in N_{G}(s)$.  There is a witness~$w$ of the clique~$\{v_{0}, u_{1}, v_{2}\}$ by Proposition~\ref{lem:unwitnessed}.
  By Proposition~\ref{lem:nonadjacent->dominate},~$N_{G^{s}}(w)\cap U\subseteq N_{G^{s}}(v_{2})\cap U$.
  They cannot be equal by A\ref{assumption:replacable}.
  Since~$(G - u_{1})^{s}$ does not contain an annotated copy of the same configuration (with~$u_{1}$ replaced by~$w$), precisely one of~$u_{0}$ and~$v_{1}$ is in~$N_{G^{s}}(w)$.
  If~$u_{0}\in N_{G^{s}}(w)$, then~$u_{0} v_{2}$ is not collateral (Proposition~\ref{lem:P3}) and hence~$u_{0} v_{0} v_{2} u_{1} w$ is a hole of~$G$.
  Hence~$v_{1}\in N_{G^{s}}(w)$.
  By Proposition~\ref{lem:P3} and Corollary~\ref{lem:P4},~$v_{1}\in N_{G}(s)$.
  But then~$(G - u_{1})^{s}$ contains Configuration~\ref{fig:ddag+2e-unlabeled unrestricted} ($U\cup\{w\}\setminus\{u_{1}\}$), violating Lemma~\ref{thm:forbidden-configurations-unrestricted}.
  This justifies~\eqref{eq:14}.
 
  For~$i = 1, 2$, let~$w_{i}$ be a witness of~$u_{0} v_{i}$ (note that~$u_{0}$ is a witness of~$u_{0} v_{i}$ if it is not collateral).
We argue that
  \begin{equation}
    \label{eq:26}
    V(G) = U\cup\{s, w_{1}, w_{2}\}.
  \end{equation}
  We consider the graph~$(G[U\cup\{s, w_{1}, w_{2}\}])^{s}$.
By~\eqref{eq:14},~$u_{1}$ and~$u_{2}$ witness~$v_{0} v_{2}$ and~$v_{0} v_{1}$, respectively.
  Thus, the only edges in~$G^{s}[U]$ that might be missing in~$(G[U\cup\{s, w_{1}, w_{2}\}])^{s}$ are~$u_{0} v_{0}$ and~$v_{1} v_{2}$.
  There is a hole if neither~$u_{0} v_{0}$ nor~$v_{1} v_{2}$ is present, Configuration~\ref{fig:sun-unlabeled} if only~$v_{1} v_{2}$ is present, or Configuration~\ref{fig:p5x1-unlabeled} if only~$u_{0} v_{0}$ is present.
  All of them violate A\ref{assumption:minimality}.

  Next, we argue that
  \[
v_{1}, v_{2}\in V(G)\setminus N_{G}[s].
  \]
  Suppose for contradiction that at least one of~$v_{1}$ and~$v_{2}$ is in~$N_{G}(s)$; assume that~$v_{1}\in N_{G}(s)$.
  If~$v_{2}\not\in N_{G}(s)$, then~$w_{2} = u_{0}$.
  Since~$G$ is connected,~$u_{2}$ has a neighbor in~$G$.  It can only be~$w_{1}$ by~\eqref{eq:14} and~\eqref{eq:26}, which means that~$u_{0} v_{1}$ is collateral.
  Then~$v_{0} w_{1}\in E(G^{s})$ by Proposition~\ref{lem:nonadjacent->dominate}.
By Proposition~\ref{lem:nonadjacent->dominate} and A\ref{assumption:replacable},~$N_{G^{s}}(w_{1})\subsetneq N_{G^{s}}(v_{1})\cap U$, and hence~$v_{2}$ is not adjacent to~$w_{1}$.  But then the edge~$v_{1} v_{2}$, which is collateral by Corollary~\ref{lem:P4}, does not have a witness.
  Hence,~$v_{2}\in N_{G}(s)$.
  Because no vertex in~$U$ can witness~$u_{0} v_{0}$, at least one of~$u_{0}v_{1}$ and~$u_{0}v_{2}$ is collateral.
  By Lemma~\ref{lem:witness-N[s]},~$\{v_{0}, v_{1}, v_{2}\}$ has a witness~$w$ (note that~$G$ is not isomorphic to~$\overline{S_{3}^{+}}$ because~$|V(G)| > |U\cup \{s\}| = 7$).
  By~\eqref{eq:26},~$w\in \{w_{1}, w_{2}\}\setminus \{u_{0}\}$.
  Assume without loss of generality that~$w = w_{1}$.
  Note that~$w$ is a witness of~$\{u_{0}, v_{0}, v_{1}, v_{2}\}$, and it also witnesses the edge~$u_{0} v_{0}$ in particular; see Figure~\ref{fig:sun+1-3}.
    By Assumption (minimality),~$V(G) = U\cup\{s, w\}$: i.e., $w = w_{2}$ if~$u_{0} v_{2}$ is collateral.
    Then~$u_{1}$,~$u_{2}$, and~$w_{1}$ have degree one, and their neighbors are~$v_{1}$,~$v_{2}$, and~$u_{0}$, respectively; the only edge that might be missing among~$u_{0}$, $v_{0}$, $v_{1}$, and~$v_{2}$ are~$u_{0} v_{2}$; see Figure~\ref{fig:sun+1-4}.
Depending on whether~$u_{0} v_{2}$ is present,~$G - s$ is isomorphic to either~$\overline{S_{3}^{+}}$ or~$\otimes(1,3)$, both violating Assumption (minimality).

    Then~$V(G) = U\cup\{s\}$ because~$w_{1} = w_{2} = u_{0}$; see Figure~\ref{fig:sun+1-1}.
    The witness of~$u_{0} v_{0}$ must be~$v_{1}$ or~$v_{2}$.
    Assume without loss of generality that~$v_{1}$ witnesses~$u_{0} v_{0}$.
    The only undecided edge is~$v_{0} v_{2}$, which may or may not be collateral.
Thus,~$G - \{s, v_{0}\}$ is isomorphic to~$G^{s} - \{s, v_{0}\}$, and~$N_{G}(v_{0})$ is either~$\{s, u_{0}\}$ or~$\{s, u_{0}, v_{2}\}$; see Figure~\ref{fig:sun+1-2}.
    Then~$G$ is isomorphic to either~$\otimes(1,3)$ or Figure~\ref{fig:extended-net}, respectively, violating Assumption (minimality).
\end{proof}

\begin{figure}[ht]
  \centering \small
  \begin{subfigure}[b]{0.24\linewidth}
    \centering
    \begin{tikzpicture}[label distance=-2pt, scale=.7]
      \node[empty vertex] (c) at (0, 0) {} ++ (.2, -.3) node {$v_{2}$};
      \foreach \i/\l/\p in {1/c/, 2/w_{3}/below, 3/w_{0}/below} {
          \node[empty vertex, "$\l$" \p] (u\i) at ({120*\i-30}:2) {};
          \draw ({120*\i-90}:1) -- (u\i) -- ({120*\i+30}:1);
          \draw ({120*\i-90}:1) -- ({120*\i+30}:1);
      }
      \foreach[count =\j from 3] \i/\l/\p in {1/v_{1}/, 2/u/below, 3/v_{3}/} {
        \node[empty vertex, "$\l$" \p] (v\i) at ({120*\i+30}:1) {};
        \draw (v\i) -- (c) -- (u\i);
      }
      \foreach \i/\l in {1/2, 3/1} {
        \draw (v\i) -- ({90 - 60*(\i-2)}:2) node[empty vertex, "$w_{\l}$"] {};
      }
    \end{tikzpicture}
    \caption{$G_{1} = \otimes(3, 1)$}
\end{subfigure}
  \begin{subfigure}[b]{0.19\linewidth}
    \centering
    \begin{tikzpicture}[scale=0.85]
      \def\n{4}
      \def\radius{1.}      
\foreach[count=\i from 0] \v/\t in {u/b-, v_{1}/a-, v_{2}/a-, v_{3}/a-} {
        \pgfmathsetmacro{\angle}{(-1 - \i) * (360 / \n)}
        \node[\t vertex] (v\i) at (\angle:\radius) {};
        \node at (\angle:{\radius*1.4}) {$\v$};
      }
      \foreach[count = \i from 0] \p in {below, above, above, below} {
        \pgfmathsetmacro{\angle}{(-1.5 - \i) * (360 / \n)}
        \node[gray, b-vertex, "$w_\i$" \p] (w\i) at (\angle:{\radius*1.414}) {};
\pgfmathparse{int(Mod(\i+1,4))}
        \draw (v\i) -- (v\pgfmathresult);
        \draw[gray] (v\i) -- (w\i) -- (v\pgfmathresult);
      }
      \draw[witnessed edge] (v3) -- (v0) -- (v1);
    \end{tikzpicture}
    \caption{$G_{1}^{c}$}
  \end{subfigure}
  \begin{subfigure}[b]{0.24\linewidth}
    \centering
    \begin{tikzpicture}[label distance=-2pt, xscale=1]
      \node[b-vertex, "$c$"] (x0) at (2.5, 1.75) {};
      \node[b-vertex, "$v_{1}$" left] (v1) at (2, 1) {};
      \node[a-vertex, "$u$" right] (v2) at (3, 1) {};
      \foreach \i in {1, 2} {
        \draw (x0) -- (v\i);
      }

      \foreach[count=\i] \t/\l/\x in {b-/w_{3}/1.25, a-/v_{3}/2, a-/v_{2}/3, b-/w_{1}/3.75} {
        \node [\t vertex, "$\l$" below] (u\i) at (\x, 0) {};
      }
      \foreach \i in {1, 2, 3} \draw (v1) -- (u\i);
      \foreach \i in {2, 3, 4} \draw (v2) -- (u\i);
      \draw (u1) -- (u2) -- (u3) -- (u4);
      \foreach \v in {v1, v2, u3} \draw[witnessed edge] (u2) -- (\v);
      \foreach \v in {v1, v2} \draw[witnessed edge] (u3) -- (\v);
      \draw[witnessed edge] (v1) -- (v2);
    \end{tikzpicture}
    \caption{$G_{1}^{w_{0}}$}
\end{subfigure}
  \begin{subfigure}[b]{0.23\linewidth}
    \centering
    \begin{tikzpicture}[label distance=-2pt, xscale=1]
      \node[b-vertex, "$c$"] (x0) at (2.5, 1.75) {};
      \foreach \i in {1, 2} {
        \node[b-vertex, label={\i*180}:$v_{\i}$] (v\i) at ({1+\i}, 1) {};
        \draw (x0) -- (v\i);
      }

      \foreach[count=\i] \t/\l/\x in {b-/w_{2}/1.25, a-/v_{3}/2, b-/u/3, b-/w_{0}/3.75} {
        \node [\t vertex, "$\l$" below] (u\i) at (\x, 0) {};
      }
      \foreach \v in {v1, v2, u3}
      \draw[witnessed edge] (u2) -- (\v);      
      \foreach \i in {1, 2, 3} \draw (v1) -- (u\i);
      \foreach \i in {2, 3, 4} \draw (v2) -- (u\i);
      \draw (u1) -- (u2) -- (u3) -- (u4)  (v1) -- (v2);
    \end{tikzpicture}
    \caption{$G_{1}^{w_{1}}$}
\end{subfigure}

  \begin{subfigure}[b]{0.24\linewidth}
    \centering
    \tikzstyle{every node}=[empty vertex]
    \begin{tikzpicture}[label distance=-3pt, scale=.9]
      \draw (-.5, 0) node[empty vertex, "$v_{1}$" left] (c1) {} -- (.5, 0) node[empty vertex, "$v_{2}$" right] (c2) {} -- (0, .25) node["$c$"] {} -- (c1);
      \draw (c1) -- (-1, 1) -- (1, 1) -- (c2) (c1) -- (-1.5, -1) -- (1.5, -1) -- (c2);
      \foreach \i/\l in {1/w_{3}, 2/u_{1}, 3/w_{0}} {
        \node["$\l$" above] (u\i) at ({\i-2}, 1) {};
      }
      \foreach \i in {2, 3} {
        \node["$u_{\the\numexpr5-\i\relax}$" below] (v\i) at ({\i-2.5}, -1) {};
      }
      \foreach \i/\l in {1/w_{2}, 4/w_{1}} {
        \node["$\l$" below] (v\i) at ({\i-2.5}, -1) {};
      }      
      \foreach \v/\l in {u/{2}, v/{2, 3}} {
        \foreach \i in \l 
          \draw (c1) -- (\v\i) -- (c2);
      }
    \end{tikzpicture}
    \caption{$G_{2} = \otimes(1,1,1,2)$}
  \end{subfigure}
  \begin{subfigure}[b]{0.19\linewidth}
    \centering
    \begin{tikzpicture}[scale=.65]
      \def\n{5}
      \def\radius{1.5}
      \foreach[count=\i from 0] \v/\t in {u_{3}/b-, v_{2}/a-, u_{1}/b-, v_{1}/a-, u_{2}/b-} {
        \pgfmathsetmacro{\angle}{90 - (2 - \i) * (360 / \n)}
        \node[\t vertex] (v\i) at (\angle:\radius) {};
        \node at (\angle:{\radius*1.4}) {$\v$};
      }
      \foreach[count = \i from 0] \l/\p in {2/below, 3/, 0/, 1/below} {
        \pgfmathsetmacro{\angle}{90 - (1.5 - \i) * (360 / \n)}
        \node[gray, b-vertex, "$w_\l$" \p] (w\i) at (\angle:{\radius*1.2}) {};
\draw[witnessed edge] (v\i) -- (v\inteval{\i+1});
        \draw[gray] (v\i) -- (w\i) -- (v\inteval{\i+1});
      }
      \draw (v0) -- (v\inteval{\n-1});      
    \end{tikzpicture}
    \caption{$G_{2}^{c}$}
  \end{subfigure}
  \begin{subfigure}[b]{0.24\linewidth}
    \centering
    \begin{tikzpicture}[label distance=-2pt, xscale=.7]
      \node[b-vertex, "$c$"] (x0) at (3, 1.75) {};
      \node[b-vertex, "$v_{1}$" left] (v1) at (2.3, 1) {};
      \node[a-vertex, "$u_{1}$" right] (v2) at (3.7, 1) {};
      \foreach \i in {1, 2} {
        \draw (x0) -- (v\i);
      }

      \foreach[count=\i] \t/\l in {b-/w_{3}, a-/v_{2}, b-/u_{3}, b-/u_{2}, b-/w_{1}} {
        \node [\t vertex, "$\l$" below] (u\i) at (\i, 0) {};
      }
      \foreach \i in {1, 2, 3, 4} \draw (v1) -- (u\i);
      \foreach \i in {2, 3, 4, 5} \draw (v2) -- (u\i);
      \draw (u1) -- (u2) -- (u3) -- (u4) -- (u5);
      \foreach \v in {v1, v2, u3} \draw[witnessed edge] (u2) -- (\v);
\draw[witnessed edge] (v1) -- (v2);
    \end{tikzpicture}
    \caption{$G_{2}^{w_{0}}$}
\end{subfigure}
  \begin{subfigure}[b]{0.23\linewidth}
    \centering
    \begin{tikzpicture}[label distance=-2pt, xscale=1]
      \node[b-vertex, "$c$"] (x0) at (2.5, 1.75) {};
      \node[a-vertex, "$u_{2}$" left] (v1) at (2, 1) {};
      \node[b-vertex, "$v_{1}$" right] (v2) at (3, 1) {};
      \foreach \i in {1, 2} {
        \draw (x0) -- (v\i);
      }

      \foreach[count=\i] \t/\l/\x in {b-/w_{0}/1.25, b-/u_{1}/2, a-/v_{2}/3, b-/u_{3}/3.75} {
        \node [\t vertex, "$\l$" below] (u\i) at (\x, 0) {};
      }
      \foreach \v in {u4, v2, u2} \draw[witnessed edge] (u3) -- (\v);
      \foreach \v in {u4, v2} \draw[witnessed edge] (v1) -- (\v);
      \foreach \i in {1, 2, 3} \draw (v1) -- (u\i);
      \foreach \i in {2, 3, 4} \draw (v2) -- (u\i);
      \draw (u1) -- (u2) -- (u3) -- (u4)  (v1) -- (v2);
    \end{tikzpicture}
    \caption{$G_{1}^{w_{1}}$}
\end{subfigure}
  \caption{Some flipping results of $\otimes$ graphs.  In all flipped graphs, the universal vertex is omitted; In c, d, g, and h, the witness(es) are omitted.
}
  \label{fig:flipping-otimes}
\end{figure}
Recall that when~$G$ is an~$\otimes$ graph,~$G^{c}$ contains a hole.
For any other simplicial vertex~$s$ different from~$c$, the graph~$G^{s}$ contains Configuration~\ref{fig:ddag+e-simplified},~\ref{fig:ddag+2e-simplified}, or~\ref{fig:ddag-simplified}.
To observe this, it is more convenient to transform~$G^{c}$ into~$G^{s}$,
instead of constructing~$G^{s}$ directly from~$G$.
Note that~$s$ is a simplicial vertex of a gadget, and a witness of some edge of the hole in~$G^{c}$, of which at least one end is in~$N_{G}(c)$.
In the configuration in~$G^{s}$, the degree-two vertex at the top is~$c$, and its two neighbors are the ends of the collateral edge witnessed by~$s$ in~$G^{s}$.
See Figure~\ref{fig:flipping-otimes} for some examples; note that the ends of the path at the bottom of the configuration may or may not be on the hole in~$G^{c}$.

\begin{figure}[ht]
  \centering \small
  \,
  \begin{subfigure}[b]{0.2\linewidth}
    \centering
    \begin{tikzpicture}[scale=.75]
      \node[b-vertex, "$u_0$"] (x0) at (2.5, 1.75) {};
      \foreach \i in {1, 2} {
        \node[a-vertex, "$v_\i$" right, rotate =180*\i] (v\i) at ({1+\i}, 1) {};
        \draw (x0) -- (v\i);
      }

      \foreach[count=\i] \t/\l in {b-/u_{1}, empty /x_{1}, empty /x_{p}, b-/u_{2}} {
        \node [\t vertex, "$\l$" below] (u\i) at (\i, 0) {};
      }
      
\foreach \i in {1, 2, 3} \draw (v1) -- (u\i);
      \foreach \i in {2, 3, 4} \draw (v2) -- (u\i);
      \draw (u1) -- (u2) (u2) edge[dashed] (u3)  (u3) -- (u4)  (v1) -- (v2);
      \draw[witnessed edge] (v1) -- (u4)  (v2) -- (u1);
    \end{tikzpicture}
    \caption{$p \ge 2$}
    \label{fig:configuration-fan-2-n}
  \end{subfigure}
  \,
  \begin{subfigure}[b]{0.2\linewidth}
    \centering
    \begin{tikzpicture}[scale=.75]
      \node[b-vertex, "$u_0$"] (x0) at (2.5, 1.75) {};
      \foreach \i in {1, 2} {
        \node[a-vertex, "$v_\i$" right, rotate =180*\i] (v\i) at ({1+\i}, 1) {};
      }
      \draw (x0) -- (v1);

      \foreach[count=\i] \t/\l in {b-/u_{1}, empty /x_{1}, empty /x_{p}, b-/u_{2}} {
        \node [\t vertex, "$\l$" below] (u\i) at (\i, 0) {};
      }
      
\foreach \i in {1, 2, 3} \draw (v1) -- (u\i);
      \foreach \i in {2, 3} \draw (v2) -- (u\i);
\uncertain{v2}{u4} \uncertain{x0}{v2}
      \draw (u1) -- (u2) (u2) edge[dashed] (u3)  (u3) -- (u4)  (v1) -- (v2);
      \draw[witnessed edge] (v1) -- (u4) {};
    \end{tikzpicture}
    \caption{$p \ge 2$}
    \label{fig:ddag+e-2a}
  \end{subfigure}
  \,
  \begin{subfigure}[b]{0.2\linewidth}
    \centering
    \begin{tikzpicture}[scale=.75]
      \node[b-vertex, "$u_0$"] (x0) at (2.5, 1.75) {};
      \foreach \i/\t in {1/a-, 2/b-} {
        \node[\t vertex, label ={180*\i}:$v_\i$] (v\i) at ({1+\i}, 1) {};
        \draw (x0) -- (v\i);
      }

      \foreach[count=\i] \t/\l in {b-/u_{1}, empty /x_{1}, empty /x_{p}, b-/u_{2}} {
        \node [\t vertex, "$\l$" below] (u\i) at (\i, 0) {};
      }
      
\foreach \i in {1, 2, 3} \draw (v1) -- (u\i);
      \foreach \i in {2, 3, 4} \draw (v2) -- (u\i);
      \draw (u1) -- (u2) (u2) edge[dashed] (u3)  (u3) -- (u4);
      \uncertain{v2}{v1}
      \draw[witnessed edge] (v1) -- (u4) {};
    \end{tikzpicture}
    \caption{$p \ge 2$}
    \label{fig:ddag+e-1a}
  \end{subfigure}

    \begin{subfigure}[b]{0.2\linewidth}
    \centering
    \begin{tikzpicture}[scale=.75]
      \node[b-vertex, "$u_0$"] (x0) at (2.5, 1.75) {};
      \foreach \i in {1, 2} {
        \node[a-vertex, label ={180*\i}:$v_\i$] (v\i) at ({1+\i}, 1) {};
\uncertain{x0}{v\i}
      }

      \foreach[count=\i] \t/\l in {b-/u_{1}, empty /x_{1}, empty /x_{p}, b-/u_{2}} {
        \node [\t vertex, "$\l$" below] (u\i) at (\i, 0) {};
      }
      
      \foreach \i in {2, 3} \draw (v1) -- (u\i);
      \foreach \i in {2, 3} \draw (v2) -- (u\i);
      \uncertain{v1}{u1}
      \uncertain{v2}{u4}
      \draw (u1) -- (u2) (u2) edge[dashed] (u3)  (u3) -- (u4)  (v1) -- (v2);
    \end{tikzpicture}
    \caption{$p \ge 2$}
    \label{fig:ddag-labeld-2a}
  \end{subfigure}
  \,
  \begin{subfigure}[b]{0.2\linewidth}
    \centering
    \begin{tikzpicture}[scale=.75]
      \node[b-vertex, "$u_0$"] (x0) at (2.5, 1.75) {};
      \foreach \i/\t in {1/a-, 2/b-} {
        \node[\t vertex, label ={180*\i}:$v_\i$] (v\i) at ({1+\i}, 1) {};
      }

      \foreach[count=\i] \t/\l in {b-/u_{1}, empty /x_{1}, empty /x_{p}, b-/u_{2}} {
        \node [\t vertex, "$\l$" below] (u\i) at (\i, 0) {};
      }
      
      \foreach \i in {2, 3} \draw (v1) -- (u\i);
      \foreach \i in {2, 3, 4} \draw (v2) -- (u\i);
      \draw (u1) -- (u2) (u2) edge[dashed] (u3)  (u3) -- (u4) (v2) -- (x0);
      \uncertain{x0}{v1} \uncertain{v1}{u1} \uncertain{v2}{v1}
    \end{tikzpicture}
    \caption{$p \ge 2$}
    \label{fig:ddag-labeld-1a}
  \end{subfigure}
  \,
  \begin{subfigure}[b]{0.2\linewidth}
    \centering
    \begin{tikzpicture}[scale=.75]
      \node[b-vertex, "$u_0$"] (x0) at (2.5, 1.75) {};
      \foreach \i in {1, 2} {
        \node[b-vertex, "$v_\i$" right, rotate =180*\i] (v\i) at ({1+\i}, 1) {};
        \draw (x0) -- (v\i);
      }

      \foreach[count=\i] \t/\l in {b-/u_{1}, empty /x_{1}, empty /x_{p}, b-/u_{2}} {
        \node [\t vertex, "$\l$" below] (u\i) at (\i, 0) {};
      }
      
      \foreach \i in {1, 2, 3} \draw (v1) -- (u\i);
      \foreach \i in {2, 3, 4} \draw (v2) -- (u\i);
      \draw (u1) -- (u2) (u2) edge[dashed] (u3)  (u3) -- (u4)  (v1) -- (v2);
    \end{tikzpicture}
    \caption{$p \ge 2$}
    \label{fig:ddag-labeld-0a}
  \end{subfigure}
  \caption{Forbidden configurations, reproduced and labeled.}
  \label{fig:group-8}
\end{figure}

\begin{lemma}\label{lem:ddag}
  If~$G^{s}$ contains Configuration~\ref{fig:ddag+e-simplified},~\ref{fig:ddag+2e-simplified}, or~\ref{fig:ddag-simplified}, then~$G$ is isomorphic to the graph in Figure~\ref{fig:the-weird},~$\overline{S_{\ell}^{+}}$ with~$\ell\ge 3$ or~$\otimes(a_{0}, a_{1}, \ldots, a_{2 p - 1})$ with~$\sum a_{i} \ge 4$.
\end{lemma}
\begin{proof}Let~$P$ denote the path~$u_{1} x_{1} \cdots x_{p} u_{2}$ of~$G^{s}$.  For convenience, we may use~$x_{0}$ and~$x_{p+1}$ to refer to~$u_{1}$ and~$u_{2}$, respectively.
    For each~$i = 0, \ldots, p$, if the edge~$x_{i} x_{i+1}$ is collateral, let~$w_{i}$ be its witness.
  Denote by~$W$ these witnesses.

First, we argue that
  \begin{equation}
    \label{eq:6}
\text{Neither edge incident to~$u_{0}$ is collateral.}
  \end{equation}
  If~$u_{0} v_{2}$ in Figure~\ref{fig:ddag+e-2a} is collateral, then~$G^{s}[\{u_{0}, u_{1}, u_{2}, v_{1}, v_{2}, x_{1}\}]$ is Configuration~\ref{fig:double-fork+1-unlabeled unrestricted}, violating Lemma~\ref{thm:forbidden-configurations-unrestricted}.  By Corollary~\ref{lem:P4},~$v_{2} x_{1}$ is collateral in Figure~\ref{fig:ddag-labeld-2a}.
 If~$u_{0} v_{1}$ in Figure~\ref{fig:ddag-labeld-2a} is collateral, then~$(G - u_{1})^{s}$ contains Configuration~\ref{fig:ddag+e-unlabeled unrestricted} ($U\setminus \{u_{1}\}$, when~$x_{1} \in V(G)\setminus N_{G}(s)$), or~$(G - x_{1})^{s}$ contains Configuration~\ref{fig:whipping-top-1-unlabeled unrestricted} ($\{u_{0}, u_{1}, u_{2}, v_{1}, v_{2}, x_{p}\}$, when~$x_{1} \in N_{G}(s)$), which violates Lemma~\ref{thm:forbidden-configurations-unrestricted}.
 By symmetry,~$u_{0} v_{2}$ in Figure~\ref{fig:ddag-labeld-2a} is not collateral.
 Now suppose that~$u_{0} v_{1}$ in Figure~\ref{fig:ddag-labeld-1a} is collateral.
 By Corollary~\ref{lem:P4}, the edge~$v_{1} x_{p}$ is collateral.
 If~$u_{1} v_{1}$ is collateral, then either~$G^{s}[\{u_{0}, u_{1}, u_{2}, v_{1}, x_{p}\}]$ is Configuration~\ref{fig:fork-simplified} when~$x_{p}\in N_{G}(s)$, or~$G^{s}[\{u_{0}, u_{1}, v_{1}, x_{p}\}]$ is Configuration~\ref{fig:claw-simplified} otherwise, violating Lemma~\ref{lem:group-0}.
 Also note that~$x_{p}\in N_{G}(s)$: otherwise,~$G^{s}[\{u_{1}, v_{1}, u_{0}, v_{2}, x_{p}\}]$ is Configuration~\ref{fig:dag+2e-simplified} on five vertices, violating Lemma~\ref{lem:net}.
 By Proposition~\ref{lem:nonadjacent->dominate}, $x_{1}$ cannot witness~$v_{1} x_{p}$ or~$v_{2} x_{p}$ when~$p = 2$.
 Then~$(G - x_{1})^{s}$ contains either Configuration~\ref{fig:holes-unlabeled} ($u_{0} v_{1} x_{p} v_{2}$, when~$x_{1}$ is the only witness of~$v_{1} v_{2}$) or Configuration~\ref{fig:whipping-top-1-unlabeled} ($\{u_{0}, u_{1}, u_{2}, v_{1}, v_{2}, x_{p}\}$). They violate A\ref{assumption:minimality}.
  \begin{equation}
    \label{eq:7}
    \text{Neither~$u_{1} v_{1}$ nor~$u_{2} v_{2}$ is collateral.}
  \end{equation}
  If~$u_{2} v_{2}$ in Figure~\ref{fig:ddag+e-2a} or~\ref{fig:ddag-labeld-2a} is collateral, then~$(G - v_{1})^{s}$ contains Configuration~\ref{fig:dag+e-simplified} ($U\setminus \{v_{1}\}$) or Configuration~\ref{fig:dag+2e-simplified} ($U\setminus \{v_{1},u_{1}\}$).
  It is symmetric when~$u_{1} v_{1}$ in Figure~\ref{fig:ddag-labeld-2a} is collateral.
  Now suppose that~$u_{1} v_{1}$ in Figure~\ref{fig:ddag-labeld-1a} is collateral.
By Corollary~\ref{lem:P4}, the edge~$v_{1} x_{p}$ is collateral.
  If~$x_{p}\not\in N_{G}(s)$, then~$(G - u_{2})^{s}$ contains Configuration~\ref{fig:dag+2e-simplified} ($U\setminus \{u_{2}, v_{2}\}$).
  Hence,~$x_{p}\in N_{G}(s)$.
  By Proposition~\ref{lem:nonadjacent->dominate},~$v_{2}$ cannot witness~$v_{1} x_{p}$ or~$x_{i}x_{i+1}$ for~$i=1,\ldots,p$.
  Then~$(G - v_{2})^{s}$ contains Configuration~\ref{fig:dag+e-simplified} ($U\setminus \{v_{2}\}$), violating A\ref{assumption:minimality}.
  
 We then show that
 \begin{equation}
   \label{eq:34}
   V(G) = U\cup W\cup\{s\}.   
 \end{equation}
 Let~$H = G[U\cup W\cup\{s\}]^{s}$.
 By Assumption (minimality), it suffices to show that all the edges in~$G^{s}[U]$ are present in~$H$.
   For~$i = 0, \ldots, p$, the edge~$x_{i} x_{i+1}$ is collateral in~$H$ if and only if it is collateral in~$G^{s}$, witnessed by~$w_{i}$.
 Thus,~$P$ remains a path in~$H$.
 By~\eqref{eq:6} and~\eqref{eq:7}, for~$i = 1, 2$, the edges~$u_{0} v_{i}$ and~$u_{i} v_{i}$ are present in~$H$.
   Also present in~$H$ is~$v_{1} v_{2}$, witnessed by~$u_{0}$.
   By A\ref{assumption:minimality},~$H$ is chordal, which means that~$N_{H}(v_{1})$ and~$N_{H}(v_{2})$ are consecutive on~$P$.
The remaining edges are between~$\{v_{1}, v_{2}\}$ and~$\{x_{1}, \ldots, x_{p}\}$.
   Suppose that there exists~$i \le p$ such that~$x_{i}\not\in N_{H}(v_{1})$.
   We may assume that~$x_{i-1} v_{1}\in E(H)$ (i.e.,~$i$ is the smallest with this property).  Note that~$x_{i-1} v_{2}\in E(H)$ since~$H$ is chordal.
\begin{itemize}
   \item If~$i = 1$, then~$H[\{v_{2}, u_{0}, v_{1}, u_{1}, x_{1}, x_{2}\}]$ is Configuration~\ref{fig:p5x1-unlabeled unrestricted}.
\item If~$i > 1$, then~$H[\{u_{0}, v_{1}, v_{2}, x_{j-1}, x_{j}, \ldots, x_{i}\}]$, where~$j$ is the smallest index such that~$j\ge 1$ and~$x_{j} v_{2}\in E(H)$, is Configuration~\ref{fig:ddag-unlabeled unrestricted} or~\ref{fig:ddag+e-unlabeled unrestricted}.
   \end{itemize}
   Note that the first case and Configuration~\ref{fig:ddag+e-unlabeled unrestricted} in the the second case can only happen when~$u_{1}v_{2}$ is a collateral edge, i.e., in Figure~\ref{fig:configuration-fan-2-n}.
   It is symmetric if there exists~$i \le p$ such that~$x_{i}\not\in N_{H}(v_{2})$.
   They violate Lemma~\ref{thm:forbidden-configurations-unrestricted}, and hence~\eqref{eq:34} holds.

   \begin{equation}
     \label{eq:19}
\text{If~$u_{1} x_{1}$ is collateral, then~$\{u_{1}, x_{1}, v_{1}, v_{2}\}$ is a witnessed clique of~$G^{s}$.}
   \end{equation}
   Note that~$x_{1}\in N_{G}(s)$, and~$w_{0}$ is a witness of~$u_{1} x_{1}$.
By Proposition~\ref{lem:nonadjacent->dominate},~$N_{G^{s}}(w_{0})\cap U\subseteq N_{G^{s}}(x_{1})\cap U$, and they cannot be equal by A\ref{assumption:replacable}.
If~$x_{2}\in N_{G^{s}}(w_{0})$, then~$v_{1}\in N_{G^{s}}(w_{0})$---otherwise~$(G - u_{0})^{s}$ contains a hole~$v_{1}u_{1}w_{0} x_{2}$---and~$v_{2}\not\in N_{G^{s}}(w_{0})$.
Then~$(G - x_{1})^{s}$ contains Configuration~\ref{fig:ddag+e-unlabeled} or~\ref{fig:ddag-unlabeled} (with~$\{u_{1}, x_{1}\}$ replaced by~$w_{0}$ in~$U$).
In the rest,~$x_{2}\not\in N_{G^{s}}(w_{0})$.
Then we must have~$v_{1}\in N_{G^{s}}(w_{0})$ (otherwise~$(G - u_{1})^{s}$ contains Configuration~\ref{fig:dag+e-simplified} or~\ref{fig:dag-unlabeled}, $U\cup \{w_{0}\}\setminus \{u_{1}, v_{2}\}$), and~$v_2 \in N_{G^{s}}(w_{0})$ (otherwise~$(G - u_{1})^{s}$ contains Configuration~\ref{fig:ddag+e-unlabeled} or~\ref{fig:ddag-unlabeled}, $U\cup \{w_{0}\}\setminus \{u_{1}\}$).
By Proposition~\ref{lem:P3},~$v_{2}\in N_{G}(s)$ (note that~$x_{1} v_{2}$ is collateral by Proposition~\ref{lem:nonadjacent->dominate} when~$v_{2}\not\in N_{G}(s)$).
By Propositions~\ref{lem:nonadjacent->dominate} and~\ref{lem:P3},~$w_{0}$ is not adjacent to any other vertex in~$W$.
The edge~$v_{1} w_{0}$ is not collateral because there is no witness ($w_{0} u_{1} x_{1} v_{1}$ is a hole of~$G$ if~$u_{1}$ is a witness of~$v_{1} w_{0}$).
If~$v_{2}\in N_{G}(w_{0})$, then~$(G - u_{1})^{s}$ contains Configuration~\ref{fig:ddag+e-unlabeled} or~\ref{fig:ddag+2e-unlabeled} (with~$u_{1}$ replaced by~$w_{0}$ in~$U$).
Thus,~$\{u_{1}, x_{1}, v_{1}, v_{2}\}$ is a clique, where~$u_{1} v_{2}$ is a collateral edge, and~$w_{0}$ is a witness.
By symmetry,
   \begin{equation}
    \label{eq:16}
       \text{If~$u_{2} x_{p}$ is collateral, then~$\{u_{2}, x_{p}, v_{1}, v_{2}\}$ is a witnessed clique of~$G^{s}$.}
   \end{equation}
  By definition, every vertex in~$W$ is a witness of a collateral edge on~$P$.  At least one end of this edge is in~$N_{G}(s)$, which means that~$W\cap N_{G}(u_{0}) = \emptyset$ by Proposition~\ref{lem:nonadjacent->dominate}.
  By \eqref{eq:6}, the vertex~$v_i, i=1, 2,$ is in~$N_{G}(u_{0})$ if and only if it is not in~$N_{G}(s)$.
  On the other hand, the vertex~$x_i, i= 0, \ldots, p+1,$ is in~$N_{G}(u_{0})$ if and only it is in~$N_{G}(s)$, and then it is adjacent to~$\{v_{1}, v_{2}\} \setminus N_{G}(s)$.
  Thus, the vertex~$u_{0}$ is simplicial in~$G$, and the graph~$G^{u_{0}}$ is well defined.
The first observation on~$G^{u_{0}}$ is that
 \begin{equation}
   \label{eq:30}
   \text{for~$i = 1, 2$, the edge~$v_{i} u_{3-i}$ is present in~$G^{u_{0}}$, while~$v_{i} u_{i}$ is not.}   
 \end{equation}
 Since~$v_{i} u_{3 - i}\in E(G)$ if and only if~$v_{i}\in N_{G}(s)\setminus N_{G}(u_{0})$, the first follows directly from the constructions of~$G^{s}$ and~$G^{u_{0}}$.
The second follows from~\eqref{eq:7} when~$v_{i}\in N_{G}(s)\setminus N_{G}(u_{0})$.
If~$v_{i}\in N_{G}(u_{0})\setminus N_{G}(s)$, then the edge~$u_{i} v_{i}$, if it is present in~$G^{u_{0}}$, must be collateral, and its witness can only be~$w_{0}$. 
However, this is impossible by~\eqref{eq:19}.
Also by construction,
 \[
   \text{$N_{G^{u_{0}}}(s) = \{v_{1}, v_{2}\}$, and $s$ is a witness of the edge~$v_{1} v_{2}$ in~$G^{u_{0}}$.}
 \]

 Note that~$P$ remains an induced path in~$G^{u_{0}}$.
 Let~$a$ be the smallest index such that~$x_{a}\in N_{G^{u_{0}}}(v_{1})$, and~$b$ be the largest index such that~$x_{b}\in N_{G^{u_{0}}}(v_{2})$.
 Note that~$a \ge 1$ and~$b \le p$ by~\eqref{eq:30}.
 The lemma follows from Lemma~\ref{lem:hole} if~$G^{u_{0}}$ is not chordal.

 In the rest,~$G^{u_{0}}$ is chordal.  Then~$N_{G^{u_{0}}}(v_{1})$ and~$N_{G^{u_{0}}}(v_{2})$ are consecutive on~$P$, and~$a \le b$ (otherwise,~$v_{1} v_{2} x_{b} x_{b+1} \cdots x_{a}$ is a hole).
 Hence,
 \[
   1 \le a \le b \le p.
 \]
 If~$b = 1$, then~$G^{u_{0}}$ contains Configuration~\ref{fig:sun-unlabeled} ($\{s, v_{2}, v_{1}, u_{1}, x_{1}, x_{2}\}$), and the lemma follows from Lemma~\ref{lem:sun}.
 If~$1 < b < p$, then~$(G - u_{2})^{u_{0}}$ contains Configuration~\ref{fig:dag-unlabeled unrestricted} ($\{s, v_{1}, x_{a-1}, x_{a}, \ldots, x_{p}\}$, when~$u_{2}$ is the only witness of the edge~$v_{1} x_{p}$ in~$G^{u_{0}}$), or Configuration~\ref{fig:ddag-unlabeled unrestricted} ($\{s, v_{1}, v_{2}, x_{a-1}, x_{a}, \ldots, x_{b}, x_{b+1}\}$, otherwise).  Both violate Lemma~\ref{thm:forbidden-configurations-unrestricted}, and it is symmetric if~$a > 1$.
 Thus,~$a = 1$ and~$b = p$.
 In other words,
 \[
   \{x_{1}, \ldots, x_{p}\} \subseteq N_{G^{u_{0}}}(v_{1})\cap N_{G^{u_{0}}}(v_{2}).
 \]
 Since~$v_{1}$ and~$x_{i}, i = 1, \ldots, p,$ are adjacent in both~$G^{s}$ and~$G^{u_{0}}$, they are adjacent in~$G$.  (Otherwise, we must have~$x_{i}\in N_{G}(s)$, and then Proposition~\ref{lem:nonadjacent->dominate} is violated in~$G^{s}$ or~$G^{u_{0}}$.)  The same holds for~$v_{2}$ and~$x_{i}$.  Thus,
 \[
   \{x_{1}, \ldots, x_{p}\} \subseteq N_{G}(v_{1})\cap N_{G}(v_{2}).
 \]
 
 Thus,~$G^{u_{0}}[U]$ is always Configuration~\ref{fig:ddag-simplified}, i.e., one of Configurations~\ref{fig:ddag-labeld-2a}, \ref{fig:ddag-labeld-1a}, and~\ref{fig:ddag-labeld-0a} (note that the positions of~$v_{1}$ and~$v_{2}$ switched).
 Applying the same argument, with the roles of~$s$ and~$u_{0}$ switched, we can conclude that~$G^{s}[U]$ is always Configuration~\ref{fig:ddag-simplified}.
 In other words, if~$G^{s}[U]$ is Configuration~\ref{fig:ddag+e-simplified} or~\ref{fig:ddag+2e-simplified} (i.e., Configuration~\ref{fig:configuration-fan-2-n}, \ref{fig:ddag+e-2a}, or~\ref{fig:ddag+e-1a}), then either it violates Assumption (minimality), or~$G^{u_{0}}$ contains a hole.
   Moreover,~$G^{s}[U]$ is Configuration~\ref{fig:ddag-labeld-0a} if and only if~$G^{u_{0}}[U]$ is Configuration~\ref{fig:ddag-labeld-2a}.
   In the rest, we may assume without loss of generality that~$G^{s}$ is Configuration~\ref{fig:ddag-labeld-2a} or~\ref{fig:ddag-labeld-1a}.

   Note that neither~$u_{1} x_{1}$ nor~$u_{2} x_{p}$ is collateral by~\eqref{eq:19}.
   Finally, we argue that
   \[
     \text{$w_{p-1}$ witnesses the clique~$\{v_{1}, v_{2}, x_{p-1}, x_{p}\}$.}
   \]
   If~$\{v_{1}, v_{2}, x_{p-2}, x_{p-1}\} \subseteq N_{G}(s)$, it has a witness by Proposition~\ref{lem:unwitnessed}, which has to be~$w_{p-2}$ by~\eqref{eq:34}.
   Hence, we suppose that~$\{v_{1}, v_{2}, x_{p-2}, x_{p-1}\}\not\subseteq N_{G}(s)$.
   By~\eqref{eq:34}, the witness of the collateral edge~$v_{1} x_{p}$ must be~$w_{p-1}$.
   If~$w_{p-1}$ is not adjacent to~$v_{2}$ in~$G^{s}$, then~$(G - u_{1})^{s}$ contains Configuration~\ref{fig:sun-unlabeled} ($\{u_{0}, u_{2}, v_{1}, v_{2}, x_{p}, w_{p-1}\}$).
   Hence,~$w_{p-1}$ and~$v_{2}$ are adjacent in~$G^{s}$, and we are done if~$v_{2}\not\in N_{G}(s)$.
   In the rest,~$v_{2}\in N_{G}(s)$, and hence precisely one of~$x_{p-1}$ and~$x_{p}$ is in~$N_{G}(s)$.
   Now that~$v_{2} w_{p-1}$ is a collateral edge, it must be witnessed by~$w_{p-2}$ by~\eqref{eq:34} (hence~$p > 2$).
   By Proposition~\ref{lem:P3},~$x_{p-2}, x_{p-1}\in N_{G}(s)$.
   But then there cannot be a witness of the collateral edge~$v_{1} x_{p-1}$ in~$G^{u_{0}}$.

   As a result,~$x_{p}\in N_G(s)$; otherwise,~$(G - u_{2})^{s}$ contains Configuration~\ref{fig:ddag+e-unlabeled} ($U\setminus \{u_{2}\}$).
   But then there cannot be a witness of the collateral edge~$v_{2} x_{p}$ in~$G^{u_{0}}$.  This concludes the proof.
\end{proof}

Summarizing this section, we conclude the proof of Theorem~\ref{thm:main}, for which we need to recall two results mentioned earlier.

\begin{proposition}[folklore]
  \label{lem:connected}
  Let~$G$ be a minimal chordal graph that is not a circular-arc graph.  If~$G$ is not connected, then~$G$ is isomorphic to net$^\star$, whipping top$^\star$, or~$\otimes(2, a), a \ge 1$.
\end{proposition}

\begin{theorem}[\cite{cao-24-split-cag}]\label{thm:split}
  A split graph is a circular-arc graph if and only if it does not contain an induced copy of any graph in Figure~\ref{fig:net-star}--\ref{fig:whipping-top-derived},~$\overline{S_{k}^{+}}, k \ge 3$, or~$\otimes(a, b)$ with~$1 \le b \le 2 \le a$.
\end{theorem}

\begin{proof}[Proof of Theorem~\ref{thm:main}]
  Necessity.
  It is straightforward to check that long claw, whipping top$^\star$, and all six graphs in Figure~\ref{fig:chordal-non-cag} are not circular-arc graphs.
  By Theorem~\ref{thm:split} and Proposition~\ref{lem:o-graphs}, neither~$\overline{S_{k}^{+}}, k \ge 3$, nor any~$\otimes$ graph can be a circular-arc graph.

  For sufficiency, let~$G$ be a minimal chordal graph that is not a circular-arc graph.
If~$G$ is not connected, then~$G$ is isomorphic to net$^\star$, whipping top$^\star$, or~$\otimes(2, a), a \ge 1$ (Proposition~\ref{lem:connected}).
  Henceforth,~$G$ is a minimal connected chordal graph that is not a circular-arc graph. 
  By Theorem~\ref{thm:simplified-forbidden-configurations},~$G^{s}$ contains a forbidden configuration in Figure~\ref{fig:simplified-forbidden-configurations}.
  The statement then follows from
  Lemma~\ref{lem:sun} (Configuration~\ref{fig:sun-simplified}),
  Lemma~\ref{lem:whipping-top} (Configuration~\ref{fig:whipping-top-simplified}),
  Lemma~\ref{lem:p5x1} (Configuration~\ref{fig:p5x1-simplified}),  
  Lemma~\ref{lem:whipping-top-1} (Configuration~\ref{fig:whipping-top-1-simplified}),
  Lemma~\ref{lem:hole} (Configuration~\ref{fig:holes}), and
  Lemma~\ref{lem:ddag} (Configurations~\ref{fig:ddag+e-simplified},~\ref{fig:ddag+2e-simplified}, and~\ref{fig:ddag-simplified}).
\end{proof}

\appendix
\section{Appendix}
\label{sec:appendix}

\subsection{Omitted illustrations}

The family of graphs~$\overline{S_{3}^{+}}$ is illustrated in Figure~\ref{fig:sun-complement+}, and the smallest~$\otimes$ graphs are shown in Figures~\ref{fig:four-part} and~\ref{fig:Spq}.
The are 20 chordal forbidden induced subgraphs of order at most ten.  Six are given in Figure~\ref{fig:chordal-non-cag}, and the others are:
$\otimes(1, a), 3\le a \le 6$; 
$\otimes(2, a), 1\le a \le 4$; 
$\otimes(3, a), 1\le a \le 2$; 
$\otimes(1, 1, 1, a), 1\le a \le 2$; and
$\overline{S_{k}^{+}}, k = 3, 4$.
Of them, only Figure~\ref{fig:long-claw-derived}, Figure~\ref{fig:(S4+)-e},~$\overline{S_{3}^{+}}$, and~$\otimes(1, 1, 1, 1)$ have been previously known in literature~\cite{bonomo-09-partial-characterization-cag, francis-14-blocking-quadruple}.

\begin{figure}[ht]
  \centering \small
  \begin{subfigure}[b]{0.16\linewidth}
    \centering
    \begin{tikzpicture}[scale=.6]
      \node[filled vertex] (c) at (0, 0) {};
      \foreach[count =\j] \i in {1, 2, 3} 
      \draw ({120*\i-30}:2) -- ({120*\i-30}:1) -- ({120*\i+90}:1);
      \foreach[count =\j] \i in {1, 2, 3} {
        \node[empty vertex] (u\i) at ({120*\i-30}:2) {};
        \node[filled vertex] (v\i) at ({120*\i-30}:1) {};
        \draw (c) -- (v\i);
      }
    \end{tikzpicture}
    \caption{}
  \end{subfigure}
  \,
  \begin{subfigure}[b]{0.16\linewidth}
    \centering
    \begin{tikzpicture}[scale=.8]
      \node[filled vertex] (c) at (0, -0.4) {};

      \def\n{4}
      \def\radius{1.5}
      
      \foreach \i in {1, ..., \n}
      \draw ({(\i - .5) * (360 / \n)}:1) -- ({(\i + .5) * (360 / \n)}:1) -- ({(\i) * (360 / \n)}:1.25) -- ({(\i - .5) * (360 / \n)}:1);
      \foreach \i in {1, ..., \n} {
        \node[empty vertex] (u\i) at ({(\i) * (360 / \n)}:1.25) {};
        \node[filled vertex] (v\i) at ({(\i - .5) * (360 / \n)}:1) {};
        \draw (c) -- (v\i);
      }
      \draw (v1) -- (v3) (v2) -- (v4);
    \end{tikzpicture}
    \caption{}
  \end{subfigure}
  \quad
  \begin{subfigure}[b]{0.16\linewidth}
    \centering
    \begin{tikzpicture}[scale=.45]
      \node[filled vertex] (c) at (0, 0) {};

      \def\n{5}
      \def\radius{1.5}      
      \foreach \i in {1,..., \n} {
        \pgfmathsetmacro{\angle}{90 - (\i) * (360 / \n)}
        \foreach \j in {-1, 0, 1}
        \draw (\angle:{\radius*1.6}) -- ({90 - (\i + \j) * (360 / \n)}:{\radius});
      }
      \foreach \i[evaluate={\p=int(\i - 1);}] in {1,..., \n} {
        \pgfmathsetmacro{\angle}{90 - (\i) * (360 / \n)}
        \node[filled vertex] (v\i) at (\angle:\radius) {};
        \node[empty vertex] (u\i) at (\angle:{\radius*1.6}) {};
        \foreach \j in {1, ..., \p} {
          \ifthenelse{\i>1}{\draw (v\i) -- (v\j);}{};
        }
        \draw (c) -- (v\i);
      }
    \end{tikzpicture}
    \caption{}
  \end{subfigure}
  \caption{The graphs~$\overline{S_{3}^{+}}$,~$k = 3, 4, 5$.}
  \label{fig:sun-complement+}
\end{figure}

\begin{figure}[ht]
  \centering \small
  \tikzstyle{every node}=[empty vertex]
  \begin{subfigure}[b]{0.15\linewidth}
    \centering
    \begin{tikzpicture}[scale=.8]
      \draw (-.5, 0) node[filled vertex] (c1) {} -- (.5, 0) node[filled vertex] (c2) {} -- (0, .25) node["$c$"] {} -- (c1);
      \draw (c1) -- (-1, 1) -- (1, 1) -- (c2) (c1) -- (-1, -1) -- (1, -1) -- (c2);
      \foreach \i in {1, 2, 3} {
        \node (u\i) at ({\i-2}, 1) {};
p      }
      \foreach \i in {1, 2, 3} {
        \node (v\i) at ({\i-2}, -1) {};
      }
      \foreach \v/\l in {u/{2}, v/{2}} {
        \foreach \i in \l 
          \draw (c1) -- (\v\i) -- (c2);
      }
    \end{tikzpicture}
    \caption{$p = q = 1$}
    \label{fig:double-sun}
  \end{subfigure}
  \,
  \begin{subfigure}[b]{0.18\linewidth}
    \centering
    \begin{tikzpicture}[scale=.8]
      \draw (-.5, 0) node[filled vertex] (c1) {} -- (.5, 0) node[filled vertex] (c2) {} -- (0, .25) node["$c$"] {} -- (c1);
      \draw (c1) -- (-1, 1) -- (1, 1) -- (c2) (c1) -- (-1.5, -1) -- (1.5, -1) -- (c2);
      \foreach \i in {1, 2, 3} {
        \node (u\i) at ({\i-2}, 1) {};
p      }
      \foreach \i in {1, 2, 3, 4} {
        \node (v\i) at ({\i-2.5}, -1) {};
      }
      \foreach \v/\l in {u/{2}, v/{2, 3}} {
        \foreach \i in \l 
          \draw (c1) -- (\v\i) -- (c2);
      }
    \end{tikzpicture}
    \caption{$p = 1, q = 2$}
  \end{subfigure}
  \,
  \begin{subfigure}[b]{0.22\linewidth}
    \centering
    \begin{tikzpicture}[scale=.8]
      \draw (-.5, 0) node[filled vertex] (c1) {} -- (.5, 0) node[filled vertex] (c2) {} -- (0, .25) node["$c$"] {} -- (c1);
      \draw (c1) -- (-1, 1) -- (1, 1) -- (c2) (c1) -- (-2, -1) -- (2, -1) -- (c2);
      \foreach \i in {1, 2, 3} {
        \node (u\i) at ({\i-2}, 1) {};
      }
      \foreach \i in {1, 2, 3, 4, 5} {
        \node (v\i) at ({\i-3}, -1) {};
      }
      \foreach \v/\l in {u/{2}, v/{2, 3, 4}} {
        \foreach \i in \l 
          \draw (c1) -- (\v\i) -- (c2);
      }
    \end{tikzpicture}
    \caption{$p = 1, q = 3$}
    \label{fig:o(1,1,1,3)}
  \end{subfigure}
  \,
  \begin{subfigure}[b]{0.18\linewidth}
    \centering
    \begin{tikzpicture}[scale=.8]
      \draw (-.5, 0) node[filled vertex] (c1) {} -- (.5, 0) node[filled vertex] (c2) {} -- (0, .15) node["$c$" below] {} -- (c1);
      \draw (c1) -- (-1.5, 1) -- (1.5, 1) -- (c2) (c1) -- (-1.5, -1) -- (1.5, -1) -- (c2);
      \foreach \i in {1, 2, 3,4} {
        \node (u\i) at ({\i-2.5}, 1) {};
      }
      \foreach \i in {1, 2, 3, 4} {
        \node (v\i) at ({\i-2.5}, -1) {};
      }
      \foreach \v/\l in {u/{2, 3}, v/{2, 3}} {
        \foreach \i in \l 
          \draw (c1) -- (\v\i) -- (c2);
      }
    \end{tikzpicture}
    \caption{$p = q = 2$}
  \end{subfigure}
  \caption{Graphs~$\otimes(1, p, 1, q)$.  Note that a gadget consists of a solid node along with the two nodes at the opposite side's corners (specifically, the solid node on the left is with the nodes at the top right and bottom right corners).
}
  \label{fig:four-part}
\end{figure}

\begin{figure}[ht]
  \centering \small
  \begin{subfigure}[b]{0.2\linewidth}
    \centering
    \begin{tikzpicture}[scale=.8]
      \filldraw[draw = gray, fill=gray!30] (0, 1.6) ellipse [x radius=1.8, y radius=0.75];
      \node[empty vertex, "$c$"](a) at (0, 3) {};
      \draw (-1.4, 1.5) --  (-.9, 0) -- (0, 0) (0, 0) edge[dashed] (.9, 0) (.9, 0) -- (1.4, 1.5);
      \draw (a) -- (0, 1.5) node[filled vertex, label={[xshift={0.3cm}, yshift=-0.3cm]$v_{1}$}] (b) {};
      \foreach \x in {-1, 0, 1} 
      \draw (b) -- (0.9*\x, 0);
      \foreach[count=\i] \x/\sub in {-1/1, 1/2} {
        \node[empty vertex](v\i 1) at ({0.9*\x}, 0) {};
        \node[empty vertex](v\i 2) at ({1.4*\x}, 1.5) {};
      }
\node[empty vertex] (v0) at (0, 0) {};
      \foreach \i in {1, 2} \draw (v0) -- (b);      
    \end{tikzpicture}
    \caption{$p = 1$}
  \end{subfigure}
  \,
  \begin{subfigure}[b]{0.22\linewidth}
    \centering
    \begin{tikzpicture}[scale=.8]
      \filldraw[draw = gray, fill=gray!30] (0, 1.6) ellipse [x radius=2.1, y radius=0.75];
      \node[empty vertex] at (90:2.1) {};

      \node[empty vertex, "$c$"](a) at (0, 3) {};
      \draw (-1.8, 1.5) --  (-.9, 0) -- (0, 0) (0, 0) edge[dashed] (.9, 0) (.9, 0) -- (1.8, 1.5);
      \foreach[count=\i] \x/\sub in {-1/1, 1/2} {
        \node[filled vertex, "$v_{\sub}$", rotate = {\x*-45}] 
        (b\i) at ({0.9*\x}, 1.5) {};
        \draw (a) -- (b\i);
        \node[empty vertex](v\i 1) at ({0.9*\x}, 0) {};
        \node[empty vertex](v\i 2) at ({1.8*\x}, 1.5) {};
        \foreach\j in {1, 2} {
          \draw (b\i) -- (v\i\j);
        }
      }
      \draw (b1) -- (b2);
      \draw (b1) -- (v21) (b2) -- (v11);
      \node[empty vertex] (v0) at (0, 0) {};
      \foreach \i in {1, 2} \draw (v0) -- (b\i);      
    \end{tikzpicture}
    \caption{$p = 2$}
  \end{subfigure}
  \,
  \begin{subfigure}[b]{0.22\linewidth}
    \centering
    \begin{tikzpicture}[scale=.8]
      \filldraw[draw = gray, fill=gray!30] (0, 1.6) ellipse [x radius=2.1, y radius=0.95];

      \node[empty vertex, "$c$"] (a3) at (0, 3) {};
      \draw (-1.8, 1.5) --  (-.9, 0) -- (0, 0) (0, 0) edge[dashed] (.9, 0) (.9, 0) -- (1.8, 1.5);
      \draw (a3) -- (0, 1) node[filled vertex] (c) {};
      \foreach[count=\i] \x/\sub in {-1/1, 1/3} {
        \node[filled vertex, "$v_{\sub}$", rotate = {\x*-45}] (b\i) at ({0.9*\x}, 1.5) {};
        \draw (b\i) --  ({0.9*\x}, 2.2) node[empty vertex] {};
        \draw (a3) -- (b\i) -- (c);
        \node[empty vertex](v\i 1) at ({0.9*\x}, 0) {};
        \node[empty vertex](v\i 2) at ({1.8*\x}, 1.5) {};
        \foreach\j in {1, 2} {
          \draw (b\i) -- (v\i\j) -- (c);
        }
      }
      \draw (b1) -- (b2);
      \draw (b1) -- (v21) (b2) -- (v11);
      \node[empty vertex] (v0) at (0, 0) {};
      \foreach \v in {b1, b2, c} \draw (v0) -- (\v);      
    \end{tikzpicture}
    \caption{$p = 3$}
  \end{subfigure}
  \,
  \begin{subfigure}[b]{0.27\linewidth}
    \centering
    \begin{tikzpicture}[yscale=.8]
      \filldraw[draw = gray, fill=gray!30] (0, 1.65) ellipse [x radius=2.1, y radius=1.05];

      \node[empty vertex, "$c$"](a3) at (0, 3) {};
      \draw (-1.8, 1.5) --  (-.9, 0) -- (0, 0) (0, 0) edge[dashed] (.9, 0) (.9, 0) -- (1.8, 1.5);
      \foreach[count=\i] \x/\sub in {-1/1, 1/4} {
        \node[filled vertex, "$v_{\sub}$", rotate = {\x*-45}] 
        (b\i) at ({0.9*\x}, 1.5) {};
        \draw (a3) -- (b\i);
        \node[empty vertex](c\i) at ({0.9*\x}, 0) {};
        \node[empty vertex](a\i) at ({1.8*\x}, 1.5) {};
        \foreach\l in {c, a} {
          \draw (b\i) -- (\l\i);
        }
      }
      \draw (b1) -- (b2);
      \draw (b1) -- (c2) (b2) -- (c1);

    \foreach[count=\i] \x in {-1, 1} {
      \node[filled vertex](x\i) at ({\x/3}, .8) {};
      \foreach \k in {1, 2, 3}  \draw (a\k) -- (x\i);
      \foreach \k in {1, 2}  \draw (b\k) -- (x\i) -- (c\k);
    }
    \draw (x1) -- (x2);
    \foreach \i/\x in {1/-1, 2/1} {
      \draw (b\i) -- ({0.9*\x}, 2.2) node[empty vertex] {} -- (x\i);
    }
    \draw (b1) -- (90:2.2) node[empty vertex] {} -- (b2);
      \node[empty vertex] (v0) at (0, 0) {};
      \foreach \v in {b1, b2, x1, x2} \draw (v0) -- (\v);      
    \end{tikzpicture}
    \caption{$p = 4$}
  \end{subfigure}
  \caption{Graphs~$\otimes(p, q)$ with~$p + q \ge 4$.
    Inside the shadowed ellipse is the gadget~$D_{p}$.
    At the bottom is a path on~$q$ vertices.
}
  \label{fig:Spq}
\end{figure}

\begin{figure}[ht]
  \centering \small
  \begin{subfigure}[b]{0.3\linewidth}
    \centering
    \begin{tikzpicture}\def\n{3}
      \node[filled vertex] (c) at (0, 0) {};
      \foreach \i in {1, ..., \n} 
        \foreach[count=\j] \x in {-1, 1} 
          \draw ({360/\n*\i-90}:1.5) node[filled vertex] {} -- ({360/\n*\i-90+\x*50}:1.5);
      \foreach \i in {1, 2, 3} {
        \foreach[count=\j] \x in {-1, 1} {
          \node[filled vertex](u\i\j) at ({360/\n*\i - 30+\x*10}:{1.5}) {};
          \draw ({360/\n*\i-90}:1) -- (u\i\j) -- ({360/\n*\i+30}:1);
}
        \draw ({360/\n*\i-90}:1) -- ({360/\n*\i+30}:1);
        \node[filled vertex] at ({360/\n*\i-90}:1.5) (u\i) {};
      }

      \foreach \i in {1, 2, 3} {
        \node[filled vertex](v\i) at ({360/\n*\i-90}:1) {};
        \draw (v\i) -- (c);
        \foreach \j in {1, 2, 3} 
          \draw (v\i) -- (u\j);
      }
    \end{tikzpicture}
    \caption{}
  \end{subfigure}
  \,
  \begin{subfigure}[b]{0.3\linewidth}
    \centering
    \begin{tikzpicture}\def\n{4}
      \node[filled vertex] (c) at ({45}:.25) {};
      \foreach \i in {1, ..., \n} 
        \foreach[count=\j] \x in {-1, 1} {
          \draw ({360/\n*\i-45}:1.5) -- ({360/\n*\i-45+\x*35}:1.5);
        }
      \foreach \i in {1, ..., \n} {
        \foreach[count=\j] \x in {-1, 1} {
          \node[empty vertex](u\i\j) at ({360/\n*\i+\x*10}:{1.5}) {};
          \draw ({360/\n*\i-90}:1) -- (u\i\j) -- ({360/\n*\i}:1) (u\i\j) -- ({360/\n*\i + 90}:1);
        }
        \node[filled vertex] at ({360/\n*\i-45}:1.5) (u\i) {};
      }

      \foreach \i in {1, ..., \n} {
        \node[filled vertex](v\i) at ({360/\n*\i-90}:1) {};
        \draw (v\i) -- (c);
        \foreach \j in {1, ..., \n} 
        \draw (v\i) -- (u\j);
      }
      \foreach \i in {2, ..., \n} {
        \foreach \j in {1,..., \the\numexpr\i-1\relax}
        \draw (v\i) -- (v\j);
      }
        
    \end{tikzpicture}
    \caption{}
  \end{subfigure}

\caption{(a)~$\otimes(1, 1, 1, 1, 1, 1)$ and (b)~$\otimes(1, 1, 1, 1, 1, 1, 1, 1)$.
}
  \label{fig:split-non-helly-2}
\end{figure}

\subsection{Proof of Proposition~\ref{lem:o-graphs}}

This section is devoted to showing that every~$\otimes$ graph is a minimal chordal forbidden induced subgraph of (Helly) circular-arc graphs.
The main difficulty is on the minimality, i.e., to show that every proper induced subgraph of an~$\otimes$ graph is a Helly circular-arc graph.

Recall that~$\otimes$ graphs are graphs of the form~$\otimes(a_{0}, a_{1}, \ldots, a_{2 p - 1})$ with~$(a_{0}, a_{1}, \ldots, a_{2 p - 1})\not\in \{(1, 1), (1, 2)\}$.
Let~$G$ be an~$\otimes$ graph.
It comprises~$p$ gadgets,$D_{a_{2 i}}, i = 0, \ldots, p-1$,~$p$ paths, of lengths $a_{2i+1}, i = 0, \ldots, p-1$, and the special vertex~$c$.
For~$i = 0, \ldots, p - 1$, let
\[
  \{v^{2i}_{1}, v^{2i}_{2}, \ldots, v^{2i}_{a_{2i}}\}\uplus \{w^{2i}_{0}, w^{2i}_{1}, \ldots, w^{2i}_{a_{2i}}\}
\]
be the split partition of the gadget corresponding to~$a_{2 i}$ (see Figure~\ref{fig:gadget}), 
and let the path corresponding to~$a_{2 i+1}$ be
\[
  v^{2i+1}_{1}  v^{2i+1}_{2} \cdots v^{2i+1}_{a_{2i+1}}.
\]
By definition,
\[
  N(c) = \bigcup^{p-1}_{i=0}  \{v^{2i}_{1}, v^{2i}_{2}, \ldots, v^{2i}_{a_{2i}}\},
\]
and for~$i = 0, 1, \ldots, p - 1$,  \[
  N(w^{2i}_{j}) =
  \begin{cases}
    \left( N(c) \setminus \{v^{2i}_{1}\}\right) \cup \{v^{2i-1\ (\mathrm{mod}\ 2p)}_{a_{2i-1\ (\mathrm{mod}\ 2p)}}\} & \text{ if } j = 0,
    \\
    N(c) \setminus \{v^{2i}_{j-1}, v^{2i}_{j}\} & \text{ if } 1 \le j < a_{2i-1},
    \\
    \left( N(c) \setminus \{v^{2i}_{a_{2i-1}}\}\right) \cup \{v^{2i+1}_{a_{2 i + 1}}\} & \text{ if } j = a_{2i-1}.
  \end{cases}
\]
See Figure~\ref{fig:o(1,1,1,3)-labeled} for an example.

\begin{figure}[ht]
  \centering \small
  \begin{subfigure}[b]{0.3\linewidth}
    \centering
    \tikzstyle{every node}=[empty vertex]
    \begin{tikzpicture}[xscale=.8]
      \draw (-.5, 0) node[filled vertex, "$v^{0}_{1}$" left] (c1) {} -- (.5, 0) node[filled vertex, "$v^{2}_{1}$" right] (c2) {} -- (0, .25) node["$c$"] {} -- (c1);
      \draw (c1) -- (-1, 1) -- (1, 1) -- (c2) (c1) -- (-2, -1) -- (2, -1) -- (c2);
      \foreach \i/\l in {1/w^{2}_{1}, 2/v^{1}_{1}, 3/w^{0}_{2}} {
        \node["$\l$" above] (u\i) at ({\i-2}, 1) {};
      }
      \foreach \i in {2, 3, 4} {
        \node["$v^{3}_{\the\numexpr5-\i\relax}$" below] (v\i) at ({\i-3}, -1) {};
      }
      \foreach \i/\l in {1/w^{2}_{2}, 5/w^{0}_{1}} {
        \node["$\l$" below] (v\i) at ({\i-3}, -1) {};
      }      
      \foreach \v/\l in {u/{2}, v/{2, 3, 4}} {
        \foreach \i in \l 
          \draw (c1) -- (\v\i) -- (c2);
      }
    \end{tikzpicture}
    \caption{}
    \label{fig:o(1,1,1,3)-labeled}
  \end{subfigure}
  \begin{subfigure}[b]{0.3\linewidth}
    \centering
    \begin{tikzpicture}[scale=.8]
      \def\n{6}
      \def\radius{1.5}
      \node[a-vertex, "$c$"] (c) at (0, 0) {};
      \coordinate (v0) at ({180 + 180 / \n}:\radius) {};
      \foreach \i in {0, ..., 3} {
        \pgfmathsetmacro{\angle}{(3 - \i) * (360 / \n)}
        \node[b-vertex] (w\i) at (\angle:{\radius*1.1}) {};
        \node at (\angle:{\radius*1.1+.5}) {$w_\i$};
        \draw ({(2.5 - \i) * (360 / \n)}:\radius) -- (w\i) -- ({(3.5 - \i) * (360 / \n)}:\radius);
}
      \foreach \i in {1,..., \n} {
        \pgfmathsetmacro{\angle}{(3.5 - \i) * (360 / \n)}
        \coordinate (v\i) at (\angle:\radius) {};
        \draw let \n1 = {int(\i - 1)} in (v\n1) -- (v\i);
        \node at (\angle:{\radius+.4}) {$v_\i$};
}
      \foreach \i in {1, 3}
      \node[a-vertex] at (v\i) {};
      \foreach \i in {2, 4, 5, 6}
      \node[b-vertex] at (v\i) {};
      \foreach \i in {-1, 1}
      \foreach \j in {-1, 1}
        \draw (c) -- ++(\i/4, \j/4);      
    \end{tikzpicture}
    \caption{}
    \label{fig:o(1,1,1,3)-c}
  \end{subfigure}
  \begin{subfigure}[b]{0.3\linewidth}
    \centering
    \begin{tikzpicture}[scale=.14]
      \draw[dashed,thin] (10,0) arc (0:360:10);
      \pgfmathsetmacro{\outer}{13}
      \pgfmathsetmacro{\inner}{12}
      
      \pgfmathsetmacro{\n}{6}
      \foreach \i in {1, ..., \n} {
        \pgfmathsetmacro{\radius}{\inner + Mod(\i, 2)}
        \pgfmathsetmacro{\span}{360 * 1.5 / \n}
        \pgfmathsetmacro{\start}{90 - 360/\n*(\i-2.75)}
        \draw[{|[left]}-{|[right]}, thick]  (\start:\radius) arc (\start:{\start-\span}:\radius); 
        \draw[dashed] (\start:\radius) -- (\start:\inner-2);
        \node at (\start:{\inner-3}) {$v_{\i}$};
      }
      \foreach \i in {0, ..., 3} {
        \pgfmathsetmacro{\radius}{\inner + 2}
        \pgfmathsetmacro{\span}{6}
        \pgfmathsetmacro{\start}{90 - 360/\n*(\i - 1.5) + \span / 2}
        \draw[{|[left]}-{|[right]}]  (\start:\radius) node[left] {$w_{\i}$} arc (\start:{\start-\span}:\radius); 
}
    \end{tikzpicture}
    \caption{} \label{fig:o(1,1,1,3)-model}
  \end{subfigure}

  \begin{subfigure}[b]{0.3\linewidth}
    \centering
    \begin{tikzpicture}[scale=.13]
      \draw[dashed,thin] (10,0) arc (0:360:10);
      \pgfmathsetmacro{\outer}{13}
      \pgfmathsetmacro{\inner}{12}
      
      \pgfmathsetmacro{\n}{6}
      \foreach \i in {2, 4, 5, 6} {
        \pgfmathsetmacro{\radius}{\inner + Mod(\i, 2)}
        \pgfmathsetmacro{\span}{360 * 1.5 / \n}
        \pgfmathsetmacro{\start}{90 - 360/\n*(\i-2.75)}
        \draw[{|[left]}-{|[right]}, thick]  (\start:\radius) arc (\start:{\start-\span}:\radius); 
        \draw[dashed] (\start:\radius) -- (\start:\inner-2);
        \node at (\start:{\inner-3}) {$v_{\i}$};
      }
      \foreach \i in {1, 3} {
        \pgfmathsetmacro{\radius}{\inner + 3 + Mod(\i, 3)}
        \pgfmathsetmacro{\span}{90}
        \pgfmathsetmacro{\start}{90 - 360/\n*(\i-2.75)}
        \draw[Sepia, dotted] (\start:\radius) arc (\start:{\start-\span}:\radius); 
        \pgfmathsetmacro{\span}{270}
        \pgfmathsetmacro{\start}{ - 360/\n*(\i-2.75)}
        \draw[Sepia, {|[left]}-{|[right]}, thick] (\start:\radius) arc (\start:{\start-\span}:\radius); 
        \draw[dashed] (\start:\radius) -- (\start:\inner-2);
        \node at (\start:{\inner-3}) {$v_{\i}$};
      }
      \foreach \i in {1, ..., 4} {
        \pgfmathsetmacro{\radius}{\inner + 2}
        \pgfmathsetmacro{\span}{6}
        \pgfmathsetmacro{\start}{90 - 360/\n*(\i - 2.5) + \span / 2}
        \draw[{|[left]}-{|[right]}]  (\start:\radius) arc (\start:{\start-\span}:\radius); 
}
    \end{tikzpicture}
    \caption{}
    \label{fig:o(1,1,1,3)-c-model}\end{subfigure}
  \begin{subfigure}[b]{0.3\linewidth}
    \centering
    \begin{tikzpicture}[scale=.13]
      \draw[dashed,thin] (10,0) arc (0:360:10);
      \pgfmathsetmacro{\outer}{13}
      \pgfmathsetmacro{\inner}{12}
      
      \pgfmathsetmacro{\n}{6}
      \foreach \i in {2, 4, 5, 6} {
        \pgfmathsetmacro{\radius}{\inner + Mod(\i, 2)}
        \pgfmathsetmacro{\span}{360 * 1.5 / \n}
        \pgfmathsetmacro{\start}{90 - 360/\n*(\i-2.75)}
        \draw[{|[left]}-{|[right]}, thick]  (\start:\radius) arc (\start:{\start-\span}:\radius); 
        \draw[dashed] (\start:\radius) -- (\start:\inner-2);
        \node at (\start:{\inner-3}) {$v_{\i}$};
      }
      \foreach \i in {1} {
        \pgfmathsetmacro{\start}{360/8 - 360/\n*(\i-2.75)}

        \pgfmathsetmacro{\radius}{\inner + 2}
        \pgfmathsetmacro{\span}{6}
        \pgfmathsetmacro{\start}{\start + \span / 2}
        \draw[thick, Cyan, {|[left]}-{|[right]}]  (\start:\radius) node[below] {$c$} arc (\start:{\start-\span}:\radius); 
      }
      \foreach \i in {3} {
        \pgfmathsetmacro{\radius}{\inner + 3 + Mod(\i, 3)}
        \pgfmathsetmacro{\span}{90}
        \pgfmathsetmacro{\start}{90 - 360/\n*(\i-2.75)}
        \draw[Sepia, dotted] (\start:\radius) arc (\start:{\start-\span}:\radius); 
        \pgfmathsetmacro{\span}{270}
        \pgfmathsetmacro{\start}{ - 360/\n*(\i-2.75)}
        \draw[Sepia, {|[left]}-{|[right]}, thick] (\start:\radius) arc (\start:{\start-\span}:\radius); 
        \draw[dashed] (\start:\radius) -- (\start:\inner-2);
        \node at (\start:{\inner-3}) {$v_{\i}$};
      }
      \foreach \i in {1, ..., 4} {
        \pgfmathsetmacro{\radius}{\inner + 2}
        \pgfmathsetmacro{\span}{6}
        \pgfmathsetmacro{\start}{90 - 360/\n*(\i - 2.5) + \span / 2}
        \draw[{|[left]}-{|[right]}]  (\start:\radius) arc (\start:{\start-\span}:\radius); 
}
    \end{tikzpicture}
    \caption{} \label{fig:o(1,1,1,3)-v1-model}
  \end{subfigure}
  \begin{subfigure}[b]{0.3\linewidth}
    \centering
    \begin{tikzpicture}[scale=.13]
      \draw[dashed,thin] (10,0) arc (0:360:10);
      \pgfmathsetmacro{\outer}{13}
      \pgfmathsetmacro{\inner}{12}
      
      \pgfmathsetmacro{\n}{6}
      \foreach \i in {2, 4, 5, 6} {
        \pgfmathsetmacro{\radius}{\inner + Mod(\i, 2)}
        \pgfmathsetmacro{\span}{360 * 1.5 / \n}
        \pgfmathsetmacro{\start}{90 - 360/\n*(\i-2.75)}
        \draw[{|[left]}-{|[right]}, thick]  (\start:\radius) arc (\start:{\start-\span}:\radius); 
        \draw[dashed] (\start:\radius) -- (\start:\inner-2);
        \node at (\start:{\inner-3}) {$v_{\i}$};
      }
      \foreach \i in {1} {
        \pgfmathsetmacro{\radius}{\inner + 3 + Mod(\i, 3)}
        \pgfmathsetmacro{\span}{90-360/8}
        \pgfmathsetmacro{\start}{90 - 360/\n*(\i-2.75) + 360/8}
        \draw[Sepia, dotted] (\start:\radius) arc (\start:{\start-\span}:\radius); 
        \pgfmathsetmacro{\span}{270+360/8}
        \pgfmathsetmacro{\start}{360/8 - 360/\n*(\i-2.75)}
        \draw[Sepia, {|[left]}-{|[right]}, thick] (\start:\radius) arc (\start:{\start-\span}:\radius); 
        \draw[dashed] (\start:\radius) -- (\start:\inner-2);
        \node at (\start:{\inner-3}) {$v_{\i}$};

        \pgfmathsetmacro{\radius}{\inner + 2}
        \pgfmathsetmacro{\span}{6}
        \pgfmathsetmacro{\start}{\start + \span / 2}
        \draw[thick, Cyan, {|[left]}-{|[right]}]  (\start:\radius) node[below] {$c$} arc (\start:{\start-\span}:\radius); 
      }
      \foreach \i in {3} {
        \pgfmathsetmacro{\radius}{\inner + 3 + Mod(\i, 3)}
        \pgfmathsetmacro{\span}{90}
        \pgfmathsetmacro{\start}{90 - 360/\n*(\i-2.75)}
        \draw[Sepia, dotted] (\start:\radius) arc (\start:{\start-\span}:\radius); 
        \pgfmathsetmacro{\span}{270}
        \pgfmathsetmacro{\start}{ - 360/\n*(\i-2.75)}
        \draw[Sepia, {|[left]}-{|[right]}, thick] (\start:\radius) arc (\start:{\start-\span}:\radius); 
        \draw[dashed] (\start:\radius) -- (\start:\inner-2);
        \node at (\start:{\inner-3}) {$v_{\i}$};
      }
      \foreach \i in {1, 3, 4} {
        \pgfmathsetmacro{\radius}{\inner + 2}
        \pgfmathsetmacro{\span}{6}
        \pgfmathsetmacro{\start}{90 - 360/\n*(\i - 2.5) + \span / 2}
        \draw[{|[left]}-{|[right]}]  (\start:\radius) arc (\start:{\start-\span}:\radius); 
}

    \end{tikzpicture}
    \caption{} \label{fig:o(1,1,1,3)-w1-model}
  \end{subfigure}
  \caption{Illustration for the proof of Proposition~\ref{lem:o-graphs}.
    (a)~$G = \otimes(1, 1, 1, 3)$; (b)~$G^{c}$, where~$c$ is universal; (c)~a circular-arc model for~$G^{c} - c$; (d)~a circular-arc model for~$G - c$; (e)~a circular-arc model for~$G - v_{1}$; (f)~a circular-arc model for~$G - w_{1}$.
    Note that~$v_{1} = v^{0}_{1}, v_{2} = v^{1}_{1}, v_{3} = v^{2}_{1}, v_{4} = v^{3}_{1}, v_{5} = v^{3}_{2}, v_{6} = v^{3}_{3}, w_{0} = w^{0}_{1}, w_{1} = w^{0}_{2}, w_{2} = w^{2}_{1}, w_{3} = w^{2}_{2}$.
}
  \label{fig:otimes-models}
\end{figure}
It is easy to verify that the following is a hole of~$G^{c}$:
\[
  v^{0}_{1} v^{0}_{2} \cdots v^{0}_{a_{0}}
  v^{1}_{1} v^{1}_{2} \cdots v^{1}_{a_{1}}
  \cdots
  v^{2p-1}_{1}, v^{2p-1}_{2}, \ldots, v^{2p-1}_{a_{2p-1}}.
\]
For~$i = 0, 1, \ldots, 2p-1$ and~$j = 1, 2, \ldots, a_{i}$, we set
\[
  v_{\sum_{k = 0}^{i} a_{k} + j} = v^{i}_{j},
\]
and for~$i = 0, 2, \ldots, 2p-2$ and~$j = 1, 2, \ldots, a_{i}+1$, we set
\[
  w_{\sum_{k = 0}^{i} a_{k} + j} = w^{i}_{j}.
\]
See Figure~\ref{fig:o(1,1,1,3)-c} for an example.
For the convenience of indices arithmetics, we let~$v_{0} = v_{\ell}$,
where
\[
  \ell = \sum_{i=0}^{2 p - 1} {a_{i}}.
\]

\begin{remark}
  In the graph~$G^{c}$, there is an induced cycle~$v_{1} v_{2} \cdots v_{\ell}$, and for~$i = 0, \ldots, \ell - 1$, if~$w_{i}$ exists, then its neighborhood is~$\{c, v_{i}, v_{i+1}\}$.
\end{remark}

Thus,~$G^{c}$ is a circular-arc graph.  We build a circular-arc model~$\mathcal{A}'$ for~$G^{c} - c$ on a circle of length~$4 \ell$ (as illustrated in Figure~\ref{fig:o(1,1,1,3)-model}).
For~$i = 0, \ldots, \ell - 1$, we assign the arc
\begin{align*}
  A'(v_{i})& = [4 i - 3, 4 i + 3],
  \\
  A'(w_{i})& = [4 i + 2, 4 i + 2].
\end{align*}

We are ready for the Proof of Proposition~\ref{lem:o-graphs}.
For each vertex~$x\in V(G)$, we transform~$\mathcal{A}'$ into a circular-arc model of~$G - x$ in a similar way as \cite[Theorem 1.3]{cao-24-split-cag}.  

\begin{proof}[Proof of Proposition~\ref{lem:o-graphs}]
  Let~$G = \otimes(a_{0}, a_{1}, \ldots, a_{2 p - 1})$, and let~$\ell = \sum_{i=0}^{2 p - 1} {a_{i}}$.
  To show that~$G$ is chordal, we give a elimination sequence of simplicial vertices.
  The set~$N(c)$ comprises all clique vertices of the~$p$ gadgets and is a clique.
  All vertices in~$N(c)$ are adjacent to all the vertices on the paths.
  Thus,~$c$ and all the empty vertices in the gadgets are simplicial in~$G$.
  After their removal, the ends of each path are simplicial, and we can proceed in order.
  We are left with the clique~$N(c)$ after all the path vertices removed.
  
  The only~$\otimes$ graph with~$\ell < 4$ is~$\otimes(2, 1)$, i.e., sun$^\star$.
  It is easy to verify that sun$^\star$ is a minimal chordal forbidden induced subgraph of circular-arc graphs, and every proper induced subgraph of sun$^\star$ is a Helly circular-arc graph.
  In the rest, we assume that~$\ell \ge 4$.
  Since~$G^{c}$ contains a hole, the graph~$G$ itself is not a circular-arc graph by Theorem~\ref{thm:forbidden-configurations}.

  For each vertex~$x\in V(G)$, we show that~$G - x$ is a Helly circular-arc graph by constructing a Helly circular-arc model for~$G - x$ on a circle of length~$4 \ell$.
  In all these models, if the vertex~$w_{i}$ is present, we use the same arc as~$A'(w_{i})$, i.e., 
  \[
    A(w_{i}) = A'(w_{i}) = [4 i + 2, 4 i + 2].
  \]

  The circular-arc model for~$G - c$ is produced by flipping all arcs in~$\mathcal{A}'$ for vertices in~$N_{G}(c)$:
  \begin{equation}
    \label{eq:1}
    A(v_{i}) = \begin{cases}
      [4 i + 3, 4 i - 3] & v_{i}\in N_{G}(c) \\
      [4 i - 3, 4 i + 3] & v_{i}\not\in N_{G}(c).
    \end{cases}
  \end{equation}
See Figure~\ref{fig:o(1,1,1,3)-c-model} for an illustration.
  Note that the length of each arc for a vertex in~$N_{G}(c)$  is~$4 \ell - 6$, and the length of other arcs is either~$6$ or~$0$.
  Arcs for~$N_{G}(c)$ intersect pairwise because their lengths are~$4 \ell - 6$.
  For~$i = 1, \ldots, p$, arcs for the~$i$th path remains unchanged, and they intersect all the arcs for~$N_{G}(c)$ (note that each has length six).
  Each vertex in~$N_{G}(c)$ has precisely two non-neighbors, both in the same gadget.
Their arcs are~$[4 i - 2, 4 i - 2]$ and~$[4 i + 2, 4 i + 2]$.
  Hence,~$\mathcal{A}$ is a circular-arc model of~$G - c$.
To see that the model constructed above is Helly, note that the maximal cliques of~$G - c$ are~$N_{G}[w_{i}]$, and~$\{v^{2i+1}_{j},  v^{2i+1}_{j+1}\}\cup N_{G}(c)$, where~$v^{2i+1}_{j}  v^{2i+1}_{j+1}$ is an edge of the~$i$th path.  For the clique~$N_{G}[w_{i}]$, it follows from that~$A(w_{i})$ has zero length.
  For the second type, suppose that~$v_{k} = v^{2i+1}_{j}$.  Note that for every~$v\in N_{G}(c)$,
  \[
    A(v_{k}) \cap A(v_{k+1}) = [4 k + 1, 4 k + 3]\subseteq A(v).
  \]

  In the rest, the argument for the correctness of the constructed model is similar to the above and hence omitted.
  
  For the subgraph~$G - v_{i}, i = 1, \ldots, \ell$, we use \eqref{eq:1} for vertices in~$\{v_{1}, \ldots, v_{i - 1}\}\cup \{v_{i + 1}, \ldots, v_{\ell}\}$, and set
  \[
    A(c) = [4 i, 4 i].
  \]
  See Figure~\ref{fig:o(1,1,1,3)-v1-model} for an example, where~$i = 1$.

  Now consider the subgraph~$G - w_{i}, i = 1, \ldots, \ell$, and~$w_{i}$ exists.  We use \eqref{eq:1} for vertices in~$\{v_{1}, \ldots, v_{i - 1}\}\cup \{v_{i + 2}, \ldots, v_{\ell}\}$.
  If both~$v_{i}$ and~$v_{i+1}$ are in~$N_{G}(c)$, then we set
  \begin{align*}
    A(v_{i}) &= [4 i + 1, 4 i - 3]\\
    A(v_{i+1}) &= [4 i + 7, 4 i + 3],\\
    A(c) &= [4 i + 2, 4 i + 2].
  \end{align*}
  If~$v_{i}$ is in~$N_{G}(c)$ and~$v_{i+1}$ is not, then we set
  \begin{align*}
    A(v_{i}) &= [4 i, 4 i - 3]\\
    A(v_{i+1}) &= [4 i + 1, 4 i + 7],\\
    A(c) &= [4 i, 4 i].
  \end{align*}
  See Figure~\ref{fig:o(1,1,1,3)-w1-model} for an example.
  It is symmetric if~$v_{i+1}$ is in~$N_{G}(c)$ and~$v_{i}$ is not.
\end{proof}

\bibliographystyle{plainurl}
\bibliography{references}

\end{document}